\pdfoutput=1
\documentclass[11pt, a4paper, twoside,reqno]{amsart}
\usepackage[centering, totalwidth = 390pt, totalheight = 610pt]{geometry}
\usepackage{amssymb, amsmath, amsthm}
\usepackage{microtype, stmaryrd, url, lmodern, eucal, extdash}
\usepackage[shortlabels]{enumitem}
\usepackage[latin1]{inputenc}
\usepackage{color}
\definecolor{darkgreen}{rgb}{0,0.45,0}
\usepackage[colorlinks,citecolor=darkgreen,linkcolor=darkgreen]{hyperref}
\usepackage[capitalize]{cleveref}
\usepackage[british]{babel}

\makeatletter
\def\@cite#1#2{[{#1\if@tempswa ,~#2\fi}]}
\makeatother
\allowdisplaybreaks

\DeclareMathAlphabet{\mathbf}{OT1}{cmr}{b}{n}
\DeclareFontSeriesDefault[rm]{bf}{b}

\usepackage[arrow, matrix, tips, curve, graph, rotate]{xy}
\SelectTips{cm}{10}

\makeatletter
\def\matrixobject@{%
  \edef \next@{={\DirectionfromtheDirection@ }}%
  \expandafter \toks@ \next@ \plainxy@
  \let\xy@@ix@=\xyq@@toksix@
  \xyFN@ \OBJECT@}
\let\xy@entry@@norm=\entry@@norm
\def\entry@@norm@patched{%
  \let\object@=\matrixobject@
  \xy@entry@@norm }
\AtBeginDocument{\let\entry@@norm\entry@@norm@patched}
\makeatother

\newcommand{\twocong}[2][0.5]{\ar@{}[#2] \save ?(#1)*{\cong}\restore}
\newcommand{\twoeq}[2][0.5]{\ar@{}[#2] \save ?(#1)*{=}\restore}
\newcommand{\rtwocell}[3][0.5]{\ar@{}[#2] \ar@{=>}?(#1)+/l 0.2cm/;?(#1)+/r 0.2cm/^{#3}}
\newcommand{\ltwocell}[3][0.5]{\ar@{}[#2] \ar@{=>}?(#1)+/r 0.2cm/;?(#1)+/l 0.2cm/^{#3}}
\newcommand{\ltwocello}[3][0.5]{\ar@{}[#2] \ar@{=>}?(#1)+/r 0.2cm/;?(#1)+/l 0.2cm/_{#3}}
\newcommand{\dtwocell}[3][0.5]{\ar@{}[#2] \ar@{=>}?(#1)+/u  0.2cm/;?(#1)+/d 0.2cm/^{#3}}
\newcommand{\dltwocell}[3][0.5]{\ar@{}[#2] \ar@{=>}?(#1)+/ur  0.2cm/;?(#1)+/dl 0.2cm/^{#3}}
\newcommand{\drtwocell}[3][0.5]{\ar@{}[#2] \ar@{=>}?(#1)+/ul  0.2cm/;?(#1)+/dr 0.2cm/^{#3}}
\newcommand{\dthreecell}[3][0.5]{\ar@{}[#2] \ar@3{->}?(#1)+/u  0.2cm/;?(#1)+/d 0.2cm/^{#3}}
\newcommand{\utwocell}[3][0.5]{\ar@{}[#2] \ar@{=>}?(#1)+/d 0.2cm/;?(#1)+/u 0.2cm/_{#3}}
\newcommand{\dtwocelltarg}[3][0.5]{\ar@{}#2 \ar@{=>}?(#1)+/u  0.2cm/;?(#1)+/d 0.2cm/^{#3}}
\newcommand{\utwocelltarg}[3][0.5]{\ar@{}#2 \ar@{=>}?(#1)+/d  0.2cm/;?(#1)+/u 0.2cm/_{#3}}

\newdir{(}{{}*!<0em,-.14em>-\cir<.14em>{l^r}}
\newdir{ (}{{}*!/-5pt/\dir{(}}
\newdir{ >}{{}*!/-5pt/\dir{>}}

\DeclareMathOperator{\im}{im}

\newcommand{\cat}[1]{\mathrm{\mathcal #1}}
\newcommand{\thg}{{\mathord{\text{--}}}}
\newcommand{\quot}{\delimiter"502F30E\mathopen{}}

\newcommand{\dbr}[1]{\mathord{\left\llbracket{#1}\right\rrbracket}}
\newcommand{\res}[2]{\left.{#1}\right|_{#2}}

\newcommand{\spn}[1]{{\langle{#1}\rangle}}

\newcommand{\defeq}{\mathrel{\mathop:}=}
\newcommand{\cd}[2][]{\vcenter{\hbox{\xymatrix#1{#2}}}}

\renewcommand{\phi}{\varphi}
\newcommand{\A}{{\mathcal A}}
\newcommand{\B}{{\mathcal B}}
\newcommand{\C}{{\mathcal C}}
\newcommand{\D}{{\mathcal D}}
\newcommand{\E}{{\mathcal E}}

\newcommand{\J}{{\mathcal J}}
\newcommand{\K}{{\mathcal K}}

\newcommand{\N}{{\mathcal N}}
\renewcommand{\O}{{\mathcal O}}
\renewcommand{\P}{{\mathcal P}}

\let\sec=\S
\renewcommand{\S}{{\mathcal S}}
\newcommand{\T}{{\mathcal T}}
\newcommand{\U}{{\mathcal U}}
\newcommand{\V}{{\mathcal V}}

\newcommand{\xtor}[1]{\cdl[@1]{{} \ar[r]|-{\object@{|}}^{#1} & {}}}

\makeatletter

\def\hookleftarrowfill@{\arrowfill@\leftarrow\relbar{\relbar\joinrel\rhook}}
\def\twoheadleftarrowfill@{\arrowfill@\twoheadleftarrow\relbar\relbar}
\def\leftbararrowfill@{\arrowdoublefill@{\leftarrow\mkern-5mu}\relbar\mapstochar\relbar\relbar}
\def\Leftbararrowfill@{\arrowdoublefill@{\Leftarrow\mkern-2mu}\Relbar\Mapstochar\Relbar\Relbar}
\def\leftringarrowfill@{\arrowdoublefill@{\leftarrow\mkern-3mu}\relbar{\mkern-3mu\circ\mkern-2mu}\relbar\relbar}
\def\lefttriarrowfill@{\arrowfill@{\mathrel\triangleleft\mkern0.5mu\joinrel\relbar}\relbar\relbar}
\def\Lefttriarrowfill@{\arrowfill@{\mathrel\triangleleft\mkern1mu\joinrel\Relbar}\Relbar\Relbar}

\def\hookrightarrowfill@{\arrowfill@{\lhook\joinrel\relbar}\relbar\rightarrow}
\def\twoheadrightarrowfill@{\arrowfill@\relbar\relbar\twoheadrightarrow}
\def\rightbararrowfill@{\arrowdoublefill@{\relbar\mkern-0.5mu}\relbar\mapstochar\relbar\rightarrow}
\def\Rightbararrowfill@{\arrowdoublefill@{\Relbar\mkern-2mu}\Relbar\Mapstochar\Relbar\Rightarrow}
\def\rightringarrowfill@{\arrowdoublefill@\relbar\relbar{\mkern-2mu\circ\mkern-3mu}\relbar{\mkern-3mu\rightarrow}}
\def\righttriarrowfill@{\arrowfill@\relbar\relbar{\relbar\joinrel\mkern0.5mu\mathrel\triangleright}}
\def\Righttriarrowfill@{\arrowfill@\Relbar\Relbar{\Relbar\joinrel\mkern1mu\mathrel\triangleright}}

\def\leftrightarrowfill@{\arrowfill@\leftarrow\relbar\rightarrow}
\def\mapstofill@{\arrowfill@{\mapstochar\relbar}\relbar\rightarrow}

\renewcommand*\xleftarrow[2][]{\ext@arrow 20{20}0\leftarrowfill@{#1}{#2}}
\providecommand*\xLeftarrow[2][]{\ext@arrow 60{22}0{\Leftarrowfill@}{#1}{#2}}
\providecommand*\xhookleftarrow[2][]{\ext@arrow 10{20}0\hookleftarrowfill@{#1}{#2}}
\providecommand*\xtwoheadleftarrow[2][]{\ext@arrow 60{20}0\twoheadleftarrowfill@{#1}{#2}}
\providecommand*\xleftbararrow[2][]{\ext@arrow 10{22}0\leftbararrowfill@{#1}{#2}}
\providecommand*\xLeftbararrow[2][]{\ext@arrow 50{24}0\Leftbararrowfill@{#1}{#2}}
\providecommand*\xleftringarrow[2][]{\ext@arrow 10{26}0\leftringarrowfill@{#1}{#2}}
\providecommand*\xlefttriarrow[2][]{\ext@arrow 80{24}0\lefttriarrowfill@{#1}{#2}}
\providecommand*\xLefttriarrow[2][]{\ext@arrow 80{24}0\Lefttriarrowfill@{#1}{#2}}

\renewcommand*\xrightarrow[2][]{\ext@arrow 01{20}0\rightarrowfill@{#1}{#2}}
\providecommand*\xRightarrow[2][]{\ext@arrow 04{22}0{\Rightarrowfill@}{#1}{#2}}
\providecommand*\xhookrightarrow[2][]{\ext@arrow 00{20}0\hookrightarrowfill@{#1}{#2}}
\providecommand*\xtwoheadrightarrow[2][]{\ext@arrow 03{20}0\twoheadrightarrowfill@{#1}{#2}}
\providecommand*\xrightbararrow[2][]{\ext@arrow 01{22}0\rightbararrowfill@{#1}{#2}}
\providecommand*\xRightbararrow[2][]{\ext@arrow 04{24}0\Rightbararrowfill@{#1}{#2}}
\providecommand*\xrightringarrow[2][]{\ext@arrow 01{26}0\rightringarrowfill@{#1}{#2}}
\providecommand*\xrighttriarrow[2][]{\ext@arrow 07{24}0\righttriarrowfill@{#1}{#2}}
\providecommand*\xRighttriarrow[2][]{\ext@arrow 07{24}0\Righttriarrowfill@{#1}{#2}}

\providecommand*\xmapsto[2][]{\ext@arrow 01{20}0\mapstofill@{#1}{#2}}
\providecommand*\xleftrightarrow[2][]{\ext@arrow 10{22}0\leftrightarrowfill@{#1}{#2}}
\providecommand*\xLeftrightarrow[2][]{\ext@arrow 10{27}0{\Leftrightarrowfill@}{#1}{#2}}

\makeatother

\crefname{equation}{}{}
\crefname{Lemma}{Lemma}{Lemmas}
\crefname{Thm}{Theorem}{Theorems}
\crefname{Defn}{Definition}{Definitions}
\crefname{Not}{Notation}{Notations}
\crefname{Ex}{Example}{Examples}
\crefname{Exs}{Examples}{Examples}
\crefname{sec}{Section}{Sections}
\crefname{Prop}{Proposition}{Propositions}
\crefname{Rk}{Remark}{Remarks}

\numberwithin{equation}{section}

\theoremstyle{plain}
\newtheorem{Thm}{Theorem}[section]
\newtheorem{Prop}[Thm]{Proposition}

\newtheorem{Lemma}[Thm]{Lemma}
\newtheorem*{Thm*}{Theorem}

\theoremstyle{definition}
\newtheorem{Defn}[Thm]{Definition}
\newtheorem{Not}[Thm]{Notation}

\newtheorem{Rk}[Thm]{Remark}

\newcommand{\GBA}{\mathrm{gr}\cat{\B \A lg}}
\newcommand{\BA}{\cat{B \A lg}}
\newcommand{\Mon}{\cat{Mon}}

\newcommand{\Top}{\J}
\newcommand{\BJ}{B_{\J}}

\newcommand{\dc}[2]{\left[\smash{{#1} \mathbin{\mid}{#2} }\right]}
\newcommand{\BM}{{\dc B M}}
\newcommand{\BJM}{{\dc {\BJ} M}}

\newcommand{\sheq}[2]{{\dbr{{#1}\mathrel{\!\texttt{\upshape =}\!}{#2}}}}

\begin{document}
\leftmargini=2em \title[Cartesian closed varieties II]{Cartesian
  closed varieties II: \\ links to algebra and self-similarity} \author{Richard Garner} \address{School of
  Math.~\& Phys.~Sciences, Macquarie University, NSW 2109,
  Australia} \email{richard.garner@mq.edu.au}

\date{\today}

\thanks{The support of Australian Research Council grant DP190102432
  is gratefully acknowledged.}

\begin{abstract}
  This paper is the second in a series investigating cartesian closed
  varieties. In first of these, we showed that every non-degenerate
  finitary cartesian variety is a variety of sets equipped with an
  action by a Boolean algebra $B$ and a monoid $M$ which interact to
  form what we call a matched pair $\BM$. In this paper, we show that
  such pairs $\BM$ are equivalent to \emph{Boolean restriction
    monoids} and also to \emph{ample source-\'etale topological
    categories}; these are generalisations of the Boolean inverse
  monoids and ample \'etale topological groupoids used to encode
  self-similar structures such as Cuntz and Cuntz--Krieger
  $C^\ast$-algebras, Leavitt path algebras and the $C^\ast$-algebras
  associated to self-similar group actions. We explain and
  illustrate these links, and begin the programme of understanding how
  topological and algebraic properties of such groupoids can be
  understood from the logical perspective of the associated varieties.
\end{abstract}
\maketitle
\setcounter{tocdepth}{1}
\tableofcontents
\section{Introduction}

This paper is a continuation of the investigations
of~\cite{Garner2023CartesianI} into \emph{cartesian closed
  varieties}---that is, varieties of single-sorted, possibly
infinitary algebras which, seen as categories, are cartesian closed.
One of the main results of \emph{op.~cit.}~was that the category of
non-degenerate, finitary, cartesian closed varieties is equivalent to
the category of non-degenerate \emph{matched pairs of algebras} $\BM$.
Here, a \emph{matched pair of algebras} comprises a Boolean algebra
$B$ and a monoid $M$ which act on each other in a way first described
in~\cite{Jackson2009Semigroups}; one way to say it is that $M$ acts on
$B$ via continuous endomorphisms of its associated Stone space, while
$B$ acts on $M$ so as to make it into a \emph{sheaf} of continuous
functions on $B$. When $M$ acts faithfully on $B$, the structure
generalises that of a \emph{pseudogroup}~\cite{Ehresmann1954Structures} of
automorphisms, where the generalisation is that $M$ is a monoid of
not-necessarily-invertible functions.

This description points to a connection between our matched pairs of
algebras and the study of self-similar structures in non-commutative
algebra, operator algebra and semigroup theory. Following the
pioneering work of Renault~\cite{Renault1980A-groupoid} and, later,
Steinberg~\cite{Steinberg2010Groupoid}, a key idea in this area has
been that analytic and algebraic objects such as the Cuntz
$C^\ast$-algebra or the Leavitt algebras can be constructed from
certain kinds of topological groupoids known as \emph{ample
  groupoids}; these are groupoids whose space of objects $C_0$ is a
Stone (= totally disconnected compact Hausdorff) space and which are
\emph{source-\'etale}, meaning that the source map
$s \colon C_1 \rightarrow C_0$ is a local homeomorphism.
In~\cite{Lawson2012Non-commutative}, Lawson showed that such groupoids
correspond under ``non-commutative Stone duality'' to \emph{Boolean
  inverse monoids}, which are abstract monoids of partial isomorphisms
equipped with extra structure allowing them to be represented on the
inverse monoid of partial homeomorphisms of a Stone space.

The first main result of this paper shows that the two-sorted notion
of matched pair of algebras $\BM$ corresponds to a single-sorted
notion which generalises that of a Boolean inverse monoid, namely,
that of \emph{Boolean restriction monoid}~\cite{Cockett2009Boolean} or
a \emph{modal restriction semigroup with preferential
  union}~\cite{Jackson2011Modal}; this is an abstract monoid of
partial \emph{functions} equipped with extra structure allowing it to
be represented on the monoid of partial \emph{endomorphisms} of a
Stone space. Thus, in Section~\ref{sec:links-non-comm} we prove
(Theorems~\ref{thm:6} and~\ref{thm:7}):
\begin{Thm*}
  The category of (Grothendieck) matched pairs of algebras is
  equivalent to the category of (Grothendieck) Boolean restriction
  monoids.
\end{Thm*}
We should explain the modifier ``Grothendieck''. The matched pairs of
algebras $\BM$ described above corresponds to \emph{finitary}
cartesian closed varieties. However, there are also what we have
termed \emph{Grothendieck} matched pairs $\BJM$ which correspond to
possibly \emph{infinitary} cartesian closed varieties. In these, our
Boolean algebra $B$ comes equipped with a collection $\J$ of
``well-behaved'' infinite partitions, encoding the operations of
infinite arity. Correspondingly, there is a notion of
\emph{Grothendieck Boolean restriction monoid} involving partial
functions which can be patched together over possibly \emph{infinite}
partitions from such a collection $\J$; these, then, are the two sides
of the extended correspondence above.

Now, as shown in~\cite{Cockett2021Generalising}, Boolean restriction
monoids correspond under an extended non-commutative Stone duality to
what might be termed \emph{ample topological categories}---namely,
source-\'etale topological categories with Stone space of objects.
Thus, our matched pairs $\BM$ present, among other things, the ample
topological groupoids of interest to operator algebraists. (In the
Grothendieck case, a little more care is necessary; for here, the
analogue of the Boolean prime ideal lemma may fail to hold, i.e.,
$B_\J$ may fail to have enough points, so there may be no faithful
representation by a topological category; nonetheless, in the spirit
of~\cite{Resende2006Lectures}, we do always obtain a zero-dimensional
\emph{localic} category.)

The preceding observations indicate a potentially interesting new
research direction. A particularly fruitful line of enquiry in recent
years has involved relating analytic properties of the
$C^\ast$-algebras generated by ample groupoids, and algebraic
properties of the corresponding algebras (``Steinberg algebras'') over
a ring. The new direction would seek to further relate these to
syntactic and semantic properties of the \emph{variety} associated to
a given ample groupoid. (At present, there is rather little to the
analytic or algebraic side that matches up with the varieties
associated to ample topological \emph{categories}, but some recent
progress has been made in~\cite{Castro2022Etale}.)

The second and third main results of this paper can be seen as first
steps in this new direction. We begin by re-addressing a question
considered by Johnstone in~\cite{Johnstone1985When}: when is a variety
a topos? As we recall in Section~\ref{sec:when-variety-topos} below, a
\emph{topos} is a finitely complete cartesian closed category with a
subobject classifier, and so we can equally well phrase the question
as: when is a cartesian closed variety a topos?
In~\cite{Johnstone1985When}, Johnstone gives a rather delicate
syntactic description, but using our now-richer understanding of
cartesian closed varieties, we can simplify this drastically. We will
show (Theorem~\ref{thm:8}):
\begin{Thm*}
  The cartesian closed variety of $\BJM$-sets is a topos just when,
  for any $b \neq 0 \in B$, there exists some $m \in M$ such that
  $m^\ast b = 1$; or equivalently, just when the associated
  topological or localic category is minimal.
\end{Thm*}
Here, $m^\ast b$ is the action of $m$ on $b$---which from the spatial
perspective is obtained by taking the inverse image of the clopen set
$b$ along the continuous endomorphism $m$. Rather than prove the above
theorem directly, we approach it via a new proof of one of the main
results of~\cite{Johnstone1990Collapsed}. Theorem~1.2 of
\emph{op.~cit.} shows that every cartesian closed variety arises as
the ``two-valued collapse'' of an essentially-unique topos $\E$, where
the ``two-valued collapse'' is obtained by restricting to those
objects whose support is either~$0$ or~$1$.
In~\cite{Johnstone1990Collapsed} the topos $\E$ whose collapse is a
given cartesian closed variety is found via a \emph{tour de force}
construction which leaves its nature rather mysterious. Our results
allow us to give a concrete presentation of $\E$ as a topos of
\emph{sheaves} on the (Grothendieck) matched pair of algebras which
classifies our variety. Once we have this (in
Proposition~\ref{prop:29}), Theorem~\ref{thm:8} follows easily.

The third main result of this paper describes the semantic and
syntactic properties of a variety which corresponds to its associated
topological or localic category actually being a \emph{groupoid}, and
as such in the more traditional purview of operator algebra. These
properties of a variety can be motivated by the case of $M$-sets, for
which the obvious ``groupoidal'' condition is that the monoid $M$
should in fact be a group. This syntactic condition on $M$ corresponds
to a \emph{semantic} one: the monoid $M$ is a group precisely when the
forgetful functor from $M$-sets to sets preserves the cartesian closed
structure. It turns out that exactly the same semantic condition
characterises the groupoidality of the associated category for an
\emph{arbitrary} $\BJM$; this is our Theorem~\ref{thm:9}, which shows,
among other things, that:
\begin{Thm*}
  \label{thm:1}
  The associated localic category of a Grothendieck matched pair
  $\BJM$ is a groupoid if, and only if, the forgetful functor from
  $\BJM$-sets to $\BJ$-sets preserves the cartesian closed structure.
\end{Thm*}

Corresponding to this semantic condition, we provide a syntactic
condition on $\BJM$ which is slightly complex, but is very natural in
terms of the associated Boolean restriction monoid, where it becomes
precisely the condition that this should be generated by its Boolean
inverse monoid of partial isomorphisms.

The final contribution made by this paper is not in further results,
but in further examples, which describe explicitly the cartesian
closed varieties which give rise to some of the better-known ample
topological groupoids studied in operator algebra. In particular, we
show (Section~\ref{sec:jonsson-tarski-topos-1}) that the \emph{Cuntz
  groupoid} $\mathfrak{O}_2$, whose $C^\ast$-algebra is the Cuntz
$C^\ast$-algebra $\O_2$, is the associated groupoid of the cartesian
closed variety---in fact a topos---of \emph{J\'onsson--Tarski
  algebras}, that is, sets $X$ endowed with an isomorphism
$X \cong X \times X$. This result has an obvious generalisation,
replacing $2$ by any finite cardinal $n$, but in fact, since we have
the notion of Grothendieck Boolean algebra available, we can consider
(Section~\ref{sec:infin-jonss-tarski}) an \emph{infinitary}
generalisation which replaces $2$ by an arbitrary set $A$, and
considers the topos of sets endowed with an isomorphism
$X \rightarrow X^A$. As a further generalisation of this, we describe
(Section~\ref{sec:nekrashevych-toposes}) a topos which encodes the
topological groupoid associated to a \emph{self-similar group action}
in the sense of~\cite{Nekrashevych2005Self-similar,
  Nekrashevych2009Cstar}. For our final substantive example
(Section~\ref{sec:cuntz-krieg-topos-1}), we describe
following~\cite{Leinster2007Jonsson, Henry2016Convolution} a cartesian
closed variety which encodes the \emph{graph groupoid} associated to
any directed graph by the machinery of~\cite{Kumjian1997Graphs}.

We should note that here we have only really scratched the surface of
the links with operator algebra. For example, the varieties just
described can be extended to ones which encode the topological
groupoids associated to higher rank graphs~\cite{Kumjian2000Higher};
self-similar actions of groupoids on
graphs~\cite{Laca2018Equilibrium}; or graphs of
groups~\cite{Brownlowe2017C}. Moreover, it seems there may be
low-hanging fruit towards a general structure theory of matched pairs
$\BJM$; for example, both the self-similar group examples studied here
and also the examples involving higher-rank graphs should arise as
instances of a \emph{Zappa-Sz\'ep product} or \emph{distributive law}
between matched pairs $\BJM$ and $\dc {C_\K}{N}$. In a similar spirit,
we could enquire after a general notion of \emph{correspondence}
between two matched pairs, and a \emph{Cuntz--Pimsner} construction
for building new matched pairs out of such a correspondence: but all
of this must await future work.

\section{Background}
\label{sec:background}

\subsection{$B$-sets and $\BJ$-sets}
In this preliminary section, we gather together background
from~\cite{Garner2023CartesianI} that will be needed for the further
developments of this paper. We begin by recalling the notion of an
``action'' of a Boolean algebra on a set, due
to~\cite{Bergman1991Actions}.

\begin{Defn}[$B$-sets]
  \label{def:5}
  Let $B = (B, \wedge, \vee, 0, 1, (\thg)')$ be a non-degenerate
  Boolean algebra (i.e., $0 \neq 1$). A \emph{$B$-set} is a set $X$
  with an operation $B \times X \times X \rightarrow X$, written
  $(b,x,y) \mapsto b(x,y)$, satisfying the axioms
  \begin{equation}
    \label{eq:2}
    \begin{aligned}
      b(x,x) &= x \qquad b(b(x,y),z) = b(x,z) \qquad b(x,b(y,z)) = b(x,z)\\
      1(x,y) &= x \qquad b'(x,y) = b(y,x) \qquad (b \wedge c)(x,y) = b(c(x,y),y)\rlap{ .}
    \end{aligned}
  \end{equation}
\end{Defn}
One way to think of a $B$-set is as a set of ``random variables''
varying over the (logical) state space $B$; then the element $b(x,y)$
can be interpreted as the random variable
$\mathsf{if}\ b\ \mathsf{then}\ x\ \mathsf{else}\ y$. Another
interpretation is that elements of a $B$-set $X$ are objects with
``parts'' indexed by the elements of $B$; then $b(x,y)$ is the result
of restricting $x$ to its $b$-part and $y$ to its $b'$-part, and
glueing the results back together again. One readily recognises this
as part of the structure of a \emph{sheaf} on the Boolean algebra
$B$---more precisely, the structure borne by the set of global
sections of such a sheaf. Not every sheaf on $B$ has a global section;
but for one which does, \emph{every} section can be extended to a
global section, so that the $B$-sets are equally those sheaves on $B$
which are either empty, or have at least one global section.

Now, the notion of $B$-set is a finitary one, and this may be
inconvient in a Boolean algebra which admits \emph{infinite}
partitions; one may wish to ``logically condition'' on the elements of
such an infinite partition, but none of the finitary $B$-set
operations allow for this. This can be rectified with a more refined
kind of action by a Boolean algebra that is equipped with a
suitable collection of ``well-behaved'' infinite joins:

\begin{Defn}[Partition]
  \label{def:6}
  Let $B$ be a Boolean algebra and $b \in B$. A \emph{partition of
    $b$} is a subset $P \subseteq B \setminus \{0\}$ such that
  $\bigvee P = b$, and $c \wedge d = 0$ whenever $c \neq d \in P$. An
  \emph{extended partition of $b$} is a subset $P \subseteq B$
  (possibly containing $0$) satisfying the same conditions. If $P$ is
  an extended partition of $b$, then we write
  $P^- = P \setminus \{0\}$ for the corresponding partition. We 
  say merely ``partition'' to mean ``partition of $1$''.
\end{Defn}
\begin{Defn}[Zero-dimensional topology, Grothendieck Boolean algebra]
  \label{def:7}
  A \emph{zero-dimensional topology} on a Boolean
  algebra $B$ is a collection $\Top$ of partitions of $B$ which
  contains every finite partition, and satisfies:
  \begin{enumerate}[(i)]
  \item If $P \in \Top$, and $Q_b \in \J$ for each $b \in P$, then
    $P(Q) = \{ b \wedge c : b \in P, c \in Q_b\}^- \in \J$;
  \item If $P \in \J$ and $\alpha \colon P \rightarrow I$ is a
    surjective map, then each join $\bigvee \alpha^{-1}(i)$ exists and
    $\alpha_!(P) = \{\textstyle\bigvee \alpha^{-1}(i) : i \in I\} \in
    \Top$.
  \end{enumerate}
  A \emph{Grothendieck Boolean algebra} $\BJ$ is a Boolean algebra $B$
  with a zero-dimensional topology $\J$. A \emph{homomorphism} of
  Grothendieck Boolean algebras $f \colon \BJ \rightarrow C_\K$ is a
  Boolean homomorphism $f \colon B \rightarrow C$ such that $P \in \J$
  implies $f(P)^- \in \K$.
  If $\BJ$ is a Grothendieck Boolean algebra and $b \in B$, then we
  write $\Top_b$ for the set of partitions of $b$ characterised by:
    \begin{equation*}
      P \in \J_b \iff P \cup \{b'\} \in \J \iff P \subseteq Q \in \J \text{ and } \bigvee P = b\rlap{ .}
    \end{equation*}
\end{Defn}

Given a Grothendieck Boolean algebra, we can now define a variety of
(infinitary) algebras which allows for infinite conditioning over its
privileged partitions.
\begin{Defn}[$\BJ$-sets]
  \label{def:9}
  Let $\BJ$ be a non-degenerate Grothendieck Boolean algebra. A
  \emph{$\BJ$-set} is a $B$-set $X$ equipped with a function
  $P \colon X^P \rightarrow X$ for each infinite $P \in \Top$,
  satisfying:
  \begin{equation}\label{eq:5}
    P(\lambda b.\, x) = x\ \ \ \ \  P(\lambda b.\, b(x_b, y_b)) = P(\lambda b.\, x_b)\ \ \ \ \ 
    b(P(x), x_b) = x_b \text{ $\forall b \in P$.}
  \end{equation}
\end{Defn}
In this definition, and henceforth, we use the following notational
conventions:
\begin{Not}
  Given sets $I$ and $J$ we write $J^I$ for
  the set of functions from $I$ to $J$. If $u \in J^I$, we write $u_i$
  for the value of the function $u$ at $i \in I$; on the other hand,
  given a family of elements $(t_i \in J : i \in I)$, we write
  $\lambda i.\, t_i$ for the corresponding element of $J^I$. We may
  identify a natural number $n$ with the set
  $\{1, \dots, n\} \subseteq \mathbb{N}$.
\end{Not}

It turns out (\cite[Proposition~3.9]{Garner2023CartesianI}) that an
operation $P$ on a $B$-set $X$ satisfying the axioms~\eqref{eq:5} is
unique if it exists, and that any homomorphism of $B$-sets
$f \colon X \rightarrow Y$ will preserve it. Thus, the category of
$\BJ$-sets and homomorphisms is a \emph{full} subcategory of the
category of $B$-sets. Moreover, any non-degenerate Boolean algebra $B$
has a least zero-dimensional topology given by the collection of all
finite partitions of $B$, and in this case, $B_\J$-sets are just
$B$-sets; as such, we may without loss of generality work exclusively
with $\BJ$-sets in what follows.

As explained in~\cite{Garner2023CartesianI}, $\BJ$-set structure on a
set $X$ can also be described in terms of a family of equivalence
relations $\equiv_b$ which we read as
``$\mathsf{if}\ b\ \mathsf{then}\ x = y$'' or as ``$x$ and $y$ have
the same restriction to $b$''. The following result combines
Propositions~3.2,~3.10 and~3.11 and Lemma~3.12 of \emph{op.~cit.}

\begin{Prop}
  \label{prop:2}
  Let $\BJ$ be a non-degenerate Grothendieck Boolean algebra. Any
  $\BJ$-set structure on a set $X$ induces equivalence relations
  $\equiv_b$ (for $b \in B$) given by:
  \begin{equation*}
    x \equiv_b y \quad \quad \iff \quad \quad b(x,y) = y\rlap{ .}
  \end{equation*}
  These equivalence relations satisfy the following axioms:
  \begin{enumerate}[(i)]
  \item If $x \equiv_b y$ and $c \leqslant b$ then $x \equiv_c y$;
  \item $x \equiv_1 y$ if and only if $x = y$, and $x \equiv_0 y$ always;
  \item For any $P \in \J_b$, if $x \equiv_{c} y$ for all $c \in P$,
    then $x \equiv_{b} y$;
  \item For any $P \in \J$ and $x \in X^P$, there is $z \in X$ such
    that $z \equiv_b x_b$ for all $b \in P$.
  \end{enumerate}
  Any family of equivalence relations $(\mathord{\equiv_b} : b \in B)$
  satisfying (i)--(iv) arises in this way from a unique $\BJ$-set
  structure on $X$ whose operations are characterised by the fact that
  $b(x,y) \equiv_b x$ and $b(x,y) \equiv_{b'} y$ for all $b \in B$ and
  $x,y \in X$; and that $P(x) \equiv_b x_b$ for all $P \in \J$,
  $x \in X^P$ and $b \in P$. Such a $\BJ$-set structure is equally
  well determined by equivalence  relations  $\equiv_b$ satisfying
  (i)  and:
  \begin{enumerate}
    \looseness=-1
  \item[(ii)$'$] For any $P \in \J$ and $x \in X^P$, there is a unique
    $z \in X$ with $z \equiv_{b} x_b$ for all $b \in P$.
  \end{enumerate}
  Under the above correspondences, a function $X \rightarrow Y$
  between $\BJ$-sets is a homomorphism just when it preserves each
  $\equiv_b$.
\end{Prop}

\begin{Rk}
  \label{rk:7}
  The conditions (i)--(iii) imply that, for all elements $x,y$ in a
  $\BJ$-set $X$, the set $\sheq x y = \{b \in B : x \equiv_b y\}$ is
  an ideal of the Boolean algebra $B$, and in fact a
  \emph{$\J$-closed} ideal---meaning that $b \in \sheq x y$ whenever
  $P \subseteq \sheq x y$ for some $P \in \J_b$.
\end{Rk}

\subsection{Matched pairs of algebras $\BM$ and $\BJM$}

We now describe the algebraic structure
which~\cite{Garner2023CartesianI} identifies as encoding precisely the
non-degenerate cartesian closed varieties. In the finitary case, this
structure was already considered in~\cite[\sec
4]{Jackson2009Semigroups}, in a related, though different, context.

\begin{Defn}[Matched pair of algebras]
  \label{def:16}
  A non-degenerate \emph{Grothendieck matched pair of algebras} $\BJM$
  comprises a non-degenerate Grothendieck Boolean algebra $\BJ$; a
  monoid $M$; $\BJ$-set structure on $M$, written as
  $b, m, n \mapsto b(m,n)$; and left $M$-set structure on $B$, written
  as $m, b \mapsto m^\ast b$. We require that $M$ acts on $\BJ$ by
  Grothendieck Boolean homomorphisms, and that  the  following axioms hold:
  \begin{itemize}
  \item $b(m,n)p = b(mp,np)$;
  \item $m(b(n,p)) = (m^\ast b)(mn,mp)$; and
  \item $b(m,n)^\ast(c) = b(m^\ast c, n^\ast c)$,
  \end{itemize}
  for all $m,n,p \in M$ and $b,c \in B$. Here, in the final axiom, we
  view $B$ itself is a $B$-set under the operation of conditioned
  disjunction $b(c,d) = (b \wedge c) \vee (b' \wedge d)$. These axioms
  are equivalently the conditions that:
  \begin{itemize}
  \item $m \equiv_b n  \implies mp \equiv_b np$;
  \item $n \equiv_b p  \implies mn \equiv_{m^\ast b} mp$;
  \item $m \equiv_b n \implies m^\ast c \equiv_b n^\ast c$, i.e., $b
    \wedge m^\ast c = b \wedge n^\ast c$.
  \end{itemize}
  When $\J$ is the topology of finite partitions, we can drop the $\J$
  and the modifier ``Grothendieck'' and speak simply of a
  \emph{matched pair of algebras} $\BM$.

  A \emph{homomorphism}
  $\dc \varphi f \colon \BJM \rightarrow \dc {B'_{\J'}} {M'}$ of Grothendieck
  matched pairs of algebras comprises a Grothendieck Boolean
  homomorphism $\varphi \colon \BJ \rightarrow B'_{\J'}$ and a monoid
  homomorphism $f \colon M \rightarrow M'$ such that, for all
  $m,n \in M$ and $b \in B$:
  \begin{equation}\label{eq:28}
    \varphi(b)(f(m),f(n)) = f(b(m,n)) \ \ \  \text{and} \ \ \ f(m)^\ast(\varphi(b)) = \varphi(m^\ast b)\rlap{ ,}
  \end{equation}
  or equivalently, such that
  \begin{equation}\label{eq:29}
    \ \ \quad m \equiv_b n \implies f(m) \equiv_{\varphi(b)} f(n) \qquad \text{and} \qquad f(m)^\ast(\varphi(b)) = \varphi(m^\ast b)\text{ .} 
  \end{equation}
\end{Defn}
The cartesian closed variety which corresponds to the Grothendieck
matched pair of algebras $\BJM$ can be described explicitly as the
variety of \emph{$\BJM$-sets}:

\begin{Defn}[Variety of $\BJM$-sets]
  \label{def:17}
  Let $\BJM$ be a non-degenerate matched pair of algebras. A
  \emph{$\BJM$-set} is a set $X$ endowed with $\BJ$-set structure and
  $M$-set structure, such that in addition we have:
  \begin{equation}\label{eq:30}
    b(m,n) \cdot x = b(m \cdot x, n \cdot x) \qquad \text{and} \qquad m \cdot b(x,y) = (m^\ast b)(m \cdot x, m \cdot y)
  \end{equation}
  for all $b\in B$, $m,n \in M$ and $x,y \in X$; or equivalently, such
  that:
  \begin{equation}
    \label{eq:53}
    m \equiv_b n \implies m \cdot x \equiv_b n \cdot x \qquad \text{and} \qquad
    x \equiv_b y \implies m \cdot x \equiv_{m^\ast b} m \cdot y\rlap{ .}
  \end{equation}
  A homomorphism of $\BJM$-sets is a function which preserves both
  $\BJ$-set and an $M$-set structure. We write $\BJM\text-\cat{Set}$
  for the variety of $\BJM$-sets. In the finitary case, we speak
  of ``$\BM$-sets'' and the (finitary) variety
  $\BM\text-\cat{Set}$.
\end{Defn}

The fact that $\BJM$-sets are indeed a cartesian closed variety was
verified in~\cite[Proposition~7.11]{Garner2023CartesianI}, which we
recall as:

\begin{Prop}
  \label{prop:3}
  For any non-degenerate Grothendieck Boolean matched pair $\BJM$, the
  category $\BJM\text-\cat{Set}$ is cartesian closed.
\end{Prop}
\begin{proof}
Given $\BJM$-sets $Y$ and $Z$, the function
space $Z^Y$ is the set of $\BJM$-set homomorphisms
$f \colon M \times Y \rightarrow Z$. We make this into an $M$-set
under the action
\begin{equation*}
  (m \cdot f)(n,y) = f(nm, y)\rlap{ ,}
\end{equation*}
and  into a $\BJ$-set via the equivalence relations:
\begin{equation*}
  f \equiv_b g \qquad \iff \qquad f(m,y) \equiv_{m^\ast b} g(m,y) \text{ for all $m,y \in M \times Y$.}
\end{equation*}
The evaluation homomorphism
$\mathrm{ev} \colon Z^Y \times Y \rightarrow Z$ is given by
$\mathrm{ev}(f, y) = f(1, y)$; and given a $\BJM$-set homomorphism
$f \colon X \times Y \rightarrow Z$, its transpose
$\bar f \colon X \rightarrow Z^Y$ is given by
$\bar f(x)(m, y) = f(mx, y)$.
\end{proof}  

Conversely, if we are presented with a cartesian closed variety $\C$,
then we can reconstruct the $\BJM$ for which
$\C \cong \BJM\text-\cat{Set}$ using
\cite[Proposition~7.12]{Garner2023CartesianI}, which we restate
(slightly less generally) here as:

\begin{Prop}
  \label{prop:44}
  Let $\C$ be a non-degenerate cartesian closed variety, and let
  $X \in \C$ be the free algebra on one generator. Then $\C \cong
  \BJM\text-\cat{Set}$, where
  \begin{enumerate}[(a),itemsep=0.25\baselineskip]
  \item The monoid $M$ is $\C(X,X)$, with unit $\mathrm{id}_X$ and
    product given by composition in \emph{diagrammatic} order, i.e.,
    $mn$ is $m$ followed by $n$;
  \item Writing $1$ for the one-element algebra, and
    $\iota_\top, \iota_\bot \colon 1 \rightarrow 1+1$ for the two
    coproduct injections, the Boolean algebra $B$ is $\C(X,1 + 1)$
    with operations
    \begin{gather*}
      \smash{1 = X \xrightarrow{!} 1 \xrightarrow{\iota_\top} 1+1 \qquad \qquad b' = X \xrightarrow{b} 1 + 1 \xrightarrow{\spn{\iota_2, \iota_1}} 1+1}\\
      \text{and} \quad 
      \smash{b \wedge c = X \xrightarrow{(b,c)} (1 + 1) \times (1 + 1) \xrightarrow{\wedge} 1+1}
    \end{gather*}
    where $\wedge \colon (1+1) \times (1+1) \rightarrow 1+1$ satisfies
    $\wedge \circ (\iota_i \times \iota_j) = \iota_{i \wedge j}$ for
    $i,j \in \{\top, \bot\}$;
  \item The zero-dimensional coverage $\J$ on $B$ has
    $P \subseteq B$ is in $\J$ just when there exists a map
    $f \colon X \rightarrow P \cdot 1$ with
    $\spn{\delta_{bc}}_{b \in B} \circ f = c$ for all $c \in P$, where
    here $\delta_{bc} \colon 1 \rightarrow 1+1$ is given by
    $\delta_{bc} = \iota_{\top}$ when $b = c$ and
    $\delta_{bc} = \iota_\bot$ otherwise;
  \item $M$ acts on $B$ via precomposition;
  \item $B$ acts on $M$ via:
    \begin{equation*}
      \smash{(b, m, n) \mapsto X \xrightarrow{(b, \mathrm{id})} (1+1) \times X \xrightarrow{\cong} X + X \xrightarrow{\spn{m,n}} X\rlap{ .}}
    \end{equation*}
  \end{enumerate}
  The isomorphism $\C \cong \BJM\text-\cat{Set}$ sends $Y \in \C$ to
  the set $\C(X,Y)$, made into a $\BJM$-set via the action of $M$ by
  precomposition, and the action of $B$ by
  \begin{equation*}
    \smash{(b, x, y) \mapsto X \xrightarrow{(b, \mathrm{id})} (1+1) \times X \xrightarrow{\cong} X + X \xrightarrow{\spn{x,y}} Y\rlap{ .}}
  \end{equation*}
\end{Prop}

Finally, by~\cite[Remark~7.9]{Garner2023CartesianI}, the \emph{free}
$\BJM$-set on a given set of generators $X$ can be described in terms
of the notion of $B$-valued \emph{distribution}:

\begin{Defn}
  \label{def:10}
  Let $\BJ$ be a non-degenerate Grothendieck Boolean algebra. A
  \emph{$\BJ$-valued distribution} on a set $I$ is a function
  $\omega \colon I \rightarrow B$ whose restriction to
  $\mathrm{supp}(\omega) = \{i \in I : \omega(i) \neq 0\}$ is an
  injection onto a partition in $\J$. We write $T_{\BJ}(I)$ for the
  set of $\BJ$-valued distributions on $I$.
\end{Defn}

Now the free $\BJM$-set on a set $X$ is given by the product of
$\BJM$-sets $M \times T_{\BJ} X$. Here, $M$ is seen as a $\BJM$-set
via its canonical structures of $\BJ$- and $M$-set, while $T_{\BJ}(X)$
is seen as a $\BJ$-set via
\begin{equation*}
  \omega \equiv_b \gamma \qquad \iff \qquad b \wedge \omega(x) = b \wedge \gamma(x) \text{ for all } x \in X
\end{equation*}
and as an $M$-set via 
$n \cdot (m, \omega) = (nm, n^\ast \circ \omega)$. The function
$\eta \colon X \rightarrow M \times T_{\BJ}(X)$ exhibiting
$M \times T_{\BJ}(X)$ as free on $X$ is given by
$x \mapsto (1, \pi_x)$.

\section{Matched pairs as Boolean restriction monoids}
\label{sec:links-non-comm}

In this section, we prove our first main result, identifying
(Grothendieck) matched pairs of algebras with (Grothendieck) Boolean
restriction monoids. We begin by recalling the notion of restriction
monoid. These appear in the semigroup literature under the name
``left weakly $E$-ample semigroups''~\cite{Fountain1999Munn}, with the below axiomatisation
first appearing in~\cite{Jackson2001An-invitation}; the name
``restriction monoid'' is now standard, with the nomenclature coming
from~\cite{Cockett2002Restriction}. See~\cite{Hollings2009From} for a
historical overview.

\renewcommand{\r}[1]{{#1}^{+}}
\newcommand{\q}[2]{\res {#1} {#2}}
\begin{Defn}[Restriction monoid]
  \label{def:20}
  A (left) \emph{restriction monoid} is a monoid $S$ endowed with a
  unary operation $s \mapsto \r s$ (called \emph{restriction}),
  satisfying the axioms
  \begin{equation*}
    \r s s = s \qquad \r{(\r st)} = \r s \r t \qquad \r s \r t = \r t \r s \qquad \text{and} \qquad  s \r t = \r{(st)}s \rlap{ .}
  \end{equation*}
  A \emph{homomorphism} of restriction monoids is a monoid
  homomorphism $\varphi$ which also preserves restriction, i.e.,
  $\varphi(s^+) = \varphi(s)^+$.
\end{Defn}

Some basic examples of a restriction monoid are the monoid of partial
endofunctions of a set $X$, or the partial continuous endofunctions of
a space $X$. In both cases, the restriction of a partial map
$f \colon X \rightharpoonup X$ is the idempotent partial function
$\r f \colon X \rightharpoonup X$ with $\r f(x) = x$ if $x$ is defined
and $\r f(x)$ undefined otherwise. In general, each element $s^+$ in a
restriction monoid $S$ is idempotent, and an element $b$ is of the
form $s^+$ if, and only if, $b^+ = b$; we write $E(S)$ for the set of
all $s^+$ and call them \emph{restriction idempotents}. On the other
hand, we call $s \in S$ \emph{total} if $\r s = 1$. Total maps are
easily seen to constitute a submonoid $\mathrm{Tot}(S)$ of $S$. 

There is a partial order $\leqslant$ on any restriction monoid $S$
defined by $s \leqslant t$ iff $\r s t = s$, expressing that $s$ is
the restriction of $t$ to a smaller domain of definition. When ordered
by $\leqslant$, the set of restriction idempotents $E(S)$ becomes a
meet-semilattice, with top element $1$ and binary meet
$b \wedge c = bc$. Of course, $b, c \in E(S)$ are \emph{disjoint} if
$bc = 0$; more generally, we say that $s,t \in S$ are \emph{disjoint}
if $\r s \r t = 0$.

The above axioms have various consequences; one of the more important
is the fact that $(s\,t^+)^+ = (st)^+$, which can be derived as follows.
\begin{equation*}
  (s\,t^+)^+ = ((st)^+s)^+ = (st)^+ s^+ = s^+ (st)^+ = (s^+st)^+ = (st)^+
\end{equation*}

\begin{Defn}[Boolean restriction monoid~\cite{Cockett2009Boolean}]
  \label{def:21}
  A \emph{Boolean restriction monoid} is a restriction monoid $S$ 
  in which:
  \begin{itemize}
  \item $(E(S), \leqslant)$ admits a negation $(\thg)'$ making it into a
    Boolean algebra;
  \item The least element $0$ of $E(S)$ is also a least element of $S$;
  \item Every pair of disjoint elements $s,t \in S$ has a join
    $s \vee t$ with respect to $\leqslant$;
  \item We have $s0 = 0$ and $s(t \vee u) = st \vee
    su$ for all $s, t,u \in S$ with $t,u$ disjoint.
  \end{itemize}
  As explained in~\cite[Proposition~2.14]{Cockett2011Differential},
  these conditions imply moreover that:
  \begin{itemize}
  \item $0 s = 0$ and $(s \vee t)u = su \vee tu$ for all
    $s,t,u \in S$ with $t,u$ disjoint;
  \item $0^+ = 0$ and $\r{(s \vee t)} = \r s \vee \r t$.
  \end{itemize}
  A homomorphism of Boolean restriction monoids is a restriction
  monoid homomorphism $S \rightarrow T$ which also preserves the least
  element $0$ and joins of disjoint elements; or equivalently,
  by~\cite[Lemma~2.10]{Cockett2021Generalising}, which restricts to a
  Boolean homomorphism $E(S) \rightarrow E(T)$.
\end{Defn}

Boolean restriction monoids are also the same thing as the~\emph{modal
  restriction semigroup with preferential union}
of~\cite{Jackson2011Modal}. We now wish to show, further, that
non-degenerate Boolean restriction monoids are coextensive with
non-degenerate matched pairs of algebras. In our arguments we will
freely use basic consequences of the restriction monoid axioms as
found, for example, in~\cite[Lemma~2.1]{Cockett2002Restriction}. In
one direction, we have:

\begin{Prop}
  \label{prop:25}
  Let $S$ be a non-degenerate Boolean restriction monoid (i.e.,
  $0\neq 1$ in $S$). The Boolean
  algebra $B = (E(S), \leqslant)$ and the monoid $M = \mathrm{Tot}(S)$
  constitute a non-degenerate matched pair of algebras
  $S^\downarrow = \BM$, where $B$ becomes an $M$-set by taking
  $m^\ast b = (mb)^+$, and $M$ becomes a $\BJ$-set by taking
  $m \equiv_b n \iff bm = bn$.
\end{Prop}
\begin{proof}
  For axiom (i), if $c \leqslant b$, then $cb = c$
  and so $bm = bn$ implies $cm=cbm=cbn=cn$, i.e., $m \equiv_b n$
  implies $m \equiv_c n$. For axiom (ii), we have $m \equiv_{\top} n$
  just when $1m = 1n$, i.e., when $m = n$. For (iii), if
  $m \equiv_b n$ and $m \equiv_c n$, then
  $(b \vee c)m = bm \vee cm = bn \vee cn = (b \vee c)n$ so that
  $m \equiv_{b \vee c} n$. Finally, for (iv), if $m,n \in M$ and
  $b \in B$, then the element $bm \vee b'n$ is clearly total, and
  satisfies $b(bm \vee b'n) = bbm \vee bb'n = bm \vee 0 n = bm$ and
  similarly $b'(bm \vee b'n) = b'n$; whence $b(m,n) = bm \vee b'm$
  satisfies $b(m,n) \equiv_b m$ and $b(m,n) \equiv_{b'} n$ as desired.
  
  We next check that $m \mapsto m^\ast$ is an action by
  Boolean homomorphisms.
  Firstly:
  \begin{equation*}
    1^\ast b= \r{(1b)} = \r b = b \qquad \text{and} \qquad 
    m^\ast n^\ast b = \r{(m \r{(nb)})} =
    \r{(mnb)} = (mn)^\ast b\rlap{ .}
  \end{equation*}
  Next, we have $m^\ast(1) = \r{(m1)} = \r m =  1$
  since $m$ is assumed \emph{total}, and
  \begin{equation*}
    m^\ast(b \wedge c) = \r{(mbc)} = \r{(m \r b c)} = \r{(\r {(mb)} mc)} = \r{(mb)}\r{(mc)} = (m^\ast b) \wedge (m^\ast c)\rlap{ .}
  \end{equation*}
  Furthermore, since $m^\ast(b) \wedge m^\ast(b') = m^\ast(b \wedge
  b') = m^\ast(0) = \r{(m0)} = 0$ and
  \begin{equation*}
    m^\ast(b) \vee m^\ast(b') = \r{(mb)} \vee \r{(mb')} = \r{(m(b \vee b'))} = \r{m} = 1
  \end{equation*}
  we have $m^\ast b' = (m^\ast b)'$ so that $m^\ast$ is a Boolean
  homomorphism. It remains to check the three axioms for a matched
  pair of algebras. Axiom (i) is the trivial fact that $bm = bn$
  implies $bmp = bnp$. Axiom (ii) is the calculation
  \begin{equation*}
    bn = bp \quad \implies \quad (m^\ast b)mn = \r{(mb)}mn = mbn = mbp = \r{(mb)}mp = (m^\ast b)mp\rlap{ ,}
  \end{equation*}
  and, finally, axiom (iii) is:
  \begin{equation*}
    bm = bn \quad \implies \quad b \wedge m^\ast c = b\r{(mc)} = \r{(bmc)} = \r{(bnc)} = b \r{(nc)} = b \wedge n^\ast c\text{ .}\qedhere
  \end{equation*}
\end{proof}

In the converse direction, we have the following construction, which
also appears, in a more general context, in unpublished work of
Stokes~\cite{Stokes2020Restriction}.

\begin{Prop}
  \label{prop:26}
  For any non-degenerate matched pair of algebras $\BM$, there is a
  non-degenerate Boolean restriction monoid $S$ with
  $S^\downarrow \cong \BM$.
\end{Prop}
\begin{proof}
  We define $S = \{(b,m) : b \in B, m \in M / \mathord{\equiv_b}\}$,
  whose elements we write more suggestively as $\q m b$. We claim this
  is a Boolean restriction monoid on taking $1 = \q 1 1$,
  $\r{(\q m b)} = \q 1 b$ and
  $\q m b \q n c = \q {mn} {b \wedge m^\ast c}$. First, the
  multiplication $\q m b \q n c$ is well-defined, as if
  $m \equiv_b m'$ and $n \equiv_{c} n'$, then
  $m^\ast c \equiv_{b} (m')^\ast c$, i.e.,
  $b \wedge m^\ast c = b \wedge (m')^\ast c$; moreover, we have
  $mn \equiv_b mn'$ and $mn' \equiv_{m^\ast c} mn'$, whence
  $mn \equiv_{b \wedge m^\ast c} m'n'$. So
  $\q{mn}{b \wedge m^\ast c} = \q{m'n'}{b \wedge (m')^\ast c}$ as
  required.

  We now check the monoid axioms for $S$, noting the
  equality $\q 1 b \q n c = \q n {b \wedge c}$, which we will use
  repeatedly. For the unit axioms,
  $\q 1 1 \q m b = \q m {1 \wedge b} = \q m b$ and
  $\q m b \q 1 1 = \q m {b \wedge m^\ast 1} = \q m b$. For
  associativity,
  \begin{align*}
    (\q m b \q n c)\q p d &= \q {mn} {b \wedge m^\ast c} \q p d = \q {mnp} {b \wedge m^\ast c \wedge (mn)^\ast d}
    = \q {mnp} {b \wedge m^\ast c \wedge m^\ast n^\ast d} \\&= \q {mnp} {b \wedge m^\ast(c \wedge n\ast d)} 
    = \q m b \q {np} {c \wedge n^\ast d} = \q m b (\q n c \q p d)\rlap{ .}
  \end{align*}
  The following calculations now establish the four restriction monoid
  axioms:
  \begin{align*}
    \r{\q m b\!}\q m b &= \q 1 b \q m b = \q m {b \wedge b} = \q m b\\
    \smash{\r{\smash{(\r{\q m b\!}\!\q n c)}}} &= \r{(\q 1 b\! \q n
      c)} = \r{\q n {b \wedge c}\!} = \q 1 {b \wedge c} = \q 1 b \!\q 1 c =
    \r{\q n c\!}\r{\q m b\!}\\
    \r{\q m b\!}\r{\q n c\!} &= \q 1 b \q 1 c = \q 1 {b \wedge c} = \q 1 {c \wedge b} = \q 1 c \q 1 b =
    \r{\q n c\!}\r{\q m b\!}\\
    \q m b\r{\q n c\!} &= \q m b \q 1 c = \q m {b \wedge m^\ast c} =
    \q 1 {b \wedge m^\ast c}\q m b = \r{\q {mn} {b \wedge m^\ast c}\!} \q m b = \r{(\q m b\! \q n c)}\q m b\text{ .}
  \end{align*}
  So $S$ is a restriction monoid, wherein
  $E(S) = \{\res 1 b : b\in B\}$, and $\q m b \leqslant \q n c$ just
  when $b \leqslant c$ and $m \equiv_b n$. In particular, the map
  $B \rightarrow E(S)$ sending $b$ to $\res 1 b$ is an isomorphism of
  partially ordered sets, and so $E(S)$ is a Boolean algebra.
  Moreover, the least element $\q 1 0$ of $E(S)$ is a least element of
  $S$, as $1 \equiv_0 m$ is always true.

  We next show that any pair $s = \q m b$ and $t= \q n c \in S$ which
  are disjoint (i.e., $b \wedge c = 0$) have a join with respect to
  $\leqslant$. We claim $u = \res {b(m,n)}{b \vee c}$ is suitable.
  Indeed, as $b \leqslant b \vee c$ and $m \equiv_b b(m,n)$, we have
  $s \leqslant u$; while as $c \leqslant b \vee c$ and
  $n \equiv_{c} b(m,n)$ (since $c \leqslant b'$) we also have
  $t \leqslant u$. Now let $v = \res p d$ and suppose
  $s,t \leqslant v$. Then $b,c \leqslant d$ and so
  $b \vee c \leqslant d$; moreover, $m \equiv_b p$ and $n \equiv_c p$
  and so also $b(m,n) \equiv_b p$ and $b(m,n) \equiv_c p$. Thus
  $b(m,n) \equiv_{b \vee c} p$ and so $u \leqslant v$ as required.

  Finally, we show joins are stable under left multiplication. For
  the nullary case we have
  $\q m b \q 1 0 = \q m {b \wedge m^\ast 0} = \q m 0 = \q 1 0$.
  For binary joins, given $s = \q m b$ and $t = \q n c$ and $u = \q p
  d$ with $c, d$ disjoint, we necessarily have $st \vee su \leqslant s(t \vee u)$, since
  $st \leqslant s(t \vee u)$ and $su \leqslant s(t \vee u)$; so it
  suffices to show $(st \vee su)^+ = (s(t \vee u))^+$. But:
  \begin{align*}
    (st \vee su)^+ &= (\q {mn}{b \wedge m^\ast c} \vee \q {mp} {b \wedge m^\ast d})^+ = \q 1 {(b \wedge m^\ast c) \vee (b \wedge m^\ast d)}\\
    \text{while} \ \ (s(t \vee u))^+ &= (s(t \vee u)^+)^+ = (\q m b \q 1 {c \vee d})^+  = \q 1 {b \wedge m^\ast(c \vee d)}
  \end{align*}
  which are the same since $m^\ast$ is a Boolean homomorphism.

  This proves $S$ is a Boolean restriction monoid. Now we already saw
  that $b \mapsto \res 1 b$ is an isomorphism of Boolean algebras
  $B \rightarrow E(S)$, and the map $m \mapsto \res m 1$ is likewise a
  monoid isomorphism $M \rightarrow \mathrm{Tot}(S)$;
  To see that these maps constitute an isomorphism of
  matched pairs of algebras $\BM \rightarrow
  S^\downarrow$,  we must check the two axioms in~\eqref{eq:28}.
  On the one hand, for all
 $(b,m,n) \in B \times M \times M$, we have
  \begin{equation*}
    \res 1 b(\res m 1, \res n 1) = \res 1 b \res m 1 \vee \res 1 {b'} \res n 1 = \res m b \vee
    \res n {b'} = \res {b(m,n)} {b \vee b'} = \res{b(m,n)}1
  \end{equation*}
  which gives the first axiom in~\eqref{eq:28}. On the other hand, for
  all $(m,b) \in M \times B$, 
  $(\res m 1)^\ast (\res 1 b) = \r{(\res m 1 \res 1 b)} = \r{\res m {m^\ast b}} = \res 1 {m^\ast b}$
  giving the second axiom in~\eqref{eq:28}.   
\end{proof}

\looseness=-1 We now show that the two processes just described
underlie a functorial equivalence. Let us write $\mathrm{br}\cat{Mon}$
for the category of Boolean restriction monoids and their
homomorphisms.

\begin{Thm}
  \label{thm:6}
  The assignment $S \mapsto S^\downarrow$ of Proposition~\ref{prop:25}
  is the action on objects of an equivalence of categories
  $(\thg)^\downarrow \colon \mathrm{br}\cat{Mon} \rightarrow \dc {\BA}
  {\Mon}$ which on morphisms sends $\varphi \colon S \rightarrow T$ to
  $\dc {\res \varphi {E(S)}} {\res \varphi {\mathrm{Tot}(S)}} \colon
  S^\downarrow \rightarrow T^\downarrow$.
\end{Thm}
Recall here that a functor $F \colon \A \rightarrow \B$ is an
equivalence just when it is both full and faithful, and also
\emph{essentially surjective on objects}, meaning that each $B \in \B$ is
isomorphic to $F(A)$ for some $A \in \A$.
\begin{proof}
  Any homomorphism $\varphi \colon S \rightarrow T$ of Boolean
  restriction monoids, preserves restriction idempotents and total
  maps, and has restriction to $E(S) \rightarrow E(T)$ a Boolean
  homomorphism; moreover, these restrictions easily satisfy the two
  axioms of~\eqref{eq:29}. So $(\thg)^\downarrow$ is well-defined on
  morphisms, is clearly functorial, and is essentially surjective by
  Proposition~\ref{prop:26}. It remains to show it is full and
  faithful.
  Given a Boolean restriction monoid $S$ and $s \in S$, we write
  \begin{equation}\label{eq:35}
    s^- = (\r s)' \qquad \text{and} \qquad \check s = s \vee s^-\rlap{ .}
  \end{equation}
  Clearly $s$ and $s^-$ are disjoint, so that this join exists;
  moreover, $\check s$ is \emph{total} and so $s = \r s \check s$
  expresses $s$ as a product of a restriction idempotent and a total
  element.

  In particular, this implies fidelity of $(\thg)^\downarrow$: for if $\varphi,
  \psi \colon S \rightarrow S'$ act in the same way on
  restriction idempotents and total
  elements, then they act the same on each $s = s^+ \check s$ and so are equal.
  To show fullness, let $S$ and $S'$ be Boolean restriction monoids
  and $\dc{\varphi}f \colon \BM \rightarrow \dc {B'}{M'}$ a
  homomorphism of the associated matched pairs. By~\eqref{eq:29}, this
  is to say that for all $b \in B$ and $m,n \in M$:
  \begin{equation}
    \label{eq:36}
      bm = bn \implies  \varphi(b) f(m) = \varphi(b) f(n) \quad
 \text{and} \quad  \r{(f(m)\varphi(b))} = \varphi(\r{(mb)})\rlap{ .}
  \end{equation}
  We claim that
  $\psi \colon S \rightarrow S'$ defined by
  $\psi(s) = \varphi(\r s)f(\check s)$
  is a homomorphism of Boolean restriction monoids with
  $\psi^\downarrow = \dc \varphi f$. The latter
  claim follows easily since for $b \in E(S)$, we have
  $(b^+, \check b) = (b,1)$ and for $m \in \mathrm{Tot}(S)$ we have
  $(m^+, \check m) = (1,m)$. As for showing
  $\psi$ is indeed a homomorphism of Boolean restriction monoids,
  it is clear that it preserves $1$, and it preserves restriction
  since
  \begin{equation*}
    \r {\psi(s)} = \r{(\varphi(\r s)f(\check s))} =
    \r{\varphi(\r s)}\r{f(\check s)} = \varphi(\r s) = \psi(\r s)\rlap{ .}
  \end{equation*}
  To see that it preserves the monoid operation, we first calculate
  that:
  \begin{equation*}
    \r s \r{(\check s \r t)} = \r s \r{(\check s t)} = \r s \r{(s t)} \vee \r s \r{(s^- t)} = \r{(st)} \vee 0 = \r{(st)}
  \end{equation*}
  using that $(st^+)^+ = (st)^+$; definition of $\check s$ and
  distributivity of joins; and the fact that $\r{(st)} \leqslant \r s$
  and $\r{(s^- t)} \leqslant s^-$. Thus
  \begin{align*}
    \psi(s)\psi(t) &= \varphi(\r s)f(\check s)\varphi(\r t)f(\check t) & \text{definition}\\
    &= \smash{\varphi(\r s)\r{(f(\check s)\varphi(\r t))}f(\check s)f(\check t)} & \text{fourth restriction axiom}\\
    &= \smash{\varphi(\r s)\varphi(\r{(\check s \r t)})f(\check s)f(\check t)}& \text{right equality in \eqref{eq:36}}\\
    &= \smash{\varphi(\r s \r{(\check s \r t)})f(\check s\check t)} & \text{$\varphi, f$ homomorphisms}\\
    &= \smash{\varphi(\r{(st)})f(\check s\check t)} & \text{preceding calculation}\\
    &= \smash{\varphi(\r{(st)})f(\check{st})  = \psi(st)} & \text{left implication in \eqref{eq:36}}
  \end{align*}
  where to apply~\eqref{eq:36} in the last line, we use that
  $s \leqslant \check s$ and $t \leqslant \check t$, whence
  $st \leqslant \check s \check t$ and so
  $\r{(st)} \check s \check t = st = \r{(st)} \check{st}$. Finally,
  since $\psi$ restricts to $\varphi$ on $E(S)$, this restriction is
  a Boolean homomorphism, whence $\psi$ is a homomorphism of
  Boolean restriction monoids as required.
\end{proof}

As explained in the introduction, under the generalised
non-commutative Stone duality of~\cite{Cockett2021Generalising},
Boolean restriction monoids correspond to source-\'etale topological
categories with Stone space of objects. We do not recount the
correspondence in detail here, but simply apply it to describe
explicitly the topological category associated to a matched pair of
algebras.

\begin{Defn}[Classifying topological category]
  \label{def:25}
  Let $\BM$ be a matched pair of algebras. The \emph{classifying
    topological category} $\mathbb{C}_{\BM}$ has:
  \begin{itemize}
  \item \textbf{Space of objects} the Stone space of
    $B$, i.e., the set of all ultrafilters on $B$ under the
    topology with basic (cl)open sets
    $[b] = \{\U \in C_0 : b \in \U\}$ for $b \in B$;
  \item \textbf{Space of arrows} given by the set of all pairs
    $(\U \in C_0, m \in M \quot \equiv_\U$), where 
    $m \equiv_\U n$ just when $m \equiv_b n$ for some $b \in \U$,
    under the topology whose basic open sets
    are $\dc b m = \{(\U, m) \in C_1 : b \in \U\}$ for any
    $b \in B$ and $m \in M$;
  \item The \textbf{source} and \textbf{target} of $(\U, m)$ given
    by $\U$ and $m_! \U \defeq \{b \in B : m^\ast b \in \U\}$;
  \item The \textbf{identity} on $\U$ given by $(\U, 1) \colon \U
    \rightarrow \U$;
  \item The \textbf{composition} of $(\U, m) \colon \U \rightarrow m_!
    \U$ and $(m_! \U, n) \colon m_! \U \rightarrow n_! m_! \U = (mn)_!
    \U$ given by
    $(\U, mn) \colon \U \rightarrow (mn)_! \U$.
  \end{itemize}
When the action of $M$ on $B$ is faithful, we may under Stone duality
identify elements $m \in M$ with continuous endomorphisms of the Stone
space of $B$; whereupon the morphisms of $\mathbb{C}_{\BM}$ from $W$
to $W'$ can equally well be described as germs at $W$ of continuous
functions in $M$ which map $W$ to $W'$.
\end{Defn}

One might expect homomorphisms of matched pairs of algebras to
induce functors between the classifying topological categories, but
this is not so; rather, as in~\cite{Cockett2021Generalising}, they
induce \emph{cofunctors}~\cite{Higgins1993Duality,
  Aguiar1997Internal}, which are equally the \emph{algebraic
  morphisms} of~\cite{Buneci2005Morphisms}.
\begin{Defn}[Cofunctor]
  \label{def:35}
  A \emph{cofunctor} $F \colon \mathbb{C} \rightsquigarrow \mathbb{D}$
  between categories comprises a mapping on objects
  $\mathrm{ob}(\mathbb{D}) \rightarrow \mathrm{ob}(\mathbb{C})$,
  written $d \mapsto Fd$, and a mapping which associates to
  each $d \in \mathrm{ob}(\mathbb{D})$ and arrow
  $f \colon Fd \rightarrow c$ of $\mathbb{C}$ an object $f_* d$ of
  $\mathbb{D}$ with $F(f_* d) = c$ and an arrow
  $F_d(f) \colon d \rightarrow f_* d$, subject to the axioms that
  $F_d(1_{Fd}) = 1_d$ and $F_{f_* d}(g) \circ F_d(f) = F_d(gf)$
  (note that these imply in particular that $(1_{Fd})_* d = d$ and
  $g_* f_* d = (gf)_* d$). If $\mathbb{C}$ and $\mathbb{D}$
  are topological categories, then a \emph{topological cofunctor} is a
  cofunctor for which $d \mapsto Fd$ is continuous
  $\mathrm{ob}(\mathbb{D}) \rightarrow \mathrm{ob}(\mathbb{C})$ and
  $(d, f) \mapsto F_d(f)$ is continuous
  $\mathrm{mor}(\mathbb{C}) \times_{\mathrm{ob}(\mathbb{C})}
  \mathrm{ob}(\mathbb{D}) \rightarrow \mathrm{mor}(\mathbb{D})$.
\end{Defn}

\begin{Defn}[Classifying cofunctor]
  \label{def:36} Let
  $\dc \varphi f \colon \BM \rightarrow \dc {C} {N}$ be a homomorphism
  of matched pairs of algebras. The \emph{classifying cofunctor}
  $\mathbb{C}_{\BM} \rightarrow \mathbb{C}_{\dc C N}$ is given as
  follows:
  \begin{itemize}
  \item \textbf{On objects} it takes 
    $\U \in \mathbb{C}_{\dc C N}$ to 
    $\varphi^\ast(\U) = \{b \in B : \varphi(b) \in \U\} \in 
    \mathbb{C}_{\BM}$;
  \item \textbf{On maps} it takes an object
    $\U \in \mathbb{C}_{\dc C N}$ and map
    $(\varphi^\ast \U, m) \colon \varphi^\ast \U \rightarrow m_!
    \varphi^\ast \U$ in $\mathbb{C}_{\BM}$ to the object $f(m)_! \U$
    and map $(\U, f(m)) \colon \U \rightarrow f(m)_! \U$ in
    $\mathbb{C}_{\dc C N}$. Note this is well-defined by the left-hand
    axiom in~\eqref{eq:53}, and satisfies $\varphi^\ast f(m)_! \U =
    m_! \varphi^\ast \U$ by the right-hand one.
  \end{itemize}
\end{Defn}

Combining Theorem~\ref{thm:6} with
\cite[Theorem~5.17]{Cockett2021Generalising}, we thus see that the
operation which assigns to the variety of $\BM$-sets the topological
category $\mathbb{C}_{\BM}$ induces an equivalence between the
category of non-degenerate finitary cartesian closed varieties and the
category of non-empty ample topological groupoids and cofunctors.

We now describe the infinitary generalisations of the above.

\begin{Defn}[Grothendieck Boolean restriction monoid]
  \label{def:23}
  Let $S$ be a Boolean restriction monoid and $\J$ a zero-dimensional
  topology on $E(S)$. We say that $A \subseteq S$ is \emph{admissible}
  if its elements are pairwise-disjoint, and the set
  $A^+ = \{a^+ : a \in A\}^-$ is contained in a partition in $\J$. We
  say that $\J$ makes $S$ into a \emph{Grothendieck Boolean
    restriction monoid} $S_\J$ if any admissible subset
  $A \subseteq S$ admits a join with respect to $\leqslant$, and
  whenever $A \subseteq S$ is admissible and $s \in S$, 
  $sA = \{sa : a \in A\}$ is also admissible and
  $\bigvee sA = s (\bigvee A)$.
\end{Defn}

\begin{Prop}
  \label{prop:27}
  Let $S$ be a Boolean restriction monoid with $S^\downarrow= \BM$. A
  zero-dimensional topology $\J$ on $B$ makes $S$ a Grothendieck
  Boolean restriction monoid $S_\J$ just when it makes $\BM$ a
  Grothendieck matched pair of algebras~$\BJM$.
\end{Prop}
\begin{proof}
  Suppose first $S_\J$ is a Grothendieck Boolean restriction monoid. We
  begin by proving that $m^\ast \colon \BJ \rightarrow \BJ$ for each
  $m \in M$. Indeed, any $P \in \J$ is admissible as a subset of $S$,
  and so $mP$ is also admissible; which says that
  $\{(mb)^+ : b \in P\}^- = \{m^\ast b : b \in P\}^-$ is in $\J$,
  i.e., $m^\ast \colon \BJ \rightarrow \BJ$ as desired. We now prove
  that $M$ is a $\BJ$-set. Given $P \in \J$ and $x \in M^P$, note that
    the family $A = \{bx_b : b \in P\}$ is admissible; write $z$
    for its join, 
  and observe that for all $b \in P$ we have $bA^- = \{bx_b\}$ since
  $bc = 0$ whenever $b \neq c \in P$. Thus
  $bz = \bigvee bA = bx_b$, i.e., $z \equiv_b x_b$ for all $b \in P$. 
   Moreover, if $z' \in M$
  also satisfied $z' \equiv_b x_b$ for all $b \in P$, i.e., $bz' = bx_b$,
  then necessarily $bx_b \leqslant z'$ for all $b$, whence $z =
  \bigvee A
  \leqslant z'$; but since both $z$ and $z'$ are total, we must have
  $z = z'$ as required.

  Suppose conversely that $\BJM$ is a Grothendieck matched
  pair, and let $A \subseteq S$ be admissible. So the set $A^+ = \{a^+
  : a \in A\}^-$
  is contained in a partition $P \in \J$; thus, since $M$ is a $\BJ$-set, we
  can consider the unique element $z \in M$ such that
  \begin{equation*}
    z \equiv_{a^+} \check a \text{ for $a \in A$} \qquad \text{and} \qquad z \equiv_{b} 1 \text{ for $b \in P \setminus A$\ ,}
  \end{equation*}
  where, as in~\eqref{eq:35} we write $\check a = a \vee (a^+)'$.
  Since $A^+ \subseteq P$, the join $d = \bigvee_{a \in A} a^+$
  exists, and so we have the element $dz \in S$. Now
  $a^+dz = a^+z = a^+ \check a = a$ for all $a \in A$, i.e.,
  $a \leqslant dz$ for all $a \in A$; while if $a \leqslant u$ for all
  $a \in A$, i.e., $a^+ u = a$, then
  $a^+ z = a^+ \check a = a = a^+ u$, i.e., $z \equiv_{a^+} u$ for all
  $a \in A$, whence $z \equiv_d u$ by Proposition~\ref{prop:2}(iii),
  i.e., $dz = du = (dz)^+u$, i.e., $dz \leqslant u$. So $dz$ is the
  join of $A$.

  We now show stability of joins under left multiplication. Given
  $A \subseteq S$ admissible and $s \in S$, we may write $b = s^+$ and
  $m = \check s$ so that $s = bm$. It is easy to see that, if
  $A^+ \subseteq P \in \J$, then
  $\{(sa)^+ : a \in A\}^- \subseteq b \wedge m^\ast P \in \J$, so that
  $sA$ is also admissible. Now necessarily $\bigvee sA \leqslant s
  (\bigvee A)$, and so it suffices to show their restrictions are the
  same. But we have
  \begin{align*}
\textstyle    (s (\bigvee A))^+ &= \textstyle(s (\bigvee A)^+)^+ = (s (\bigvee A^+))^+ = b(m (\bigvee A^+))^+ = b \cdot m^\ast(\bigvee A^+)\\
    &= \textstyle b \cdot \bigvee m^\ast(A^+) = b \cdot \bigvee_{a \in A} (ma)^+  = \bigvee_{a} b(ma)^+  = \bigvee_{a} (sa)^+ = (\bigvee sA)^+
  \end{align*}
  as desired, where in going from the first to the second line we use
  the (easy) fact that any Grothendieck Boolean algebra homomorphism
  preserves admissible joins.
\end{proof}

A \emph{homomorphism} of Grothendieck Boolean restriction monoids
$\varphi \colon S_\J \rightarrow T_\K$ is a Boolean restriction
homomorphism which also preserves admissible families and joins of
admissible families. By a similar argument to before, $\varphi$ is a
Grothendieck Boolean restriction homomorphism if and only if it is a
restriction monoid homomorphism and its action on restriction
idempotents is a Grothendieck Boolean homomorphism
$E(S)_\J \rightarrow E(T)_\K$. Writing $\mathrm{gbr}\cat{Mon}$ for the
category of Grothendieck Boolean restriction monoids and their
homomorphisms, it follows that:
\begin{Thm}
  \label{thm:7}
  The equivalence of categories
  $(\thg)^\downarrow \colon \mathrm{br}\cat{Mon} \rightarrow \dc {\BA}
  {\Mon}$ extends to an equivalence of categories $(\thg)^\downarrow
  \colon \mathrm{gbr}\cat{Mon} \rightarrow \dc {\GBA}
  {\Mon}$ with action on objects $S_\J \mapsto \BJM$.
\end{Thm}

In the infinitary case, the further correspondence with topological
categories breaks down; the reason is that a Grothendieck Boolean
restriction monoid need not satisfy a ``$\J$-closed ideal lemma''
analogous to the Boolean prime ideal lemma. Instead, in the spirit
of~\cite{Resende2006Lectures}, we get a correspondence with certain
\emph{localic categories}: namely, those whose object-space is
strongly zero-dimensional and whose source projection is a local
homeomorphism. Again, we give the construction, which we extract from
the presentation of~\cite[\sec 5.3]{Cockett2021Generalising}, but none
of the further details.

\begin{Defn}[Classifying localic category]
  \label{def:26}
  Let $\BJM$ be a Grothendieck matched pair of algebras. The
  \emph{classifying localic category} $\mathbb{C}$ has:
  \begin{itemize}
  \item \textbf{Locale of objects} $C_0$ given by 
    $\mathrm{Idl}_\J(B)$;
  \item \textbf{Locale of arrows} $C_1$ given by the set of $\BJ$-set
    homomorphisms $M \rightarrow \mathrm{Idl}_\J(B)$ ordered
    pointwise; here $\mathrm{Idl}_\J(B)$ is a $\BJ$-set via
    $I \equiv_b J$ when
    $I \cap \mathop{\downarrow} b = J \cap \mathop{\downarrow} b$;
  \item The \textbf{source} map $s \colon C_1
    \rightarrow C_0$ is given by $s^\ast(I) = \lambda m.\, I$;
  \item The \textbf{target} map $t \colon C_1
    \rightarrow C_0$ is given by $t^\ast(I) = \lambda m.\, m^\ast I$, where $m^\ast I$ is the
    $\J$-closed ideal generated by the elements $m^\ast b$ for $b \in I$;
  \item The \textbf{identity} map $i \colon C_0 \rightarrow C_1$ is
    given by $i^\ast(f) = f(1)$.
  \item The \textbf{composition} map $m \colon C_1 \times_{C_0} C_1
    \rightarrow C_1$ is given by $m^\ast(f) = \lambda m, n.\, f(mn)$.
    Here, we identify $C_1 \times_{C_0} C_1$ with the locale of all
    functions $f \colon M \times M \rightarrow \mathrm{Idl}_\J(B)$ for
    which each $f(\thg, n)$ is a $\BJ$-set homomorphism $M \rightarrow
    \mathrm{Idl}_\J(B)$ and each $f(m, \thg)$ is a $\BJ$-set homomorphism $M \rightarrow
    m^\ast\mathrm{Idl}_\J(B)$.
  \end{itemize}
\end{Defn}
Like before, we can also associate a localic cofunctor to each
homomorphism of Grothendieck matched pairs of algebras, and in this
way obtain an equivalence between the category of non-degenerate
cartesian closed varieties, and the category of non-empty ample
localic categories and cofunctors.

\section{When is a variety a topos?}
\label{sec:when-variety-topos}

In this section, we prove the second main result of the paper, which
gives a \emph{syntactic} characterisation of when a given cartesian
closed variety is a topos, and shows that this condition can be
re-expressed in terms of the \emph{minimality} of the classifying
topological or localic category.
Recall that a \emph{topos} is a cartesian closed category $\C$ which
has all pullbacks and a \emph{subobject classifier}: that is, an
object $\Omega$ endowed with a map
$\top \colon 1 \rightarrowtail \Omega$ with the property that, for any
monomorphism $m \colon Y \rightarrowtail X$ in $\C$ there is a unique
``classifying map'' $\chi_m \colon X \rightarrow \Omega$ for which the
following square is a pullback:
\begin{equation}\label{eq:37}
  \cd{
    {Y} \ar[r]^-{!} \ar[d]_{m} &
    {1} \ar[d]^{\top} \\
    {X} \ar[r]_-{\chi_m} &
    {\Omega}\rlap{ .}
  }
\end{equation}
As explained in the introduction, the question posed in the title of
this section was answered by Johnstone in~\cite{Johnstone1985When},
yielding a slightly delicate syntactic characterisation theorem
(Theorem~3.1 of \emph{op.~cit.}). Of course, a non-degenerate variety
which is a topos is in particular cartesian closed, and so, as we know
now, must be a variety of $\BJM$-sets. It is therefore natural to ask
whether Johnstone's conditions in~\cite{Johnstone1985When} can be
transformed in light of this knowledge into a condition on a
Grothendieck matched pair $\BJM$ which ensures that
$\BJM\text-\cat{Set}$ not just cartesian closed, but a topos. The
answer is yes: we will show $\BJM\text-\cat{Set}$ is a topos
precisely when:
\begin{equation}\label{eq:38}
  \text{For all } b \in B \setminus \{0\} \text{, there exists } m \in M \text{ with } m^\ast b = 1\rlap{ .}
\end{equation}

While it would be possible to prove this result directly, it is
scarcely any extra effort to do something more general.
In~\cite{Johnstone1990Collapsed}, Johnstone shows that any
non-degenerate cartesian closed variety $\V$ has an associated topos
$\E$, which is uniquely characterised by the fact that $\V$ can be
re-found as its two-valued collapse. This implies that a
non-degenerate cartesian closed variety $\V$ is itself a topos just
when its associated topos $\E$ is two-valued, i.e., equal to its
two-valued collapse. Here, the notion of ``two-valued collapse'' is
given by:
\begin{Defn}[Two-valued collapse]
  \label{def:27}
  Let $\E$ be a cartesian closed category. The \emph{two-valued
    collapse} $\E_{\mathrm{tv}}$ is the full subcategory of $\E$ whose
  objects $X$ are either \emph{well-supported}---meaning that the unique
  map $X \rightarrow 1$ is epimorphic---or initial.
\end{Defn}

For a given cartesian closed variety $\V$, finding the topos which
collapses to it is done by Theorem~6.1 of \emph{op.~cit.}, which is
quite delicate; but armed with the knowledge that
$\V \cong \BJM\text-\cat{Set}$, we are able to give a simpler
construction of the associated topos\footnote{We should clarify that
  we do not recover the full force of
  \cite[Theorem~6.1]{Johnstone1990Collapsed}, which can reconstruct a
  topos from a more general cartesian closed category than a cartesian
  closed variety.}, from which the characterisation~\eqref{eq:38}
above will follow straightforwardly.

\begin{Defn}[Category of $\BJM$-sheaves]
  \label{def:28}
  Let $\BJM$ be a Grothendieck matched pair of algebras. A
  \emph{$\BJM$-presheaf} $X$ comprises sets $X(b)$ for all
  $b \in B \setminus \{0\}$, together with:
  \begin{itemize}
  \item For all $c \in B$ and $m \in M$ with $m^\ast c \neq 0$, a
    function $m \cdot (\thg) \colon X(c) \rightarrow X(m^\ast c)$;
  \item For all $b,c \in B$ with $b \wedge c \neq 0$, a
    function $b \wedge (\thg) \colon X(c) \rightarrow X(b \wedge c)$;
  \end{itemize}
  such that for all $x \in X(c)$ and all
  suitable $a,b \in B$ and $m,n \in M$ we have:
  \begin{enumerate}[(i)]
  \item $c \wedge x = x$ and $(a \wedge b) \wedge x = a \wedge (b
    \wedge x)$;
  \item $1 \cdot x = x$ and $(mn) \cdot x = m \cdot (n \cdot x)$;
  \item $m \cdot (b \wedge x) = (m^\ast b) \wedge (m \cdot x)$; and
  \item If $m \equiv_b n$ then
    $b \wedge (m \cdot x) = b \wedge (n \cdot x)$.
  \end{enumerate}
  
  Such a presheaf is a \emph{$\BJM$-sheaf} if for each $P \in \J_c$
  and family $x \in \prod_{b \in P}X(b)$, there is given an
  element $P(x) \in X(c)$, and these elements satisfy:
  \begin{equation}
    \label{eq:51}
    b \wedge P(x) = x_b \text{ for all } x \in \textstyle\prod_{b \in P}X(b) \ \text{ and } \  P(\lambda b.\, b \wedge x) = x \text{ for all
    } x \in X(c)\text{ .}
  \end{equation}

  A \emph{homomorphism} of $\BJM$-presheaves is a family of functions
  ${f_c \colon X(c) \rightarrow Y(c)}$ that preserve each
  $m \cdot (\thg)$ and $b \wedge (\thg)$; between sheaves, such an $f$
  will necessarily also preserve each $P(\thg)$. We write
  $\BJM\text-\cat{Shv}$ for the category of $\BJM$-sheaves.
\end{Defn}

\begin{Prop}
  \label{prop:43}
  For any Grothendieck matched pair of algebras $\BJM$, the category
  $\BJM\text-\cat{Shv}$ is both a many-sorted variety and a
  (Grothendieck) topos.
\end{Prop}
\begin{proof}
  The only axiom for a $\BJM$-sheaf which is not obviously equational
  is the condition that if $m \equiv_b n$ then
  $b \wedge (m \cdot x) = b \wedge (n \cdot x)$; however, this can be
  re-expressed as the condition that
  $b \wedge (m \cdot x) = b \wedge (b(m,n) \cdot x)$ for all
  $m,n \in M$, $b \in B$ and $x \in X(c)$. Thus $\BJM\text-\cat{Shv}$
  is a many-sorted variety. To see that it is a Grothendieck topos, it
  suffices to exhibit it as equivalent to the category of sheaves on a
  suitable \emph{site}~\cite[\sec C2]{Johnstone2002Sketches2}. So
  consider the category $\C$ in which:
  \begin{itemize}
  \item Objects are elements of $B \setminus \{0\}$;
  \item Morphisms $b \rightarrow c$ are elements
    $m \in M \quot \equiv_b$ for which $b \leqslant m^\ast c$; 
    this is well-posed, as if $m \equiv_b n$ then
    $b \wedge m^\ast c = b \wedge n^\ast c$, so
    $b \leqslant m^\ast c$ if and only if $b \leqslant n^\ast c$;
  \item The identity on $b$ is $1 \colon b \rightarrow b$;
  \item The composition of $m \colon b \rightarrow c$ and
    $n \colon c \rightarrow d$ is $mn \colon b \rightarrow d$. This is
    well-posed, as if $m \equiv_{b} m'$ and $n \equiv_c n'$ then
    $mn \equiv_b m'n \equiv_b m'n'$, using 
    $b \leqslant (m')^\ast(c)$ for the second equality; and clearly
    $b \leqslant m^\ast c$ and $c \leqslant n^\ast d$ imply
    $b \leqslant (mn)^\ast d$.
  \end{itemize}
  Given a family of sets $X(b)$, the $\BJM$-presheaf structures
  thereon are now in bijection with the $\C$-presheaf structures;
  indeed, from the former we obtain the latter by defining
  $X(m \colon b \rightarrow c)$ as $b \wedge (m \cdot \thg)$, while
  from the latter we obtain the former by defining $m \cdot (\thg)$
  and $c \wedge (\thg)$ as $X(m \colon m^\ast b \rightarrow b)$ and
  $X(1 \colon c \wedge b \rightarrow b)$. Under this correspondence,
  axioms (i)--(iii) correspond to functoriality in $\C$, while axiom
  (iv) corresponds to the equivalence relation on the homs of $\C$.

  Now consider the Grothendieck coverage $J$ on the category $\C$ for
  which the covers of $c \in \C$ are the families
  $(1 \colon b \rightarrow c)_{b \in P}$ for each $P \in \J_c$. This
  is indeed a coverage: for given the above cover of $c$ and a map
  $m \colon d \rightarrow c$ in $\C$, since $m^\ast$ is a Grothendieck
  Boolean algebra homomorphism we have $m^\ast P \in \J_{m^\ast c}$
  and so by axiom (i) for a Grothendieck Boolean algebra that
  $d \wedge m^\ast P = \{d \wedge m^\ast b : b \in P\}^-$ is in
  $\J_{d}$; and for each $1 \colon d \wedge m^\ast b \rightarrow d$ in
  the corresponding cover, the composite
  $m \colon d \wedge m^\ast b \rightarrow c$ factors through
  $1 \colon b \rightarrow c$ via
  $m \colon d \wedge m^\ast b \rightarrow b$.

  Now given a $\C$-presheaf $X$, a matching family for the cover
  $(1 \colon b \rightarrow c)_{b \in P}$ is, by disjointness of $P$,
  simply a family $x \in \prod_{b \in P} X(b)$, and the sheaf axiom
  for this cover asserts that there is a unique $P(x) \in X(c)$ whose
  image under $X(1 \colon b \rightarrow c)$ is $x_b$ for all
  $b \in P$. But in terms of the corresponding $\BJM$-presheaf, this
  asserts exactly the existence of elements $P(x)$
  satisfying~\eqref{eq:51}. So $(\C, J)$-sheaves correspond
  bijectively with $\BJM$-sheaves; since clearly the homomorphisms
  match up under this correspondence,
  $\BJM\text-\cat{Shv} \cong \cat{Sh}(\C, J)$ is a Grothendieck topos.
\end{proof}

We will now show that, if $\BJM$ is a Grothendieck matched pair of
algebras, then the topos $\BJM\text-\cat{Shv}$ has
$\BJM\text-\cat{Set}$ as its two-valued collapse. The key point is how
we embed $\BJM\text-\cat{Set}$ into $\BJM\text-\cat{Shv}$. To motivate
this, note that what a $\BJM$-set lacks relative to a $\BJM$-sheaf are
the actions $b \wedge (\thg)$, so it makes sense to adjoin these
``formally''. To this end, if $X$ is a $\BJM$-set, let us suggestively
write elements of the quotient $X \quot \equiv_b$ as $b \wedge x$; so
$b \wedge x = b \wedge y$ just when $x \equiv_b y$. Using this
notation, we now have:
\begin{Prop}
  \label{prop:28}
  For any $\BJM$-set $X$, there is a $\BJM$-sheaf $B \wedge X$ with
  \begin{equation*}
    (B \wedge X)(b)\, = \,\mathord{X \quot \mathord{ \equiv_b}}\, =\, \{ b \wedge x : x \in X\}\rlap{ ,}
  \end{equation*}
  and operations $b \wedge (\thg) \colon X(c) \rightarrow X(b \wedge c)$
  and $m \cdot (\thg) \colon X(c) \rightarrow X(m^\ast c)$ given by
  \begin{equation*}
    b \wedge (c \wedge x) = (b \wedge c) \wedge x \qquad  \text{and} \qquad
    m \cdot (c \wedge x) = (m^\ast c) \wedge (m \cdot x)\rlap{ .}
  \end{equation*}
\end{Prop}
\begin{proof}
  The $\BJM$-presheaf operations are well-defined by
  Proposition~\ref{prop:2}(i) and the second $\BJ$-set axiom
  in~\eqref{eq:30}; they trivially satisfy axiom (i) for a
  $\BJM$-presheaf and satisfy axioms (ii) and (iii) since $M$ acts on
  $B$ via Boolean homomorphisms. As for axiom (iv), if $m \equiv_b n$
  then
  $b \wedge (m \cdot (c \wedge x)) = (b \wedge m^\ast c) \wedge (m
  \cdot x) = (b \wedge n^\ast c) \wedge (n \cdot x) = b \wedge (m
  \cdot (c \wedge x))$ where the first and last equalities just unfold
  definitions, and the middle equality follows from $m \equiv_b n$,
  since this condition implies that
  $b \wedge m^\ast c = b \wedge n^\ast c$ and
  $m \cdot x \equiv_b n \cdot x$.

  It remains to show $B \wedge X$ is in fact a \emph{sheaf}. If $X$ is
  empty then this is trivial; otherwise, choose an arbitrary element
  $u \in X$ and now for any $P \in \J_c$ and family
  $x \in \prod_{b \in P}X(b)$, define $P(x) = c \wedge z$, where
  $z \in X$ is unique such that
  \begin{equation*}
    z \equiv_b x_b \text{ for all $b \in P$} \quad \text{and} \quad z \equiv_{c'} u\rlap{ .}
  \end{equation*}
  Now $b \wedge P(x) = b \wedge z = b \wedge x_b$ for each $c \in P$,
  giving the first axiom in~\eqref{eq:51}; furthermore, for any
  $x \in X(c)$ we have $P(\lambda b.\, b \wedge x) = c \wedge z$ where
  $z$ is unique such that $z \equiv_b x$ for all $b \in P$ and
  $z \equiv_{c'} u$. By Proposition~\ref{prop:2}(iii) we conclude that
  $z \equiv_c x$, i.e., $P(\lambda b.\, b \wedge x) = x$, which is the
  second axiom of~\eqref{eq:51}.
\end{proof}
\begin{Prop}
  \label{prop:29}
  Let $\BJM$ be a non-degenerate Grothendieck matched pair. The
  assignment $X \mapsto B \wedge X$ is the action on objects of a full
  and faithful functor
  \begin{equation}
    \label{eq:52}
    B \wedge (\thg) \colon {\BJM}\text-\cat{Set} \rightarrow
    \BJM\text-\cat{Shv}
  \end{equation}
  which exhibits ${\BJM}\text-\cat{Set}$
  as equivalent to the two-valued collapse of $\BJM\text-\cat{Shv}$.
\end{Prop}
\begin{proof}
  Each $\BJM$-set homomorphism $f \colon X \rightarrow Y$ induces a
  $\BJM$-sheaf homomorphism
  $B \wedge f \colon B \wedge X \rightarrow B \wedge Y$ which sends
  $b \wedge x$ to $b \wedge f(x)$; this is well-defined since
  $x \equiv_b y$ implies $f(x) \equiv_b f(y)$, clearly preserves the
  $B$-actions, and preserves the $M$-actions because $f$ does so.
  Functoriality is obvious, and so we have a functor~\eqref{eq:52},
  which is faithful since we can recover $f$ from $B \wedge f$ via its
  action on \emph{total} elements, i.e., those in
  $(B \wedge X)(1) = X$. For fullness, suppose
  $g \colon B \wedge X \rightarrow B \wedge Y$ is a homomorphism, with
  action $f \colon X \rightarrow Y$ on total elements. Since
  $g(b \wedge x) = g(b \wedge (1 \wedge x)) = b \wedge g(1 \wedge x) =
  b \wedge f(x)$, we will have $g = B \wedge f$ so long as $f$ is a
  $\BJM$-set homomorphism. It clearly preserves $M$-actions; while if
  $x \equiv_b y$ in $X$ then $b \wedge x = b \wedge y$, so
  $b \wedge f(x) = b \wedge f(y)$, i.e., $f(x) \equiv_b f(y)$ as
  required.

  To complete the proof, it remains to show that a $\BJM$-sheaf is in
  the essential image of~\eqref{eq:52} just when it is either empty or
  well-supported. Since the terminal object of $\BJM\text-\cat{Shv}$
  has $1(b) = 1$ for all $b \in B \setminus \{0\}$, a sheaf $Y$ is
  well-supported just when each $Y(b)$ is non-empty which by virtue of
  the $B$-action happens just when $Y(1)$ is non-empty. Clearly, then,
  each $B \wedge X$ is either empty or well-supported according as $X$
  is empty or non-empty.

  Suppose conversely that $Y \in \BJM\text-\cat{Shv}$ has
  $Y(1) \neq \emptyset$. Note that this implies that each
  $b \wedge (\thg) \colon Y(1) \rightarrow Y(b)$ is \emph{surjective}.
  For indeed, let us choose some $u \in Y(1)$; then for any
  $y \in Y(b)$, the sheaf condition gives a unique 
  $z \in Y(1)$ with $b \wedge z = y$ and
  $b' \wedge z = b' \wedge u$---so, in particular, $y$ is in the
  image of $b \wedge (\thg)$.

  We now show that $X = Y(1)$ is a $\BJM$-set and that
  $B \wedge X \cong Y$. Clearly $X$ is an $M$-set via the operations
  $m \cdot (\thg)$ of $Y$; as for the $\BJ$-set structure, define
  $x \equiv_b y$ just when $b \wedge x = b \wedge y \in Y(b)$ (and
  $x \equiv_0 y$ always). Easily the $\equiv_b$'s are equivalence
  relations satisfying axiom (i) of Proposition~\ref{prop:2}; however,
  they also satisfy axiom (ii)$'$ therein. Indeed, for any $P \in \J$
  and $x \in X^P$, we have the element
  $z = P(\lambda b.\, b \wedge x_b) \in X$ which by the left equation
  of~\eqref{eq:51} satisfies $b \wedge z = b \wedge x_b$, i.e.,
  $z \equiv_b x_b$, for all $b \in P$. But if $z' \in X$ also
  satisfied $z' \equiv_b x_b$ for all $b \in P$, then we would have
  $z' = P(\lambda b.\, b \wedge z') = P(\lambda b.\, b \wedge x_b) =
  z$ by the right equation of~\eqref{eq:51}; so $z$ is unique such
  that $z \equiv_b x_b$ for all $b \in P$, as required. This proves
  that $X = Y(1)$ is a $\BJ$-set, and it remains to check the
  $\BJM$-set axioms~\eqref{eq:53}. But if $m \equiv_b n$ and $x \in X$
  then $b \wedge (m \cdot x) = b \wedge (n \cdot x)$ in $Y(b)$ by
  axiom (iv) for a $\BJM$-presheaf, i.e.,
  $m \cdot x \equiv_b n \cdot x$; while if $x \equiv_b y$ in $X$,
  i.e., $b \wedge x = b \wedge y$ in $Y(b)$, then
  $m^\ast b \wedge m \cdot x = m \cdot (b \wedge x) = m \cdot (b
  \wedge y) = m^\ast b \wedge m \cdot y$, i.e.,
  $m \cdot x \equiv_{m^\ast b} m \cdot y$.

  So $X$ is a $\BJM$-set. Now, since $x \equiv_b y$ in $X=Y(1)$ just
  when $b \wedge x = b \wedge y$ in $Y(b)$, we can identify
  $(B \wedge X)(b) = X \quot \equiv_b$ with the image of the map
  $b \wedge (\thg) \colon Y(1) \rightarrow Y(b)$. But, as noted above,
  this map is surjective, and so we have isomorphisms $(B \wedge X)(b)
  \cong Y(b)$ for each $b \in B \setminus \{0\}$. It is not hard to
  see that the presheaf structures match under these
  isomorphisms, so $B \wedge X \cong Y$ as desired.
\end{proof}

We can now give our promised characterisations of when
$\BJM\text-\cat{Set}$ is a topos. As mentioned above, one form of our
characterisation will involve a condition of \emph{minimality} on the
classifying category; the relevant notion here is the following one,
which extends the standard terminology for topological groupoids (for
which a sieve is typically called an ``invariant subset'').

\begin{Defn}[Minimality]
  \label{def:29}
  An \emph{open sieve} on a topological category $\mathbb{C}$ is an
  open subset of $\mathrm{ob}(\mathbb{C})$ which contains the source
  $s(f)$ of any arrow of $\mathbb{C}$ whenever it contains its target
  $t(f)$. Correspondingly, an \emph{open sieve} on a localic category
  $\mathbb{C}$ is an element $u \in \mathrm{ob}(\mathbb{C})$ such that
  $t^\ast(u) \leqslant s^\ast(u)$ in $C_1$. A topological (resp.,
  localic) category is \emph{minimal} if its only open sieves are
  $\emptyset$ and $\mathrm{ob}(\mathbb{C})$ (resp., $0$ and $1$).
\end{Defn}

\begin{Thm}
  \label{thm:8}
  Let $\BJM$ be a Grothendieck matched pair of algebras. The
  following are equivalent:
  \begin{enumerate}[(i)]
  \item For all $b \in B \setminus \{0\}$, there exists $m \in M$ with
    $m^\ast b = 1$;
  \item The topos $\BJM\text-\cat{Shv}$ is two-valued;
  \item $B \wedge (\thg) \colon \BJM\text-\cat{Set} \rightarrow
    \BJM\text-\cat{Shv}$ is an equivalence of categories;
  \item $\BJM\text-\cat{Set}$ is a topos;
  \item The classifying (topological or localic) category of $\BJM$ is minimal.
  \end{enumerate}
\end{Thm}
\begin{proof}
  We first show (i) $\Rightarrow$ (ii). $\BJM\text-\cat{Shv}$ is
  two-valued if any subobject $U$ of the terminal sheaf $1$ is either
  empty or equal to $1$. But if any $U(b)$ is non-empty then on
  choosing $m$ as in (i), we see that $U(1)$ is also non-empty: so $U$
  is well-supported and so must equal $1$. Now (ii) $\Rightarrow$
  (iii) follows since $B \wedge (\thg)$ exhibits $\BJM\text-\cat{Set}$
  as equivalent to the two-valued collapse of $\BJM\text-\cat{Shv}$,
  and (iii) $\Rightarrow$ (iv) is trivial as $\BJM\text-\cat{Shv}$ is
  a topos. We now prove (iv) $\Rightarrow$ (i). Given
  $b \in B \setminus \{0\}$, consider the following diagram in
  $\BJM\text-\cat{Set}$, where $\varphi \colon M \rightarrow B$ is
  the homomorphism $m \mapsto m^\ast b$, the bottom maps pick
  out $0,1 \in B$, and both squares are pullbacks:
  \begin{equation*}
    \cd{
      \varphi^{-1}(0) \ar[d]_{!} \ar@{ >->}[r]^-{} &
      M \ar[d]_-{\varphi} &
      \varphi^{-1}(1) \ar@{ >->}[l] \ar[d]^-{!} \\
      1 \ar[r]^-{0} &
      B &
      1 \ar[l]_-{1} \rlap{ .}
    }
  \end{equation*}
  The two pullback objects are given by
  \begin{equation*}
    \varphi^{-1}(0) = \{ m \in M : m^\ast b = 0\} \qquad
    \text{and} \qquad
    \varphi^{-1}(1) = \{ m \in M : m^\ast b = 1\}
  \end{equation*}
  and so to prove (i) we must show $\varphi^{-1}(1)$ is non-empty. The
  maps on the bottom row are jointly epimorphic, since $0,1$ generate
  $B$ as a $\BJ$-set; thus, as jointly epimorphic families are
  pullback-stable in a topos, the maps on the top row must also be
  jointly epimorphic. So if $\varphi^{-1}(1)$ were empty,
  $\varphi^{-1}(0) \rightarrowtail M$ would be an epimorphic
  monomorphism in a topos, and hence invertible. But then
  $1 \in \varphi^{-1}(0)$, i.e., $b = 1^\ast b = 0$, contradicting
  $b \in B \setminus \{0\}$. So $\varphi^{-1}(1)$ is non-empty as
  required.

  To complete the proof, we show that (i) is equivalent to (v). It
  suffices to consider the \emph{localic} classifying category, since
  in the finitary case, the classifying localic category is spatial,
  and the minimality of the localic category and the corresponding
  topological category come to the same thing. We first prove the
  following claim: given $b \neq 0 \in B$, the $\J$-closed ideal
  $M^\ast b \subseteq B$ generated by the elements
  $\{m^\ast b : m \in M\}$ is all of $B$ if and only if there exists
  $m \in M$ with $m^\ast b = 1$. Since $M^\ast b=B$ just when
  $1 \in M^\ast b$, the ``if'' direction is trivial. For the converse,
  to say $1 \in M^\ast b$ is to say that there exists
  $\{c_i : i \in I\} \in \J$ and $(n_i \in M : i \in I)$ such that
  $c_i \leqslant n_i^\ast(b)$ for each $i \in I$. Taking $m \in M$
  unique such that $m \equiv_{c_i} n_i$ for each $i$, we have
  $m^\ast b = \textstyle\bigvee_{i} c_i \wedge n_i^\ast(b) = \bigvee_i
  c_i = 1$ as desired.

  We now prove (i) $\Leftrightarrow$ (v). An open sieve of the
  classifying localic groupoid $\mathbb{C}_{\BJM}$ is, by definition,
  an ideal $I \in \mathrm{Idl}_\J(B)$ such that
  $t^\ast(I) \leqslant s^\ast(I) \colon M \rightarrow
  \mathrm{Idl}_\J(B)$, i.e., such that $m^\ast I \subseteq I$ for all
  $m \in M$. Clearly, any ideal of the form $M^\ast b$ is an open
  sieve; conversely, if $I$ is an open sieve and $b \in B$ then
  $M^\ast b \subseteq I$, so that we can write $I$ as a union of open
  sieves $I = \bigcup_{b \in I} M^\ast b$. By these observations, to
  ask that the only open sieves of $\mathbb{C}$ are $\{0\}$ and $B$ is
  equally well to ask that every sieve of the form $M^\ast b$ is
  either $\{0\}$ or $B$. Of course, $M^\ast b = \{0\}$ only when
  $b = 0$, and so $\mathbb{C}$ is minimal just when $M^\ast b = B$ for
  all $b \neq 0$; which, by the claim proved above, is to say that for
  all $b \neq 0$ there exists $m \in M$ with $m^\ast b = 1$.
\end{proof}

\section{The groupoidal case}
\label{sec:groupoid-case}

In this section, we describe semantic and syntactic conditions on a
cartesian closed variety which are equivalent to its classifying
topological or localic category being a \emph{groupoid}. To motivate
this, we consider the category of left $M$-sets for a monoid $M$; this
is a cartesian closed variety whose classifying topological category is
$M$ itself, seen as a one-object discrete topological category, and
clearly this is a groupoid just when $M$ is a \emph{group}.

This syntactic condition can be recast in terms of the cartesian
closed structure of the category of $M$-sets. In general, this is
given by the usual formula for internal homs in a presheaf category,
so that $Z^Y$ is the set of $M$-set maps $M \times Y \rightarrow Z$,
with the $M$-set structure $(m \cdot f)(n,y) = f(nm,y)$.
However, when $M$ is a group, we have an alternative, simpler
presentation; we may take $Z^Y = \cat{Set}(Y,Z)$ with the $M$-set
structure given by conjugation:
\begin{equation}
  \label{eq:57}
  (m \cdot f)(y) = m \cdot f(m^{-1} \cdot y)\rlap{ .}
\end{equation}
Thus, when $M$ is a group, the function-spaces in $M\text-\cat{Set}$ are
lifts of the function-spaces of $\cat{Set}$. A more precise way of
saying this is that the forgetful functor
$U \colon M\text-\cat{Set} \rightarrow \cat{Set}$ is 
\emph{cartesian closed}:

\begin{Defn}
  \label{def:30}
  Let $\C$ and $\D$ be cartesian closed categories. A
  finite-product-preserving functor $U \colon \C \rightarrow \D$ is
  \emph{cartesian closed} if, for all $Y,Z \in \C$, the map
  $U(Z^Y) \rightarrow UZ^{UY}$ in $\D$ found as the transpose of the
  following map is invertible:
  \begin{equation*}
    U(Z^Y) \times UY \xrightarrow{\cong} U(Z^Y \times Y) \xrightarrow{U(\mathrm{ev})} UY\rlap{ .}
  \end{equation*}
\end{Defn}

It is therefore natural to conjecture that, for a general
(Grothendieck) matched pair $\BJM$, the classifying topological or
localic category $\mathbb{C}_{\BJM}$ should be a groupoid precisely
when the internal homs in $\BJM\text-\cat{Set}$ are computed as in
$\BJ\text-\cat{Set}$; that is, just when the forgetful functor
$U \colon \BJM\text-\cat{Set} \rightarrow \BJ\text-\cat{Set}$ is
cartesian closed. The main result of this section will show that this
is the case. Before stating it, we need to say what it means for a
Grothendieck Boolean restriction monoid $S_\J$ to be ``generated by
partial isomorphisms'':

\begin{Defn}[Partial isomorphism, \'etale Grothendieck Boolean
  restriction monoid]
  An element $s$ of a Grothendieck Boolean restriction monoid $S_\J$
  is a \emph{partial isomorphism} if there exists
  a---necessarily unique---$t \in S$ with $st = s^+$ and
  $ts = t^+$. We call $S_\J$ is \emph{\'etale} if each
  $s \in S_\J$ is an admissible join of partial isomorphisms.
\end{Defn}

\begin{Thm}
  \label{thm:9}
  Let $\BJM$ be a Grothendieck matched pair of algebras. The
  following are equivalent:
  \begin{enumerate}[(i)]
  \item The forgetful functor $U \colon \BJM\text-\cat{Set}
    \rightarrow \BJ\text-\cat{Set}$ is cartesian closed;
  \item  The following condition holds:
    \begin{equation}
      \label{eq:41}
      \begin{gathered}
        \text{For all } m \in M \text{, there exists } \{b_i : i \in I\} \in
        \J \text{ and families } (n_i \in M : i \in I) \\[-3pt] \text{ and 
        } (c_i \in B : i \in I) \text{ with } b_i \leqslant m^\ast c_i \text{,
        } mn_i \equiv_{b_i} 1 \text{ and } n_i m \equiv_{c_i} 1 \text{ for all }i\text{.} 
      \end{gathered}
    \end{equation}
  \item The associated Grothendieck Boolean restriction monoid $S_\J$ is \'etale;
  \item The classifying (topological or localic) category of $\BJM$ is
    a groupoid.
\end{enumerate}
\end{Thm}

\begin{Rk}
  A Grothendieck topos is called an \emph{\'etendue} when it is
  equivalent to the category of equivariant sheaves on an \'etale
  localic groupoid, and it is natural to ask for which $\BJM$ the
  topos $\BJM\text-\cat{Shv}$ is an \'etendue. Since
  $\BJM\text-\cat{Shv}$ can be presented as the topos of equivariant
  sheaves on the associated localic or topological category, we see
  that for any $\BJM$ to which Theorem~\ref{thm:9} applies, the topos
  $\BJM\text-\cat{Shv}$ will be an \'etendue. However, this sufficient
  condition is \emph{not} necessary; for example, the topos
  $\mathbb{N}\text-\cat{Set}$ is an \'etendue, but does not satisfy
  Theorem~\ref{thm:9}. We leave it to further work to characterise
  \emph{exactly} which matched pairs $\BJM$ give rise to  \'etendue.
\end{Rk}

Leaving aside the equivalence of (i) and (ii), we can dispatch the
remaining parts of the proof of Theorem~\ref{thm:9} rather quickly:
\begin{proof}
  (iii) $\Leftrightarrow$ (iv) is a consequence
  of~\cite[Theorem~6.3]{Cockett2021Generalising}. To see (ii)
  $\Leftrightarrow$ (iii), note first that in~\eqref{eq:41}, on
  replacing each $c_i$ by $c_i \wedge n_i^\ast b_i$ we may without
  loss of generality assume that we also have
  $c_i \leqslant n_i^\ast b_i$ for each $i$. Considering now (iii), if
  $s \leqslant t \in S_{\J}$ and $t$ is a partial isomorphism, then so
  is $s$; whence $S_{\J}$ will be \'etale as soon as every
  \emph{total} element $\res m 1$ is an admissible join of partial
  isomorphisms. This is equally to say that, for each $m \in M$, there
  is some $\{b_i : i \in I\} \in \J$ for which each $\res m {b_i}$ has
  a partial inverse $\res {n_i} {c_i}$, i.e.,
  $\res m {b_i} \res {n_i} {c_i} = \res 1 {b_i}$ and
  $\res {n_i} {c_i} \res m {b_i} = \res 1 {c_i}$. This says that:
  \begin{equation*}
    b_i \leqslant m^\ast c_i \qquad  mn_i \equiv_{b_i} 1 \qquad c_i \leqslant n_i^\ast b_i \qquad \text{and} \qquad 
    n_i m \equiv_{c_i} 1
  \end{equation*}
  for each $i$, which are precisely the conditions of~\eqref{eq:41}
  augmented by the additional inequalities
  $c_i \leqslant n_i^\ast b_i$ which we justified above.
\end{proof}

This leaves only the proof (i) $\Leftrightarrow$ (ii); this will rest
on the fact, explained
in~\cite[Proposition~1.5.8]{Johnstone2002Sketches}, that an adjunction
$U \colon \D \leftrightarrows \C \colon F$ between cartesian closed
categories has $U$ cartesian closed just when the canonical
(``Frobenius'') maps $F(B \times UA) \rightarrow FB \times A$ are
invertible. To exploit this, we must to describe the functor
$M \otimes_B (\thg) \colon \BJ\text-\cat{Set} \rightarrow
\BJM\text-\cat{Set}$ which is left adjoint to
$U \colon \BJM\text-\cat{Set} \rightarrow \BJ\text-\cat{Set}$.

As a first approximation, we could try taking
$M \otimes_B X = M \times X$ with the free $M$-action
$m \cdot (n, x) = (mn, x)$. Of course this is an $M$-set; but how
would we define $\BJ$-set structure? Well, since the unit map
$X \rightarrow M \times X$ sending $x \mapsto (1,x)$ should be a
$\BJ$-set homomorphism, $x_1 \equiv_b x_2$ should imply
$(1, x_1) \equiv_b (1,x_1)$; but also, since
$m \cdot (1,x_i) = (m,x_i)$, that
$(m,x_1) \equiv_{m^\ast b} (m, x_2)$. Since, as in Remark~\ref{rk:7},
the set $\sheq {(m, x_1)} {(m, x_2)}$ should be a $\J$-closed ideal of
$B$, this suggests taking it to be the closed ideal generated by the
elements $m^\ast b$ where $x_1 \equiv_b x_2$, as follows:

\begin{Defn}
  \label{def:31}
  Let $\BJM$ be a Grothendieck matched pair of algebras. For any
  $m \in M$, any $\BJ$-set $X$, and any $x,y \in X$, write
  $m^\ast \sheq x y \subseteq B$ for the $\J$-closed ideal generated
  by $\{m^\ast b : x \equiv_b y\}$, and write $x \equiv^m_b
  y$ to mean that $b \in m^\ast \sheq x y$.
\end{Defn}
\begin{Rk}
  \label{rk:6}
  By axiom (i) for a zero-dimensional topology, the $\J$-closed ideal
  generated by a set $S \subseteq B$ is composed of all $b \in B$ such
  that $P \subseteq \mathop{\downarrow} S$ for some $P \in \J_b$. It
  follows that $x \equiv^m_b y$ just when there exists
  $\{b_i : i \in I\} \in \J_b$ and a family $(c_i \in B : i \in I)$
  with $b_i \leqslant m^\ast c_i$ and $x \equiv_{c_i} y$ for each $i$.
  However, in what follows, we will avoid using this concrete
  description of $\equiv^m_b$ until the very last moment---namely, in
  the proof of (ii) $\Leftrightarrow$ (iii) in
  Proposition~\ref{prop:42}.
\end{Rk}

The following lemma records the basic properties of the relations
$\equiv^m_b$. Its proof is a straightforward exercise in locale theory
but we include it for self-containedness.
\begin{Lemma}
  \label{lem:12}
  Let $\BJM$ be a Grothendieck matched pair of algebras and $X$ a
  $\BJ$-set. The relations $\equiv_b^m$ are equivalence relations, and
  satisfy the following conditions:
  \begin{enumerate}[(i)]
  \item If $x \equiv_{b} y$ then $x \equiv^m_{m^\ast b} y$;
  \item If $x \equiv^m_{b} y$ and $c \leqslant b$ then $x \equiv^m_{c} y$;
  \item If $P \in \J_b$ and 
    $x \equiv^m_{c} y$ for all $c \in P$, then $x \equiv^m_{b} y$;
  \item If $x \equiv^m_{b} y$ then $x \equiv^{nm}_{n^\ast b} y$ for any
    $n \in M$;
  \item If $X$ is a $\BJM$-set and $x \equiv^m_{b} y$ then $m
    \cdot x \equiv_b m \cdot y$;
  \item If $m \equiv_b n$ then $\equiv^m_{c}$ and $\equiv^n_{c}$
    coincide for all $c \leqslant b$.
  \end{enumerate}
\end{Lemma}
\begin{proof}
  $\equiv_b^m$ is reflexive and symmetric since
  $m^\ast \sheq x x = m^\ast B = B$ and
  $m^\ast \sheq x y = m^\ast \sheq y x$. For transitivity we proceed
  in stages:
  \begin{enumerate}[(a)]
  \item If $x\equiv_b y$ and $y \equiv_c z$, then
    $x \equiv_{b \wedge c} z$ and so
    $m^\ast(b \wedge c) = m^\ast b \wedge m^\ast c \in
    m^\ast \sheq x z$;
  \item If $x \equiv_b y$, we may consider the $\J$-closed ideal
    $I = \{d \in B : m^\ast b \wedge d \in m^\ast \sheq x z\}$. By
    (a), each $m^\ast c$ with $y \equiv_c z$ is in $I$ and so
    $m^\ast \sheq y z \subseteq I$.
  \item Consider the $\J$-closed ideal $J = \{e \in B : e \wedge d \in
    m^\ast \sheq x z \,\forall d \in m^\ast \sheq y z\}$. By (b), $J$
    contains $m^\ast b$ whenever $x \equiv_b y$ and so $m^\ast \sheq x
    y \subseteq J$.
  \end{enumerate}
  But (c) says that $x \equiv^m_b y$ and $y \equiv^m_c z$ imply
  $x \equiv^m_{b \wedge c} z$, whence each $\equiv^m_b$ is
  transitive.

  Now, conditions (i)--(iii) simply say that each $m^\ast \sheq x y$
  is a closed ideal. For (iv), note that
  $\{b : n^\ast b \in (nm)^\ast \sheq x y\}$ is a closed $\J$-ideal
  which contains the set $\{m^\ast b : x \equiv_b y\}$, and so contains
  $m^\ast \sheq x y$. (v) follows similarly starting from the
  $\J$-closed ideal $\{b : m \cdot x \equiv_b m \cdot y\}$. Finally,
  for (vi), it suffices by symmetry to show that
  $(c \in m^\ast \sheq x y \text{ and } c \leqslant b)$ implies $c \in
  n^\ast \sheq x y$; or equivalently, that
  $c \in m^\ast \sheq x y$ implies $b \wedge c \in n^\ast \sheq x y$.
  But we observe that the $\J$-closed ideal
  $K = \{d \in B : b \wedge d \in n^\ast \sheq x y\}$ contains
  $m^\ast c$ whenever $x \equiv_c y$, since $m \equiv_b n$ implies
  $b \wedge m^\ast c = b \wedge n^\ast c \leqslant n^\ast c \in n^\ast
  \sheq x y$; whence $m^\ast \sheq x y \subseteq K$ as desired.
\end{proof}

The discussion above now suggests taking $M \otimes_B X$ to be
$M \times X$ with the free $M$-action and the $\BJ$-set equalities
$(m,x) \equiv_b (n,y)$ iff $m \equiv_b n$ and $x \equiv^m_b y$
(equivalently, $x \equiv^n_b y$ by part (vi) of the previous lemma).
One immediate problem is that $\equiv_1$ with this definition need not
be the identity; so we had better quotient out by it. That is, we
refine our first guess by taking
$M \otimes_B X = \{(m, x) : m \in M, x \in X \quot \equiv^m_1\}$ under
the $M$-action and $\BJ$-set equalities described above. If we work
this through, we get all of the necessary axioms for a $\BJM$-set
\emph{except for} the condition that, for any partition $P \in \J$ and
family of elements $(m_b, x_b)$ indexed by $b \in P$, there should be
an element $(n, z)$ with $(n,z)\equiv_b (m_b, x_b)$ for all $b \in B$.
In the first component there is no problem: we use the $\BJ$-set
structure of $M$. However, in the second component, we must formally
adjoin the missing elements, while accounting for the ones which do
already exist; and we can do so by replacing $X$ by the $\BJ$-set of
distributions $T_{\BJ} X$ and quotienting appropriately. This
motivates:

\begin{Prop}
  \label{prop:31}
  Let $\BJM$ be a Grothendieck matched pair of algebras. The
  forgetful functor
  $U \colon \BJM\text-\cat{Set} \rightarrow \BJ\text-\cat{Set}$
  has a left adjoint $M \otimes_B (\thg)$, whose value $M \otimes_B X$
  at a $\BJ$-set $X$ is given by the quotient of the free
  $\BJM$-set $M \times T_{\BJ} X$ by the $\BJM$-set
  congruence $\sim$ for which 
  \begin{equation*}
    (m, \omega) \sim (n, \gamma) \quad \iff \quad m = n \text{ and } x \equiv^m_{\omega(x) \wedge \gamma(y)} y \text{ for all $x, y \in X$.}
  \end{equation*}
\end{Prop}
\begin{proof}
  We first show $\sim$ is an equivalence relation. Symmetry is clear.
  For reflexivity, if $x \neq y \in X$ then
  $\omega(x) \wedge \omega(y) = 0$ and so
  $x \equiv^m_{\omega(x) \wedge \omega(y)} y$ is always true. For
  transitivity, suppose
  $(m, \omega) \sim (m, \gamma) \sim (m, \delta)$. We must show
  $(m, \omega) \sim (m, \delta)$, i.e.
  $x \equiv^m_{\omega(x) \wedge \delta(z)} z$ for all $x,z \in X$. Now
  $\{\omega(x) \wedge \gamma(x) \wedge \delta(z) : y \in Y\}^-$ is in
  $\J_{\omega(x) \wedge \delta(z)}$ so by Lemma~\ref{lem:12}(iii) it
  suffices to check
  $x \equiv^m_{\omega(x) \wedge \gamma(y) \vee \delta(z)} z$ which
  follows from $x \equiv^m_{\omega(x) \wedge \gamma(y)} y$ (as $(m,
  \omega) \sim (m, \gamma)$) and
  $y \equiv^m_{\gamma(y) \wedge \delta(z)} z$ (as $(m, \gamma) \sim (m, \delta)$).

  We now show $\sim$ is a congruence. For the $M$-set structure, if
  $(m, \omega) \sim (m, \delta)$, i.e.,
  $x \equiv^m_{\omega(x) \wedge \gamma(y)} y$ for all $x,y \in X$,
  then $x \equiv^{nm}_{n^\ast \omega(x) \wedge n^\ast \gamma(y)} y$ by
  Lemma~\ref{lem:12}(iv), whence
  $(nm, n^\ast \circ \omega) \sim (nm, n^\ast \circ \gamma)$. For the
  $\BJ$-set structure, let $P \in \J$ and suppose
  $(m_b, \omega_b) \sim (m_b, \gamma_b)$ for all $b \in B$, i.e.,
  \begin{equation}\label{eq:42}
    x \equiv^{m_b}_{\omega_b(x) \wedge \gamma_b(y)} y \text{ for all } x,y \in X\rlap{ .}
  \end{equation}
  We must show that $(P(m), P(\omega))
  \sim (P(m), P(\gamma))$, i.e., that
  \begin{equation*}
    x \equiv^{P(m)}_{\bigvee_{b} (b \wedge \omega_b(x) \wedge \gamma_b(y))} y \text{ for all } x,y \in X\rlap{ .}
  \end{equation*}
  For this, it suffices by Lemma~\ref{lem:12}(iii) to show
  $x \equiv^{P(m)}_{b \wedge \omega_b(x) \wedge \gamma_b(y)} y$ for
  all $x,y \in X$ and $b \in P$; but since $P(m) \equiv_b m_b$, this
  is equally by Lemma~\ref{lem:12}(vi) to show that
  $x \equiv^{m_b}_{b \wedge \omega_b(x) \wedge \gamma_b(y)} y$ for all
  $x,y \in X$ and $b \in P$; which follows from~\eqref{eq:42} via
  Lemma~\ref{lem:12}(ii).
  So $\sim$ is a congruence and we can form
  the $\BJM$-set $M \otimes_B X = (M \times T_{\BJ} X) / \sim$. We now
  show that the composite map
  \begin{equation}
    \label{eq:43}
    \eta \defeq X \xrightarrow{\ \ \ } M \times T_{\BJ}X \xrightarrow{\ q \ } M \otimes_B X
  \end{equation}
  exhibits $M \otimes_B X$ as the free $\BJM$-set on the $\BJ$-set
  $X$; here, the first part is the free morphism
  $X \rightarrow M \times T_{\BJ} X$ sending $x \mapsto (1, \pi_x)$,
  and the second part is the quotient map for $\sim$.

  First of all, this map is a $\BJ$-set homomorphism, since if
  $x,y \in X$ and $b \in B$, then
  $(1, \pi_{b(x,y)}) \sim (1, b(\pi_{x}, \pi_{y}))$ in
  $M \times T_{\BJ}X$; for indeed, the only non-trivial cases for
  $\sim$ are that
  $b(x,y) \equiv^1_{1 \wedge b} x$ and $b(x,y) \equiv^1_{1 \wedge
    b'} y$, which simply says that $b(x,y) \equiv_b x$ and $b(x,y)
  \equiv_{b'} y$, which is so by definition of $b(x,y)$.

  Moreover, if $f \colon X \rightarrow Y$ is a $\BJ$-set homomorphism,
  then we have a unique extension along $\eta$ to a $\BJM$-set
  homomorphism $\bar f \colon M \times T_{\BJ} X \rightarrow Y$. To
  complete the proof, it suffices to show this extension factors
  through $q$. So suppose $(m, \omega) \sim (m, \gamma)$ in
  $M \times T_{\BJ} X$. We have that
  $\bar f(m, \omega) \equiv_{\omega(x)} m \cdot f(x)$ and
  $\bar f(m, \gamma) \equiv_{\gamma(y)} m \cdot y(x)$ for all
  $x,y \in X$; and since $x \equiv^m_{\omega(x) \wedge \gamma(y)} y$
  we have
  $m \cdot f(x) \equiv_{\omega(x) \wedge \gamma(y)} m \cdot f(y)$ by
  Lemma~\ref{lem:12}(v). Thus
  $\bar f(m, \omega) \equiv_{\omega(x) \wedge \gamma(y)} m \cdot f(x)
  \equiv_{\omega(x) \wedge \gamma(y)} m \cdot f(y) \equiv_{\omega(x)
    \wedge \gamma(y)} \bar f(m, \gamma)$, and joining over $x$ and $y$
  gives $\bar f(m, \omega) = \bar f(m, \gamma)$ as desired.
\end{proof}

We are now in a position to analyse when the forgetful functor
$U \colon \BJM\text-\cat{Set} \rightarrow \BJ\text-\cat{Set}$ is
cartesian closed. Spelling it out, we see that the condition in
Definition~\ref{def:30} is equivalent to asking that, for all
$\BJM$-sets $X$, $Y$, the function
\begin{align*}
  \BJM\text-\cat{Set}(M \times X,Y) &\rightarrow \BJ\text-\cat{Set}(X, Y) &
  f & \mapsto f(1, \thg)
\end{align*}
is invertible. Thus, $U$ is cartesian closed just if, whenever $X,Y$ are
$\BJM$-sets, each $\BJ$-set map $g \colon X \rightarrow Y$ extends
uniquely to a $\BJM$-set map $M \times X \rightarrow Y$ along the
$\BJ$-set homomorphism $\gamma \colon X \rightarrow M \times X$
sending $x$ to $(1,x)$; in other words, if $\gamma$ exhibits
$M \times X$ as the free $\BJM$-set on the $\BJ$-set $UX$. However,
since we already know that the $\BJ$-set homomorphism
$\eta \colon X \rightarrow M \otimes_B X$ of~\eqref{eq:43} exhibits
$M \otimes_B X$ as the free $\BJM$-set on $UX$, this is equally to say
that the unique extension $M \otimes_B X \rightarrow M \times X$ of
$\gamma$ to a $\BJM$-set homomorphism, as described in the proof
of~Proposition~\ref{prop:31}, is \emph{invertible}. We record this as:
\begin{Lemma}
  \label{lem:13}
  Let $\BJM$ be a Grothendieck matched pair of algebras. The forgetful
  functor
  $U \colon \BJM\text-\cat{Set} \rightarrow \BJ\text-\cat{Set}$ is
  cartesian closed if, and only if, for each $\BJM$-set $X$, the
  function:
  \begin{equation}
    \label{eq:44}
    \begin{aligned}
      \theta_X \colon M \otimes_B X &\rightarrow M \times X \\ (m, \omega) &\mapsto (m,
      \varepsilon_m(\omega))
    \end{aligned}
  \end{equation}
  is invertible, where $\varepsilon_m(\omega)$ is characterised by  $\varepsilon_m(\omega) \equiv_{\omega(x)} m \cdot x$
  for all $x \in \mathrm{supp}(\omega)$.
\end{Lemma}
We are now in a position to complete the proof of Theorem~\ref{thm:9}
by showing:
\begin{Prop}
  \label{prop:42}
  Let $\BJM$ be a Grothendieck matched pair of algebras. The
  following are equivalent:
  \begin{enumerate}[(i)]
  \item The forgetful functor $U \colon \BJM\text-\cat{Set}
    \rightarrow \BJ\text-\cat{Set}$ is cartesian closed;
  \item For all $m \in M$, there exists $\{b_i : i \in I\} \in \J$ and
    a family $(n_i \in M : i \in I)$ with 
    $mn_i \equiv_{b_i} 1$ and $n_i m \equiv^m_{b_i} 1$ for all $i$.
  \item For all $m \in M$, there exists
    $\{b_i : i \in I\} \in \J$ and families
    $(n_i \in M : i \in I)$ and $(c_i \in B : i \in I)$ with
    $b_i \leqslant m^\ast c_i$, $mn_i \equiv_{b_i} 1$ and
    $n_i m \equiv_{c_i} 1$ for all $i$.
\end{enumerate}
\end{Prop}
\begin{proof}
  We first prove (i) $\Rightarrow$ (ii). So suppose $U$ is cartesian
  closed; we begin by showing that for any $\BJM$-set $X$, any
  $x,y \in X$ and any $b \in B$, we have
  \begin{equation}
    \label{eq:45}
    x \equiv^m_b y \iff m \cdot x \equiv_b m \cdot y\rlap{ .}
  \end{equation}
  Indeed, since $\theta_X$ is an isomorphism by Lemma~\ref{lem:13}, we have
  $\theta_X(m, \pi_x) \equiv_b \theta_X(m, \pi_{y})$ in $M \times X$
  just when $(m, \pi_x) \equiv_b (m, \pi_{y})$ in $M \otimes_B X$.
  Since $\theta_X(m, \pi_x) = (m,m \cdot x)$ and similarly for $y$,
  this is equally to say that $m \cdot x \equiv_b m \cdot y$ just when
  $(m, b(\pi_x, \pi_{y})) \sim (m, \pi_{y})$ in $M \times T_{\BJ}X$;
  which by definition of $\sim$ says exactly that $x \equiv^m_b y$.
  
  Now, since $U$ is cartesian closed, \eqref{eq:44} is in
  particular invertible when $X = M$. Thus for each $m \in M$, the
  element $(m,1)$ is in the image of $\theta_M$, and so there exists a
  distribution $\omega \colon M \rightarrow B$ such that
  $\varepsilon_m(\omega) = 1$, i.e., such that
  $1 \equiv_{\omega(n)} mn$ for all $n \in \mathrm{supp}(\omega)$.
  Writing $\{b_i : i \in I\}$ for the partition $(\im \omega)^-$ and
  $n_i \in M$ for the elements with $\omega(n_i) = b_i$, we thus have
  $\{b_i : i \in I\} \in \J$ and a family $(n_i \in M : i \in I)$
  such that $mn_i \equiv_{b_i} 1$ for all $i$. It follows that
  $mn_i m \equiv_{b_i} m$ for all $i$, and so by~\eqref{eq:45} that
  $n_i m \equiv^m_{b_i} 1$ for all $i \in I$. This gives (ii).

  We now show (ii) $\Rightarrow$ (i). We again begin by
  proving~\eqref{eq:45} for any $\BJM$-set $X$. The rightward
  implication is Lemma~\ref{lem:12}(v). As for the leftward one,
  suppose $m \cdot x \equiv_b m \cdot y$. By (ii), we find
  $\{b_i\} \in \J$ and $(n_i \in M : i \in I)$ such that
  $mn_i \equiv_{b_i} 1$ and $n_im \equiv^m_{b_i} 1$ for each $i$. Now
  the $\BJM$-set axioms for $X$ and Lemma~\ref{lem:12}(i) say that
  $m \cdot x \equiv_b m \cdot y \implies n_im \cdot x \equiv_{n_i^\ast
    b} n_im \cdot y \implies n_im \cdot x \equiv^m_{m^\ast n_i^\ast b}
  n_im \cdot y$ for each $i$; and since $mn_i \equiv_{b_i} 1$, we have
  $b_i \wedge m^\ast n_i^\ast b = b_i \wedge b$, and so for each $i$
  we have $n_i m \cdot x \equiv^m_{b \wedge b_i} n_i m \cdot y$. Now,
  since $\sheq {n_i m} 1 \leqslant \sheq {n_i m \cdot x} {x}$ by the
  $\BJM$-set axioms for $X$, also
  $m^\ast\sheq {n_i m} 1 \leqslant m^\ast\sheq {n_i m \cdot x} {x}$;
  whence $n_i m \equiv^m_{b_i} 1$ implies
  $n_i m \cdot x \equiv^m_{b_i} x$. Putting this together, we have
  $x\equiv^m_{b \wedge b_i} n_i m \cdot x \equiv^m_{b \wedge b_i} n_i
  m \cdot y \equiv^m_{b \wedge b_i} y$ for each $i$, whence
  $x \equiv^m_b y$ by Lemma~\ref{lem:12}(iii).

  We immediately conclude that each~\eqref{eq:44} is injective: for
  indeed, if $\theta_X(m, \omega) = \theta_M(n, \gamma)$, then $m = n$
  and $\varepsilon_m(\omega) = \varepsilon_m(\gamma)$, which says that
  $m \cdot x \equiv_{\omega(x) \wedge \gamma(y)} m \cdot y$ for each
  $x,y \in X$. By~\eqref{eq:45} this is equivalent to
  $x \equiv^m_{\omega(x) \wedge \gamma(y)} y$ for all
  $x,y \in X$---which is to say that $(m, \omega) = (n, \gamma)$ in
  $M \otimes_B X$. Finally, to show surjectivity of $\theta_X$,
  consider $(m,x) \in M \times X$, let $\{b_i\} \in \J$ and $(n_i \in M)$ 
  be as in (ii) for $m$, and let
  $\omega \colon X \rightarrow B$ be the distribution with
  $\omega(y) = \bigvee_{y = n_i \cdot x} b_i$. We claim 
  $\theta_X(m, \omega) = (m,x)$; for which we must show that
  $x \equiv_{\omega(y)} m \cdot y$ for all $y \in X$. This is equally
  to show $x \equiv_{b_i} m n_i \cdot x$ for all $x \in X$, which
  is so since $mn_i \equiv_{b_i} 1$ for each $i$.

  Finally, we prove (ii) $\Leftrightarrow$ (iii). Given $m \in M$ and
  the associated data $\{b_i\}$, $(c_i)$ and $(n_i)$ in (iii), we have
  by Lemma~\ref{lem:12}(i) and (ii) that
  $n_i m \equiv_{c_i} 1 \implies n_i m \equiv^m_{m^\ast c_i} 1
  \implies n_i m \equiv^m_{b_i} 1$ for each $i$: which gives the data
  needed for (ii). Conversely, given the data $\{b_i\}$ and $(n_i)$ as
  in (ii), since $n_i m \equiv^m_{b_i} 1$ for each $i$ we have by
  Remark~\ref{rk:6} partitions $\{b_{ij} : j \in J_i\} \in \J_{b_i}$
  and elements $(c_{ij} : i \in I, j \in J_i)$ such that
  $b_{ij} \leqslant m^\ast c_{ij} $ and $n_i m \equiv_{c_{ij}} 1$ for
  each $i \in I$ and $j \in J_i$. Thus taking the partition
  $\{b_{ij} : i \in I, j \in J_i\} \in \J$, the elements
  $(n_{ij} = n_i \in M : i \in I, j \in J_i)$ and the elements
  $(c_{ij} \in B : i \in I, j \in J_i)$ we obtain the required
  witnesses for~(iii).
\end{proof}

Of course, when the equivalent conditions of Theorem~\ref{thm:9} are
satisfied, the function-space $Z^Y$ in $\BJM\text-\cat{Set}$ is given by
the $\BJ$-set of $\BJ$-set homomorphisms $Y \rightarrow Z$, with a
suitable $M$-set structure. From the above proof we can extract a
direct description of this structure. Given $f \in Z^Y$ a $\BJ$-set homomorphism and $m \in M$ with associated
data $\{b_i\} \in \J$, $(c_i)$ and $(n_i)$ as above, the element $m
\cdot f \in Z^Y$ is characterised by
\begin{equation}
  \label{eq:58}
  (m \cdot f)(y) \equiv_{b_i} m \cdot f(n_i \cdot y) \quad \text{for all } i \in I\rlap{ ;}
\end{equation}
this is the natural generalisation of~\eqref{eq:57} above.

\section{J\'onsson--Tarski toposes}
\label{sec:jonss-tarski-topos}

We conclude this paper by discussing a range of examples of cartesian
closed varieties whose classifying categories are the kinds of ample
topological groupoids that are of interest to operator algebraists. In
this section, we describe cartesian closed varieties (in fact toposes)
which correspond to the \emph{Cuntz groupoids}
of~\cite{Renault1980A-groupoid}, whose corresponding $C^\ast$-algebras
are Cuntz $C^\ast$-algebras and whose corresponding $R$-algebras are
Leavitt algebras.

\subsection{The J\'onsson--Tarski topos}
\label{sec:jonsson-tarski-topos-1}
We begin with the simplest non-trivial case involving a binary
alphabet $\{\ell, r\}$, for which the appropriate variety will be the
so-called \emph{J\'onsson--Tarski topos}. A \emph{J\'onsson--Tarski
  algebra}~\cite{Jonsson1961On-two-properties} is a set $X$ endowed
with functions $\ell, r \colon X \rightarrow X$---which we write as
left actions $x \mapsto \ell \cdot x$ and $x \mapsto r \cdot x$---and
a function $m \colon X \times X \rightarrow X$ satisfying the axioms
\begin{equation}
  \label{eq:54}
  m(\ell \cdot x,r \cdot x) = x \qquad \ell \cdot m(x,y) = x \quad \text{and} \quad r \cdot m(x,y) = y\rlap{ .}
\end{equation}
These say that the functions $x \mapsto (\ell \cdot x,r \cdot x)$ and
$x,y \mapsto m(x,y)$ are inverse; so a
J\'onsson--Tarski algebra is equally well a set $X$ with an
isomorphism $X \cong X \times X$.

The concrete category $\J\T$ of J\'onsson--Tarski algebras is a
non-degenerate finitary variety, but also, as observed by Freyd, a
topos; indeed, as explained in~\cite[Example~1.3]{Johnstone1985When},
it can be presented as the topos of sheaves on the free monoid
$A^\ast$---where $A$ denotes the two-letter alphabet
$\{\ell, r\}$---for the topology generated by the covering family
$\{\ell, r\}$. In particular, $\J\T$ is cartesian closed and so via
Proposition~\ref{prop:44} can be presented as a category of
$\BM$-sets.

When we calculate $B$ and $M$, it will turn out that, on the one hand,
$B$ is the Boolean algebra of clopen sets of Cantor space $C$ which,
because of our conventions, it will be best to think of as the set
$\{\ell, r\}^{-\omega}$ of words $W = \cdots a_2 a_1 a_0$ in $\ell,r$
which extend infinitely to the \emph{left}. On the other hand, $M$
will be the monoid of those (continuous) endomorphisms
$\varphi \colon C \rightarrow C$ which are specified by finite words
$u_1, \dots, u_k, v_1, \dots, v_k \in A^{\ast}$ via the formula:
\begin{equation}
  \label{eq:60}
  \left.
  \begin{aligned}
    \varphi(Wu_1) &= Wv_1 \\*[-0.3\baselineskip]
    & \  \ \vdots \\
    \varphi(Wu_k) &= Wv_k
  \end{aligned}\  \right\} \quad  \text{for all } W \in A^{-\omega}\rlap{ ;}
\end{equation}
i.e., $\varphi$ maps the clopen set $[u_i]$ of words starting with
$u_i$ affinely to the clopen set $[v_i]$. (Although our infinite words
extend to the \emph{left}, we still think of them as \emph{starting}
with their rightmost segments $a_n \cdots a_0$). The \emph{invertible}
elements of this monoid comprise Thompson's group $V$, and so it is no
surprise that $M$ is already known as a monoid generalisation of $V$;
in the nomenclature of~\cite{Birget2009Monoid}, it is the
``Thompson--Higman total function monoid $\mathit{tot}M_{2,1}$''.

Now, by Proposition~\ref{prop:44}, we can compute $M$ and $B$ as
$\J\T(F1,F1)$ and $\J\T(F1, 1+1)$, where $F1$ is the free
J\'onsson--Tarski algebra on one generator. The obvious way to find
these would be via a universal-algebraic description of $F1$ and
$1+1$; this was the approach of Higman in~\cite{Higman1974Finitely},
who used it to show that $\mathrm{Aut}(F1,F1) \cong V$.
However~\cite{Birget2009Monoid} follows a combinatorially smoother
approach due to~\cite{Scott1984Construction}, which describes $V$ and
its monoid generalisations in terms of certain morphisms between
ideals of the monoid $A^\ast$. As we now show, there is a direct
derivation of this perspective which exploits the nature of $\J\T$ as
a topos of sheaves on $A^\ast$. Again, due to our conventions it will
be best if we work with \emph{left}, rather than \emph{right},
$A^\ast$-sets; this is harmless due to the anti-homomorphism
$A^\ast \rightarrow A^\ast$ which reverses each word.

Thus $\J\T$ is related to the
category of left $A^\ast$-sets by adjunctions:
\begin{equation}
  \label{eq:56}
  \cd{
    {\J\T} \ar@<-4.5pt>[r]_-{I_2} \ar@{<-}@<4.5pt>[r]^-{L_2} \ar@{}[r]|-{\bot} &
    {{(A^\ast\text-\cat{Set})_{\mathrm{sep}}} } \ar@<-4.5pt>[r]_-{I_1} \ar@{<-}@<4.5pt>[r]^-{L_1} \ar@{}[r]|-{\bot} &
    {A^\ast\text-\cat{Set}}\rlap{ ,}
  }
\end{equation}
where $(A^\ast\text-\cat{Set})_{\mathrm{sep}}$ is the category of
\emph{separated} left $A^\ast$-sets for the Grothendieck topology on
$A^\ast$; concretely, $X$ is separated if $x = y$ whenever
$\ell \cdot x = \ell \cdot y$ and $r \cdot x = r \cdot y$. The free
separated $A^\ast$-set $L_1(X)$ on an $A^\ast$-set $X$ is
$X \quot \sim$, where $\sim$ is the smallest equivalence relation that
relates $x$ and $y$ whenever $\ell \cdot x = \ell \cdot y$ and
$r \cdot x = r \cdot y$. As for the left-hand adjunction
in~\eqref{eq:56}, we may by~\cite[Theorems~43.6 and
45.8]{Wyler1991Lecture} see $L_2$ as the functor which formally
inverts the class of \emph{dense inclusions} for the Grothendieck
topology on $A^\ast$, which we can make explicit as
follows:
\begin{Defn}
  \label{def:37}
  Let $X$ be a left $A^\ast$-set and $U \leqslant X$ a
  sub-$A^\ast$-set. We say:
  \begin{itemize}
  \item $U$ is \emph{closed} in $X$ (written $U \leqslant_c X$) if
    $\ell \cdot x \in U, r \cdot x \in U \implies x \in U$;
  \item $U$ is \emph{dense} in $X$ (written $U \leqslant_d X$) if the
    closure of $U$ in $X$ is $X$.
  \end{itemize}
Here, the closure of $U$ in $X$ is, of course, the smallest closed
$U' \leqslant X$ which contains $U$; and it is not
hard to see that it can be described explicitly as:
\begin{equation}
  \label{eq:55}
  U' = \{\,x \in X : \text{there exists } n \in \mathbb{N} \text{ with } w \cdot x \in U \text{ for all } w \in A^n\,\}\rlap{ .}
\end{equation}
\end{Defn}

Now, since the class of dense inclusions in
$(A^\ast\text-\cat{Set})_\mathrm{sep}$ is closed under composition and
under inverse image along any $A^\ast$-set homomorphism, the result of
formally inverting them can be described via a \emph{category of
  fractions}~\cite{Gabriel1967Calculus}. This is to say that $\J\T$ is
equivalent to the category $\J\T'$ wherein:
\begin{itemize}
\item \textbf{Objects} are separated left $A^\ast$-sets;
\item \textbf{Morphisms} $X \rightarrow Y$ are $\sim$-equivalence
  classes of dense partial $A^\ast$-set maps, i.e., pairs
  $(U \leqslant_d X, f \colon U \rightarrow Y)$ where $f$ is an
  $A^\ast$-set homorphism. Here, $(U, f) \sim (V, g)$ when they have a
  lower bound in the inclusion ordering $\sqsubseteq$, i.e., the
  ordering with $(U',f') \sqsubseteq (U,f)$ when $U' \leqslant U$ and
  $f' = \res f {U'}$;
\item The \textbf{composite} of $(U,f) \colon X \rightarrow Y$ and
  $(V,g) \colon Y \rightarrow Z$ is their composite as partial maps,
  namely, $(f^{-1}(V), \lambda x.\,gfx) \colon X \rightarrow Z$; and 
\item The \textbf{identity} on $X$ is $(X, 1_X)$,
\end{itemize}
In fact, we can simplify the description of $\J\T'$ further, due to
the following result; this is really a general argument about dense
and closed monomorphisms with respect to a Grothendieck topology, but
we give a concrete proof for our situation.
\begin{Lemma}
  \label{lem:7}
  Each equivalence-class of morphisms $X \rightarrow Y$ in $\J\T'$ has
  a $\sqsubseteq$-largest representative. These representatives are
  precisely those $(U,f)$ for which the graph $\{(x, fx) : x \in U\}$
  of $f$ is
  closed in $X \times Y$.
\end{Lemma}
\begin{proof}
  Given a dense partial map $(U,f) \colon X \rightarrow Y$, let
  $G \leqslant X \times Y$ be the graph of $f$ and
  $G' \leqslant_c X \times Y$ its closure. We claim that $G'$ is in
  turn the graph of a function, i.e., that if
  $\{(x,y), (x,y')\} \subseteq G'$, then $y = y'$. From~\eqref{eq:55},
  if $(x,y) \in G'$, then there is some $n$ so that
  $(w \cdot x, w \cdot y) \in G$ for every $w \in A^n$. We
  get a corresponding $n$ for $(x,y')$ and so on taking the larger of
  the two we may assume that $(w \cdot x, w \cdot y)$ and
  $(w \cdot x, w \cdot y')$ are in $G$ for all $w \in A^n$.
  But then, as $G$ is the graph of a function,
  $w \cdot y = w \cdot y'$ for all $w \in A^n$, whence
  $y = y'$ by separatedness of $Y$.

  So taking $U' = \{x \in X : (x,y) \in G'\}$ we see that
  $G'$ is the graph of a $A^\ast$-set homomorphism
  $f' \colon U' \rightarrow Y$; and since $U \leqslant U' \leqslant X$
  and $U \leqslant_d X$, also $U' \leqslant_d X$, so that $(U',f')$ is
  a dense partial map, which, since clearly
  $(U, f) \sqsubseteq (U', f')$, is equivalent to $(U,f)$. Finally,
  note that if $(U,f) \sqsubseteq (V,g) \colon X \rightarrow Y$, then
  $U \leqslant_d V$ and so the graph of $f$ is dense in the graph of
  $g$; as such, they have the same closures, so that
  $(U', f') = (V', g')$. Thus the assignment $(U,f) \mapsto (U',f')$
  picks out a $\sqsubseteq$-maximal representative of each equivalence
  class.
\end{proof}
Combining this  with our preceding observations, we arrive at:
\begin{Lemma}
  \label{lem:3}
  The category $\J\T$ is equivalent to the category $\J\T'$ wherein:
  \begin{itemize}
  \item \textbf{Objects} are separated left $A^\ast$-sets;
  \item \textbf{Morphisms} $X \rightarrow Y$ are
    dense partial maps $(U,f) \colon X \rightarrow Y$ which are
    \emph{maximal}, in the sense that the graph of $f$ is closed in $X
    \times Y$;
  \item The \textbf{composite} of $(U,f)$ and
    $(V,g)$ is the maximal extension  of  $(f^{-1}(V), \lambda x.\,gfx)$;
  \item The \textbf{identity} on $X$ is $(X, 1_X)$, 
  \end{itemize}
  via an equivalence which identifies
  $L_2 \colon (A^\ast\text-\cat{Set})_\mathrm{sep} \rightarrow \J\T$ with
  the identity-on-objects functor
  $(A^\ast\text-\cat{Set})_{\mathrm{sep}} \rightarrow \J\T'$ sending
  $f \colon X \rightarrow Y$ to $(X,f) \colon X \rightarrow Y$.
\end{Lemma}

Now, $F1 \in \J\T$ is the image under $L_2L_1$ of the free left
$A^\ast$-set on one generator which is, of course, $A^\ast$ itself. Since
$A^\ast$ is left-cancellable, it is separated as a left $A^\ast$-set, and
so $L_1(A^\ast) = A^\ast$; whence, by the preceding lemma, we can identify
$F1 \in \J\T$ with $A^\ast \in \J\T'$, and so identify the monoid
$M = \J\T(F1,F1)$ with $\J\T'(A^\ast, A^\ast)$, the monoid of maximal dense
partial left $A^\ast$-set maps $A^\ast \rightarrow A^\ast$.

To relate this to~\cite{Birget2009Monoid}, let us note that a left
ideal (i.e., sub-$A^\ast$-set) $I \leqslant A^\ast$ is dense just when
its closure $I'$ contains the empty word, which, by~\eqref{eq:55},
happens just when $A^n \subseteq I$ for some $n \in \mathbb{N}$. This
is easily equivalent to $I$ being \emph{cofinite}, i.e.,
$A^\ast \setminus I$ being finite, but also, as explained
in~\cite{Scott1984Construction}, to $I$ being finitely generated and
\emph{essential}, meaning that it intersects every non-trivial left
ideal of $A^\ast$. Thus, $M = \J\T'(A^\ast, A^\ast)$ is the monoid of
pairs $(I,f)$, where $I \leqslant A^\ast$ is a finitely generated
essential left ideal and $f \colon I \rightarrow A^\ast$ is a maximal
$A^\ast$-set map, with the monoid product given by partial map
composition followed by maximal extension. Modulo our conventions
(left, not right, actions; product in $M$ given by composition in
diagrammatic, not applicative, order), this is the definition of
$\mathit{tot}M_{2,1}$ in~\cite{Birget2009Monoid}.

To further relate this description of $M$ to the presentation
in~\eqref{eq:60}, note that any ideal $I \leqslant A^\ast$ is
generated by the (finite) set of those words $u_1, \dots, u_k \in I$
which have no proper initial segment in $I$ (where, again, ``initial''
means ``rightmost''); we call these words the \emph{basis} of $I$ and
write $I = \spn{u_1, \dots, u_k}$. Now given
$(I,f) \colon A^\ast \rightarrow A^\ast$ in $M$, on taking the basis
$\{u_i\}$ of $I$ and associated elements $v_i = f(u_i)$, we obtain
data for a function $\varphi \colon C \rightarrow C$ as
in~\eqref{eq:60}. Density of $I$ ensures this $\varphi$ is
\emph{total}; while maximality of $f$ ensures each such $\varphi$ is
represented by a \emph{unique} $(I,f)$.

We now describe the Boolean algebra $B = \J\T(F1, 1+1)$. Since
$1+1 \in \J\T$ is the image under $L_2$ of the separated $A^\ast$-set
$\{\top, \bot\}$ with the trivial $A^\ast$-action, we can describe $B$
as $\J\T'(A^\ast, \{\top, \bot\})$, that is, as the set of maximal
dense partial maps $(I,f) \colon A^\ast \rightarrow \{\top, \bot\}$.
For such a map, the inverse images $I_\top = f^{-1}(\top)$ and
$I_\bot = f^{-1}(\bot)$ are sub-ideals of $I$ which partition it and
which, by maximality of $f$, must be \emph{closed} in $A^\ast$.
Furthermore, if $I = \spn{G}$, then $I_\top = \spn{G \cap I_\top}$ and
and $I_\bot = \spn{G \cap I_\bot}$; in particular, they are finitely
generated. Of course, we can re-find $I$ from $I_\bot$ and $I_\top$ as
their (disjoint) union, whence $\J\T'(A^\ast, \{\top, \bot\})$ is
isomorphic to the set of pairs of finitely generated closed ideals
$I_\top, I_\bot \leqslant_c A^\ast$ which are \emph{complementary},
meaning that $I_\top \cap I_\bot = \emptyset$ and $I_\top \cup I_\bot$
is dense in $A^\ast$.

In fact, any finitely generated closed ideal $I$ has a unique finitely
generated closed complement $I'$; indeed, if $I = \spn{G}$ and $n$ is
the length of the longest word in $G$, then $I'$ is the closed ideal
generated by $\{w \in A^n : w \notin I\}$. Thus we can identify $B$
with the Boolean algebra of finitely generated closed left ideals of
$A^\ast$; which in turn can be identified with the Boolean algebra of
clopen sets of Cantor space $A^{-\omega}$, where
$I \leqslant_c A^\ast$ corresponds to the clopen set of words with an
initial segment in $I$; note closedness ensures each clopen set is
represented by a \emph{unique} $I$.

To complete the description of $\BM$, we must give the actions of $B$
and $M$ on each other; using the structure of $\J\T'$ it is not hard
to show that these are given as follows. If $m = (I,f)$ and $n =
(J,g)$ are in $M$, and $b = (K \leqslant_c A^\ast)$ is in $B$, then
\begin{itemize}
\item $m^\ast b \in B$ is the closure of $f^{-1}(K) \leqslant I \leqslant A^\ast$.
\item $b(m,n) \in M$ is the maximal extension of
  $(K \cap I + K' \cap J, \spn{{\res f {K\cap I}}, {\res g {K' \cap J}}})$.
\end{itemize}
Equally, if we view elements of $M$ as continuous endomorphisms
$\varphi$ of Cantor space, and elements of $B$ as clopens $U$ of
Cantor space, then the $M$-action on $B$ is given by
$\varphi, U \mapsto \varphi^{-1}(U)$, while the $B$-action on $M$ is
given by
$U, \varphi, \psi \mapsto \spn{\res \varphi U, \res \psi {U^c}}$.

Let us also indicate how each J\'onsson--Tarski algebra $X$ becomes a
$\BM$-set. First note that, viewing such an $X$ as a left $A^\ast$-set,
the maximal extension $(I',f')$ of a dense partial map
$(I \leqslant_d A^\ast, f \colon I \rightarrow X)$ is a \emph{total} map,
i.e, $I' = A^\ast$; for indeed, if not, then on choosing a word $w$ of
maximal length in $A^\ast \setminus I'$, we would have $(\ell w, x)$ and
$(rw, y)$ in the graph of $f'$ but then by closedness would have
$(w, m(x,y))$ also in the graph, a contradiction. Thus, for the
$B$-set structure on $X$, given $x,y \in X$ and
$b = (J \leqslant_c A^\ast)$ in $B$, we take $b(x,y)$ to be the element
classified by the maximal extension of the dense partial map
\begin{equation*}
  J + J' \xrightarrow{\text{inclusion}} A^\ast + A^\ast \xrightarrow{\spn{x,y}} X\rlap{ ;}
\end{equation*}
while given $m = (I, f)$ in $M$ and $x \in X$, we take $m \cdot x$ as
the element classified by the maximal extension of the dense partial
map $x \circ f \colon I \rightarrow A^\ast \rightarrow X$.

Finally, we remark on some of the other perspectives on $\BM$. The
associated Boolean restriction monoid $S$ is the
\emph{Thompson--Higman partial function monoid} $M_{2,1}$
of~\cite{Birget2009Monoid}, whose elements are maximal partial maps
$(I, f) \colon A^\ast \rightarrow A^\ast$ defined on an arbitrary finitely
generated ideal. If we consider the following elements of $S$:
\begin{equation*}
  \ell = (A^\ast, (\thg) \cdot \ell) \quad r = (A^\ast, (\thg) \cdot r) \quad \ell^\ast = (A^\ast\ell, \partial) \quad r^\ast = (A^\ast r, \partial)
\end{equation*}
where $\partial$ is the function which deletes the last element of a
non-empty word, then $S$ can equally be described as the free Boolean
restriction monoid generated by $\ell, r, \ell^\ast, r^\ast$ subject
to the axioms
\begin{equation}
  \label{eq:64}
  \ell \ell^\ast = rr^\ast = 1 \qquad \ell r^\ast = r \ell^\ast = 0 \quad \text{and} \quad \ell^\ast \ell \vee r^\ast r = 1\rlap{ .}
\end{equation}
(These may look backwards to those familiar with the Cuntz
$C^\ast$-algebra, but recall $st$ means ``first $s$ then $t$''.) If
for a word $a_1 \cdots a_k \in A^\ast$ we write
$(a_1 \cdots a_k)^\ast = a_k^\ast \cdots a_1^\ast$, then these
equations allow every $s \in S$ to be written as
$s = u_1^\ast v_1 \vee \cdots \vee u_k^\ast v_k$ where the $u_i$'s and
$v_i$'s are in $A^\ast$ with the $u_i$'s the basis of an ideal $I$;
composition is then given by juxtaposition and reduction using the
axioms~\eqref{eq:64}. Note each such element
$s = u_1^\ast v_1 \vee \cdots \vee u_k^\ast v_k$ corresponds to a
\emph{partial} endomorphism $C \rightharpoonup C$ defined as
in~\eqref{eq:60}, so that $S$ can equally be identified with the
Boolean restriction monoid of all such partial endomorphisms of $C$.

\eqref{eq:64} also implies that each generator of $S$ is a partial
isomorphism; whence $S$ is \'etale
(cf.~\cite[Proposition~5.1]{Lawson2021Polycyclic}) and so generated by
its Boolean inverse monoid of partial isomorphisms. This Boolean
inverse monoid is the ``Thompson--Higman inverse monoid''
$\mathit{Inv}_{2,1}$ of~\cite{Birget2009Monoid}, or equally, the
\emph{Cuntz inverse monoid} of~\cite{Lawson2021Polycyclic}. This last
identification implies, in turn, that the classifying topological
category of $\BM$ is the well-known \emph{Cuntz groupoid} $\mathfrak{O}_2$
of~\cite[Definition~III.2.1]{Renault1980A-groupoid}, whose Stone space
of objects is Cantor space and whose morphisms $W \rightarrow W'$ are
integers $i$ such that $W_{n} = W'_{n+i}$ for all sufficiently large
$n$. We can also see this directly; indeed, since $B$ comprises the
clopen subsets of Cantor space $C$, the classifying topological
category must have space of objects $C$; and since $M$ comprises all
continuous maps $C \rightarrow C$ of the form~\eqref{eq:60}, the
arrows $W \rightarrow W'$ must be germs at $W$ of those
maps~\eqref{eq:60} for which $\varphi(W) = W'$. This is a well-known
alternative description of $\mathfrak{O}_2$.

Now, since $\J\T$ is a topos, we recover the fact that the Cuntz
groupoid $\mathfrak{O}_2$ is minimal. On the other hand, since $\mathfrak{O}_2$ is a
groupoid and not just a category, the theory of J\'onsson--Tarski
algebras is groupoidal---which also follows from the fact that the
Boolean restriction monoid $S$ is \'etale. In particular, this yields
a simple description of the cartesian closed structure of $\J\T$.
Given $Y,Z \in \J\T$, their function-space $Z^Y$
comprises the $B$-set homomorphisms $Y \rightarrow Z$, i.e., the set
\begin{equation*}
  Z^Y = \{f \colon Y \rightarrow Z \mid w \cdot y = w \cdot y' \implies w \cdot f(y) = w \cdot f(y') \text{ for all } w \in A^\ast\}\rlap{ ,}
\end{equation*}
under an algebra structure which we can read off from~\eqref{eq:58}
as being:
\begin{equation*}
  (\ell \cdot f)(y) = \ell \cdot f(m(y,y)) \quad \text{and} \quad (r \cdot f)(y) = r \cdot f(m(y,y))\rlap{ ,}
\end{equation*}
with inverse $m \colon Z^Y \times Z^Y \rightarrow Z^Y$ given by
$m(g,h)(y) = m(g(\ell \cdot y), h(r \cdot y))$. The correspondence
between algebra maps $X \times Y \rightarrow Z$ and ones
$X \rightarrow Z^Y$ is now given by the usual exponential transpose of
functions.

\subsection{The infinite J\'onsson--Tarski topos}
\label{sec:infin-jonss-tarski}
As noted in~\cite[Example~2]{Rosenthal1981Etendues}, we may generalise
the notion of J\'onsson--Tarski algebra to involve a set $X$ endowed
with an isomorphism $X \rightarrow X^A$ for any fixed set $A$. The
resulting concrete category $\J\T_A$ is still a non-degenerate
variety, and may still be described as a topos of sheaves, now on the
free monoid $A^\ast$ for the topology generated by the cover
$\{a : a \in A\}$.

This generalisation is unproblematic when $A$ is finite, and this case
was already studied by Higman, Scott and
Birget~\cite{Higman1974Finitely, Scott1984Construction,
  Birget2009Monoid}. When $A$ is infinite, things are more
interesting, not least because $\J\T_A$ is then a non-finitary variety
of $\BJM$-sets. With this being said, much of the work we did above
carries over. We can define dense and closed inclusions \emph{mutatis
  mutandis} as before, and we still find $M$ as the monoid of maximal
dense partial maps $A^\ast \rightarrow A^\ast$. The main difference is in
the characterisation of the dense ideals. When $A$ is finite, these
correspond to \emph{finite} $A$-ary branching trees, where a given
tree $\tau$ corresponds to the ideal generated by the addresses of its
leaves. In the infinite case, they correspond to \emph{well-founded}
$A$-ary branching trees; these are potentially infinite, but have no
infinite path starting at the root. The following lemma translates
this into ideal-theoretic language.

\begin{Lemma}
  \label{lem:4}
  An ideal $I \leqslant  A^\ast$ is dense if, and only if, each infinite
  word $\cdots w_2 w_1 w_0 \in A^{-\omega}$ has an initial segment in $I$.
\end{Lemma}
\begin{proof}
  The closure of the ideal $I$ may be computed transfinitely: we take
  $I_0 = I$, take
  $I_{\alpha + 1} = \{w \in A^\ast : aw \in I_\alpha \text{ for all } a
  \in A\}$ and at limit stages take
  $I_{\gamma} = \bigcup_{\alpha < \gamma} I_\alpha$. By Hartog's
  lemma, this transfinite sequence stabilises at some $\lambda$ and
  now $I_\lambda = I'$.

  Suppose first that $I' = I_\lambda = A^\ast$ and let $W \in A^{-\omega}$.
  Writing $\res W n$ for the initial segment of $W$ of length $n$, we
  define
  $\alpha_n = \mathrm{min}\{\gamma \leqslant \lambda : \res W n \in
  I_\gamma\}$; note this is the minimum of a non-empty set of
  ordinals, since the empty word $\epsilon$ is in $I_\lambda$. Now if
  $\alpha_n > 0$ then by the construction of the transfinite sequence
  we must have $\alpha_{n+1} < \alpha_n$; thus, by well-foundedness we
  must have $\alpha_n = 0$ for some $n$, i.e., $\res W n \in I$.
  Conversely, suppose every infinite word $W$ has an initial segment
  in $I$; we show that $\epsilon \in I_\lambda$. Indeed, suppose not.
  Since $I_\lambda = I_{\lambda+1}$, for every $w \notin I_\lambda$
  there must exist some $a \in A$ for which $aw \notin I_\lambda$.
  Starting from $\epsilon$ and making countably many dependent
  choices, we thus obtain a sequence of words
  $\epsilon, w_0, w_1w_0, w_2w_1w_0, \dots$ and so an infinite
  sequence $W = \cdots w_2w_1w_0$ with no initial segment in $I_\lambda$
  and so certainly no initial segment in $I$---which is a
  contradiction.
\end{proof}
The characterisation of $B$ is likewise slightly different. Again, we
can identify its elements with complementary pairs of closed ideals of
$A^\ast$, but the characterisation of such pairs is more delicate. One
should think of them as well-founded $A$-ary trees whose leaves have
been labelled with $0$ or $1$; the addresses of the $0$- and
$1$-labelled leaves of such a tree then constitute the ideals in the
complemented pair. This leads to the following characterisation of the
complemented closed ideals:

\begin{Lemma}
  \label{lem:5}
  A closed ideal $I \leqslant_c A^\ast$ has a complement if, and only if,
  for every infinite word $W \in A^{-\omega}$ there is a finite initial
  segment $w$ of $W$ for which either $A^\ast w \leqslant I$ or
  $A^\ast w \cap I = \emptyset$.
\end{Lemma}
\begin{proof}
  If $I$ has a complement $I'$ then $I+I'$ is dense, whence for any
  $W \in A^{-\omega}$ there is a finite initial segment $w$ with
  $w \in I+I'$. If $w \in I$ then $A^\ast w \leqslant I$; while if
  $w \in I'$ then $A^\ast w \cap I = \emptyset$. Suppose conversely that
  $I$ satisfies the stated condition; then we define
  $I' = \{w \in A^\ast \mid A^\ast w \cap I = \emptyset\}$. It is easy to
  see that $I'$ is a closed ideal which is obviously disjoint from
  $I$. Moreover, $I+I'$ is dense: for if $W$ is any infinite word,
  then there is an initial segment $w$ for which either
  $A^\ast w \leqslant I$, whence $w \in I \leqslant I+I'$, or
  $A^\ast w \cap I = \emptyset$, whence $w \in I' \leqslant I+I'$.
\end{proof}
Now $B$ is the Boolean algebra of these complemented closed ideals,
and the actions of $B$ and $M$ on each other are much as before. The
extra ingredient is the zero-dimensional topology on $B$; and it is
not hard to see that a disjoint family of complemented closed ideals
$(I_x : x \in X)$ is in $\J$ just when every infinite word
$W \in A^{-\omega}$ has an initial segment in (exactly) one of the
$I_x$'s.

The motivating topological perspective also generalises to the
infinitary case. This may come as a surprise: after all, according to
what we said earlier, in the Grothendieck case we should only expect a
\emph{localic} perspective. However, in this example there are enough
$\J$-closed ideals to separate elements of $B$ (this is essentially
the force of the last two lemmas), so that $\BJ$ can be identified
with the Grothendieck Boolean algebra of clopen sets of the space of
$\J$-prime filters on $B$---which is the (non-compact) prodiscrete
space $A^{-\omega}$. With this identification made, we may now view
$M$ as the monoid of continuous functions
$A^{-\omega} \rightarrow A^{-\omega}$ of the form~\eqref{eq:60}, but
now for a possibly \emph{infinite} family of pairs $(u_i, v_i)$.

It follows from the above that the classifying localic category of
$\BJM$ is in fact spatial and, like before, a groupoid; it is the
obvious generalisation of $\mathfrak{O}_2$, with space of objects $A^{-\omega}$
and morphisms defined just as before. On the other hand, the
associated Grothendieck Boolean restriction monoid $S_\J$ is generated
by elements $a$ and $a^\ast$ for each $a \in A$, subject to the axioms
\begin{equation}
  \label{eq:63}
  aa^\ast = 1 \text{ for all } a \in A\text{,} \quad ab^\ast = 0 \text{ for all } a \neq b \in A \quad \text{and} \quad 
  \textstyle\bigvee_{a \in A} a^\ast a = 1\rlap{ ,}
\end{equation}
and, much as before, elements of $S_\J$ correspond to the
\emph{partial} continuous maps
$A^{-\omega} \rightharpoonup A^{-\omega}$ of the form~\eqref{eq:60}.

\section{Nekrashevych toposes}
\label{sec:nekrashevych-toposes}
Our next example draws on the material
of~\cite{Nekrashevych2005Self-similar, Nekrashevych2009Cstar}; the
idea is to extend the monoids $M$ studied in the previous two sections
to monoids of endomorphisms $\varphi \colon A^{-\omega} \rightarrow
A^{-\omega}$ which can be written in the form
\begin{equation}
  \label{eq:62}
  \varphi(Wu_i) = W'v_i \qquad \text{with} \qquad W' = p_i(W)\rlap{ ,}
\end{equation}
where each $p_i$ lies in a monoid of ``well-behaved'' endomorphisms of
$A^{-\omega}$.

\begin{Defn}[Self-similar monoid]
  \label{def:39}
  Let $P$ be a monoid of continuous functions
  $A^{-\omega} \rightarrow A^{-\omega}$. We say that $P$ is
  \emph{self-similar} if for every $p \in P$ and $a \in A$ there
  exists $b \in A$ and $q \in P$ such that $p(Wa) = q(W)b$ for all
  $W \in A^{-\omega}$.
\end{Defn}
In~\cite{Nekrashevych2005Self-similar, Nekrashevych2009Cstar}, the
``well-behaved'' endomorphisms are always invertible, whereupon we
speak of self-similar \emph{groups}; but the invertibility has no
bearing on constructing a cartesian closed variety, and so we develop
the more general case here.

If we name the $b$ and $q$ in the above definition as $p \star a$ and
$\res p a$, then we can finitistically encode the action of elements
of $P$ on infinite words via what a computer scientist would call a
\emph{Mealy machine}, an algebraist would call a \emph{matched pair of
  monoids}~\cite{Majid1995Foundations}, and a category theorist would
call a \emph{distributive law}~\cite{Beck1969Distributive}:

\begin{Defn}[Self-similar monoid action]
  \label{def:38}
  Let $P$ be a monoid. A \emph{self-similar action} of
  $P$ on a set $A$ is a function
  \begin{align*}
    \delta \colon A \times P &\rightarrow P \times A &
    (a,p) &\mapsto (\res p a, a \star p)
    \rlap{ ,}
  \end{align*}
  satisfying the axioms:
  \begin{itemize}
  \item $a \star 1 = a$ and $a \star (pq) = (a \star p) \star q$
    (i.e., $\star$ is a monoid action on $A$); and
  \item $\res 1 a = 1$ and $\res{(pq)}a = \res p {a} \res q {a \star p}$.
  \end{itemize}
  A self-similar action of $P$ on $A$ induces one on $A^\ast$, where:
  \begin{equation}
    \label{eq:61}
    \begin{aligned}
      \res p {a_n \cdots a_1} &= \res{\smash{(\cdots (\res{\smash{(\res
              p {a_1})}}{a_{2}}) \cdots)}}{a_n} \\ \text{and } (a_n \cdots a_1) \star p &= 
       (a_n \star \res{p}{a_{n-1} \cdots a_1}) \cdots (a_3 \star \res{p} {a_2a_1})(a_2 \star \res {p} {a_1})(a_1 \star p)\rlap{ ;}
    \end{aligned}
  \end{equation}
  and we say $\delta$ is a \emph{faithful} self-similar action if
  the action $\star$ of $P$ on $A^\ast$ is faithful.
\end{Defn}
If 
$\delta \colon A \times P \rightarrow P \times A$ is a self-similar
action, then the action of $P$ on $A^\ast$ determines a continuous
action of $P$ on $A^{-\omega}$, given by:
\begin{equation*}
  p(\cdots a_3 a_2 a_1) \qquad = \qquad \cdots( a_3 \star \res{p} {a_2a_1})(a_2 \star \res {p} {a_1})(a_1 \star p)
       \rlap{ ;}
\end{equation*}
and if $\delta$ is a faithful self-similar action, then this action on
$A^{-\omega}$ is again faithful, so that we can identify $P$ with a
self-similar monoid of continuous endomorphisms
$A^{-\omega} \rightarrow A^{-\omega}$. Thus, self-similar submonoids of
$\mathrm{End}(A^{-\omega}, A^{-\omega})$ amount to the same thing as
faithful self-similar monoid actions on $A$.

We now construct a cartesian closed variety from any self-similar
monoid action.

\begin{Defn}[Nekrashevych algebras]
  \label{def:40}
  Given a self-similar action of a monoid $P$ on $A$ and a left
  $P$-set $X$, we define a left $P$-set structure on $X^A$ via
  $(p \cdot \varphi)(a) = \res p a \cdot \varphi(a \star p)$. A
  \emph{Nekrashevych $\delta$-algebra} is a left $P$-set $X$ endowed
  with an $P$-set isomorphism $X \cong X^A$. We write $\N_\delta$ for
  the variety of Nekrashevych $\delta$-algebras.
\end{Defn}
Like before, $\N_\delta$ is cartesian closed by virtue of being a
topos of sheaves on a monoid. The monoid in question we write as
$P \bowtie_\delta A^\ast$, the \emph{Zappa-Sz\'ep product} of $P$ and
$A^\ast$ over $\delta$; its underlying set is $P \times A^\ast$, its
unit element is $(1, \epsilon)$, and its multiplication is given using
the self-similar action~\eqref{eq:61} of $P$ on $A^\ast$ by
$(p,u) (q,v) = (p (\res q u), (u \star p)v)$. (In fact, the monoids
arising in this way from self-similar group actions have an abstract
characterisation due to Perrot; see~\cite{Lawson2008Correspondence}
for the details.)

$P \bowtie_\delta A^\ast$ has an obvious presentation: the generators
are $(1,a)$ for $a \in A$ together with $(p,\epsilon)$ for $p \in P$,
and the axioms are $1 = (1,\epsilon)$,
$(p,\epsilon)(q,\epsilon) = (pq,\epsilon)$ and
$(1,a)(p,\epsilon) = (\res p a, \epsilon)(1,a \star p)$. Thus, a left
$P \bowtie_\delta A^\ast$-set structure on $X$ is the same thing as a
left $P$-set structure and a left $A^\ast$-set structure such that
$a \cdot (p \cdot x) = \res p a \cdot ((a \star p) \cdot x)$ for all
$x \in X$, $p \in P$ and $a \in A$; but this is precisely to say that
the family of maps $a \cdot (\thg) \colon X \rightarrow X$ assemble to
give a left $P$-set map $X \rightarrow X^A$, where $X^A$ is given the
$P$-set structure from Definition~\ref{def:40}. It follows as
in~\cite[Example~1.3]{Johnstone1985When} that $\N_\delta$ can be
presented as the topos of sheaves on $P \bowtie_\delta A^\ast$
for the topology generated by the covering family
$\{(1,a) : a \in A\}$.

We can now follow through the argument of the preceding sections to
obtain a presentation of the matched pair $\BJM$ for which
$\N_\delta \cong \BJM\text-\cat{Set}$. A subtle point that requires
some additional work is the following:
\begin{Prop}
  \label{prop:48}
  Let $\delta \colon A \times P \rightarrow P \times A$ be a
  self-similar action of $P$ on $A$. If $\delta$ is a faithful action,
  then $P \bowtie_\delta A^\ast$ is separated as a left
  $P \bowtie_\delta A^\ast$-set.
\end{Prop}
\begin{proof}
  Let $(p,u), (q,v) \in M_0$ and suppose that
  $(1,a) \cdot (p,u) = (1,a) \cdot (q,v)$ for all $a \in A$; we must
  show that $(p,u) = (q,v)$. The hypothesis says that
  $(\res p a, (a \star p)u) = (\res q a, (a \star q)v)$ for all
  $a \in A$; clearly, then, $u = v$. On the other hand, we have
  $a \star p = a \star q$ and $\res p a = \res q a$ for all $a \in A$,
  which implies that $p$ and $q$ have the same actions on $A^\ast$. By
  fidelity of $\delta$ we conclude that $p = q$ as desired.
\end{proof}
So when $\delta$ is faithful, we can describe $M$ like before as the
monoid of maximal dense partial $P
  \bowtie_\delta A^\ast$-set maps $P
  \bowtie_\delta A^\ast \rightarrow P
  \bowtie_\delta A^\ast$.
Here, although the ideal structure of $P \bowtie_\delta A^\ast$
is now more complex, the \emph{dense} ideals are no harder; they are
exactly the ideals of the form $P \times I$ where $I \leqslant_d A^\ast$.
Likewise, the complemented closed ideals of $M_0$ are those of the
form $P \times I$ for $I$ a complemented closed ideal of $A^\ast$;
and so we find that:
\begin{itemize}
\item $M$ is the monoid of all maximal partial maps
  $(P \times I, f) \colon P \bowtie_\delta A^\ast \rightarrow P
  \bowtie_\delta A^\ast$ where $I \leqslant_d A^\ast$, under the
  monoid operation given by partial map composition followed by
  maximal extension;
\item $B_\J$ is the Grothendieck Boolean algebra of 
  complemented closed ideals of $A^\ast$;
\item $M$ and $B_\J$ act on each other like before, after identifying
  each complemented closed ideal $I \leqslant A^\ast$ with the
  corresponding ideal $P \times I \leqslant P  \bowtie_\delta A^\ast$.
\end{itemize}
Since $\BJ$ is the same Grothendieck Boolean algebra as before, the
topological perspective on these data again involves seeing $M$ as a
monoid of continuous endomorphisms of the space $A^{-\omega}$. This
time, given
$(P \times I, f) \colon P \bowtie_\delta A^\ast \rightarrow P
\bowtie_\delta A^\ast$ in $M$ with $I = \spn{u_i}$, the elements
$\{u_i\}$ and $(p_i, v_i) = f(1, u_i)$ provide the data as
in~\eqref{eq:62} for the corresponding continuous endomorphism of
$A^{-\omega}$; note that fidelity of $\delta$ ensures that
distinct elements of $M$ encode distinct endomorphisms of
$A^{-\omega}$. It follows that the classifying topological category of
$\N_\delta$ has space of objects $A^{-\omega}$, and as morphisms
$W \rightarrow W'$, germs at $W$ of functions~\eqref{eq:62} with
$\varphi(W) = W'$. When $P$ is a group and $A$ is finite, this is
exactly the topological category~$\mathfrak{O}_G$ described
in~\cite[\sec 5.2]{Nekrashevych2009Cstar}.

Finally, let us consider the associated Grothendieck Boolean
restriction monoid $S_\J$ of $\BJM$; this is generated by elements
$a,a^\ast$ as in~\eqref{eq:63} but now augmented by total elements $p$
for each $p \in P$, which multiply as in $p$, and additionally satisfy
$ap = (\res p a)(a \star p)$. From this and
$p = \bigvee_{a \in A} a^\ast a p$, we deduce the left equality in:
\begin{equation}
  \label{eq:33}
  p = \bigvee_{a \in A} a^\ast \res  p a (a \star p) \qquad \qquad  pb^\ast = \bigvee_{a \star p = b} a^\ast \res p a
\end{equation}
which on multiplying by $b^\ast$ yields the equality to the right.
Using this, we can rewrite any element of $S_\J$ in the form
$\bigvee_i u_i^\ast p_i v_i$ where $\{u_i\}$ is the basis of a
complemented ideal; and much as before, each such element represents a
partial function $A^{-\omega} \rightarrow A^{-\omega}$ via the
formula~\eqref{eq:62}.

Now, because we are considering self-similar \emph{monoid} actions,
rather than \emph{group} actions, it need not be the case that the
cartesian closed variety $\N_\delta$ is groupoidal. As we would
hope, this is certainly the case when we do start from a group, but
prima facie there could be further examples beyond this. Part (b) of
the following result appears to indicate that this is so; however,
part (c) shows that this apparent extra generality is in fact
spurious: a theory of Nekrashevych algebras is groupoidal just
when it is the theory of $\delta$-algebras for some self-similar group
action.

\begin{Prop}
  \label{prop:49}
  For a faithful self-similar action $\delta$, the
  following are equivalent:
  \begin{enumerate}[(a)]
  \item The theory of Nekrashevych $\delta$-algebras is
    groupoidal;
  \item For each $p \in P$ there is a dense ideal $I \leqslant
    A^\ast$ with $\res p w$ invertible for all $w \in I$;
  \item The forgetful functor $\N_\delta \rightarrow \N_{\delta'}$ is
    an isomorphism, where
    $\delta' \colon A \times G \rightarrow G \times A$ is the
    restriction of $\delta$ to the group $G$ of invertible elements of
    $P$.
  \end{enumerate}
\end{Prop}
Note the restriction in (c) is well-posed, since if $p \in P$ is
invertible, then each $\res p a$ is also invertible with inverse
$\res{\smash{p^{-1}}}{a \star p}$.

\begin{proof}
  We first show (b) $\Rightarrow$ (a). The theory of Nekrashevych
  $\delta$-algebras will be groupoidal just when the associated $S_\J$
  is \'etale; since each generator $a,a^\ast$ is already a partial
  isomorphism, this will be the case just when each $p \in S_\J$ is an
  admissible join of partial isomorphisms. So assuming (b), we have
  for each $p$ a dense ideal $I$ with $\res p w$ invertible for all
  $w \in I$. Thus for each $w \in I$, the map $w^\ast w p$ has partial
  inverse $(w \star p)^\ast (\res p w)^{-1} w$, since
  $(w \star p)^\ast (\res p w)^{-1} w w^\ast w p = (w \star p)^\ast
  (\res p w)^{-1} \res p w (w \star p) = (w \star p)^\ast (w \star p)
  = (w \star p)^+$ and
  $w^\ast w p (w \star p)^\ast (\res p w)^{-1} w = w^\ast \res p w (w
  \star p) (w \star p)^\ast (\res p w)^{-1} w = w^\ast \res p w (\res
  p w)^{-1} w = w^\ast w = w^+$. So if $I = \spn{u_i}$ then the
  expression $\bigvee_i u_i^\ast u_i p$ expresses $p$ as an admissible
  join of partial isomorphisms.

  Now, towards proving (a) $\Rightarrow$ (b), let $p \in P$ and
  suppose that for some $w \in A^\ast$, the map $w^\ast wp$ has a
  partial inverse $q$. We can write $q = \bigvee_i u_i^\ast q_i v_i$
  and by using the left equation of~\eqref{eq:33} where necessary we
  can assume each $v_i$ is at least as long as $w$. Now, we calculate
  that
  \begin{equation*}
    q w^\ast wp = \bigvee_i u_i^\ast q_i v_i w^\ast wp =
    \bigvee_{\substack{i \text{ s.t.} \\ v_i \in \spn{w}}} u_i^\ast q_i v_ip =
    \bigvee_{\substack{i \text{ s.t.} \\v_i \in \spn{w}}} u_i^\ast q_i \res p {v_i} (v_i \star p)\rlap{ ;}
  \end{equation*}
  but since this must equal $q^+ = \bigvee_i u_i^\ast u_i$, we must
  have for \emph{all} $i$ that $v_i \in \spn{w}$, that $q_i \res p {v_i}
  = 1$ and that $u_i = v_i \star p$. Now using the right
  equality in~\eqref{eq:33} we have:
  \begin{equation*}
    w^\ast wpq = \bigvee_i w^\ast wp u_i^\ast q_i v_i  = \bigvee_i \bigvee_{a \star p = u_i} w^\ast w a^\ast \res p a q_i v_i\rlap{ .}
  \end{equation*}
  This join must equal $w^\ast w$; but since in particular
  $v_i \star p = u_i$, the join includes the terms
  $w^\ast wv_i^\ast \res p {v_i}q_i v_i = v_i^\ast \res p {v_i}q_i
  v_i$, which must thus be restriction idempotents: and this is only
  possible if $\res p {v_i} q_i = 1$; but since already
  $q_i \res p {v_i} = 1$ we see that $\res p {v_i}$ has inverse $q_i$.
  Now any other $a$ with $a \star p = u_i$ must satisfy
  $\res p a q_i = 1$ and so $\res p a = \res p {v_i}$. Since also
  $a \star p = u_i = v_i \star p$ we have $a = v_i$ by fidelity of the
  action. Thus the join displayed above is equal to
  $\bigvee_i w^\ast w v_i^\ast \res p {v_i} q_i v_i = \bigvee_i w^\ast
  w v_i^\ast v_i = \bigvee_i v_i^\ast v_i$; since it also equals
  $w^\ast w$, the ideal $J_w$ generated by the $v_i$'s must be dense
  in $\spn{w}$.

  Now, suppose as in (a) that every $p \in P$ is a join of partial
  isomorphisms $p = \bigvee_i u_i^\ast u_i p$; then we have ideals
  $J_{u_i} \leqslant_d \spn{u_i}$ for each $i$ such that $\res p v$ is
  invertible for all $v \in J_{u_i}$. So taking $I = \sum_i J_{u_i}$
  we have
  $I = \sum_i J_{u_i} \leqslant_d \sum_i \spn{u_i} \leqslant_d M_0$ and
  $\res p w$ invertible for all $w \in I$, which gives (b).

  Next, for (c) $\Rightarrow$ (a), note that the theory of
  Nekrashevych $\delta'$-algebras trivially satisfies (b), and so is
  groupoidal; whence also the isomorphic theory of Nekrashevych
  $\delta$-algebras. Finally, to prove (b) $\Rightarrow$ (c), it
  suffices to show that the map of Grothendieck Boolean restriction
  monoids $S^{\delta'} \rightarrow S^\delta$ induced by the inclusion
  $G \subseteq P$ is invertible. It is injective since both
  $S^{\delta'}$ and $S^{\delta}$ are submonoids of the monoid of
  partial continuous endofunctions of $A^{-\omega}$; for surjectivity
  we need only show that each $p \in P$ is in its image. But letting
  $I = \spn{u_i}$ be a dense ideal as in (b), and using the left
  equation in~\eqref{eq:33} we can write
  $p = \bigvee_{i} u_i^\ast \res p {u_i} (p \star u_i)$; since each
  $\res p {u_i}$ lies in $G$, this provides the desired expression.
\end{proof}

\section{Cuntz--Krieger toposes}
\label{sec:cuntz-krieg-topos-1}

The Cuntz $C^\ast$-algebra on alphabet $A$ can be generalised to the
\emph{Cuntz--Krieger} $C^\ast$-algebra on a directed graph
$\mathbb{A}$~\cite{Kumjian1997Graphs}; the way in which the former
becomes a special case of the latter is by considering the graph with
a single vertex and $A$ self-loops. Correspondingly, the notion of
Leavitt algebra has a generalisation to the notion of \emph{Leavitt
  path algebra}, and both of these generalisations in fact come from a
generalisation of the Cuntz topological groupoid on $A$ to the ``path
groupoid'' on $\mathbb{A}$. In this final section, we explain how this
generalisation plays out from the perspective of cartesian closed
varieties.

The situation this time is subtly different. We will again describe a
topos which is a variety, but now it will be a \emph{many-sorted}
variety, with one sort for each vertex of $\mathbb{A}$. The
corresponding variety of $\BJM$-sets will not be the topos we started
from, but rather its two-valued collapse in the sense of
Section~\ref{sec:when-variety-topos}; indeed, by virtue of
Proposition~\ref{prop:29}, the topos we started from will instead be
the category of $\BJM$-\emph{sheaves} (Definition~\ref{def:28}). The
missing result we need is the following:

\begin{Prop}
  \label{prop:1}
  Let $\C$ be a many-sorted variety which is also a non-degenerate
  topos, and let $X \in \C$ be the free algebra on one generator of
  each sort. Then $\C_{\mathrm{tv}}$ is equivalent to a single-sorted
  cartesian closed variety $\V$, with $X$ corresponding under this
  equivalence to the free $\V$-algebra on one generator. Thus
  $\C_\mathrm{tv} \simeq \BJM\text-\cat{Set}$ where $\BJM$ is defined
  from $X$ as in Proposition~\ref{prop:44}.
\end{Prop}
\begin{proof}
  Since $\C$ is a non-degenerate topos, its initial object is
  strict, so the theory which presents it as a variety has no
  constants. Hence, by~\cite{Barr1972Point},
  $\C = \E_\mathrm{tv}$ is equivalent to a variety
  when equipped with the functor $\C \rightarrow \cat{Set}$ sending a
  model $(M(s) : s \in S)$ to $\prod_{s \in S}M(s)$. But this functor
  is just $\C(X, \thg)$, and as 
  in~\cite{Johnstone1990Collapsed}, $\C_{\mathrm{tv}}$ is cartesian
  closed since $\C$ is so.
\end{proof}

\subsection{Presheaf toposes}
\label{sec:presheaf-toposes}

Before considering groupoids associated to directed graphs, as a kind
of warm-up exercise we start with a simpler case of
Proposition~\ref{prop:1} wherein $\C$ is a presheaf category.

Given our ongoing conventions, it will be most convenient to look at a
\emph{covariant} presheaf category $[\mathbb{A}, \cat{Set}]$. We call
objects $X \in [\mathbb{A}, \cat{Set}]$ \emph{left
  $\mathbb{A}$-sets}, and present them as a family of sets $X_a$
indexed by the objects of $\mathbb{A}$, together with reindexing
operators $f \cdot (\thg) \colon X_a \rightarrow X_b$ for every
morphism $f \colon a \rightarrow b$ of $\mathbb{A}$, satisfying the
usual associativity and unitality axioms. The cartesian closed variety
$[\mathbb{A}, \cat{Set}]_{\mathrm{tv}}$ to which this collapses is the
variety of left $\mathbb{A}$-sets for which either all $X_a$'s are
empty or all $X_a$'s are non-empty. An explicit theory presenting this
variety was given in~\cite[Example~8.7]{Johnstone1990Collapsed}; our
objective is to present it as a variety of $\BJM$-sets.

Now, $[\mathbb{A}, \cat{Set}]$ is a variety with set of sorts
$\mathrm{ob}(\mathbb{A})$, and the free object on one generator of
each sort is the $\mathbb{A}$-set, which will denote simply by
$\mathbb{A}$, for which $\mathbb{A}_a$ is the set of all morphisms of
$\mathbb{A}$ with codomain $a$, and for which
$f \cdot (\thg) \colon \mathbb{A}_a \rightarrow \mathbb{A}_b$ is given
by postcomposition. Now by Proposition~\ref{prop:1}, the monoid
$M$ and Boolean algebra $B$ can be found as
$[\mathbb{A}, \cat{Set}](\mathbb{A}, \mathbb{A})$ and
$[\mathbb{A}, \cat{Set}](\mathbb{A},1+1)$ respectively.

On the one hand, a map $\mathbb{A} \rightarrow \mathbb{A}$ in
$[\mathbb{A}, \cat{Set}]$ is by freeness determined uniquely by
elements $f_a \in \mathbb{A}_a$ for each $a \in \mathbb{A}$; thus, an
element $f \in M$ comprises a family of objects
$(f^\ast a)_{a \in \mathrm{ob}(\mathbb{A})}$ and a family of arrows
$(f_a \colon f^\ast a \rightarrow a)_{a \in \mathrm{ob}(\mathbb{A})}$
of $\mathbb{A}$. It is now easy to see that the unit of $M$ is
$(1_a \colon a \rightarrow a)_{a \in A}$, while the product of $f$ and
$g$ is characterised by
$(f \cdot g)_a = f_a \circ g_{f^\ast{a}} \colon g^\ast f^\ast a
\rightarrow f^\ast a \rightarrow a$. In the nomenclature
of~\cite[Chapter~I.5]{Aguiar1997Internal}, $M$ is the \emph{monoid of
  admissible sections} of $\mathbb{A}$.

On the other hand, the $\mathbb{A}$-set $1+1$ has
$(1+1)_a = \{\top, \bot\}$ for all objects $a$; whence an
$\mathbb{A}$-set map $\mathbb{A} \rightarrow 1+1$ amounts to a
function $\mathrm{ob}(\mathbb{A}) \rightarrow \{\top, \bot\}$. It
follows easily that $B$ is the power-set Boolean algebra
$\P(\mathrm{ob}(\mathbb{A}))$, and that, in the infinite case, the
zero-dimensional topology $\J$ comprises \emph{all} partitions of
$\P(\mathrm{ob}(\mathbb{A}))$. Similar straightforward calculations
now show that:
\begin{itemize}
\item $f\in M$ acts on
  $U \in B$ to yield 
  $f^\ast(U) = \{a \in \mathrm{ob}(\mathbb{A}) : f^\ast a \in U\} \in B$.
\item $U \in B$ acts on $f,g \in M$ to yield the $U(f,g) \in M$ with
  $U(f,g)_a = f_a$ for $a \in U$ and $U(f,g)_a = g_a$ for $a \notin U$.
\end{itemize}
Now, if $X \in [\mathbb{A}, \cat{Set}]_\mathrm{tv}$ then the set
$\tilde X = [\mathbb{A}, \cat{Set}](\mathbb{A},X) = \prod_{a \in A}
X_a$ becomes a $\BJM$-set as in Proposition~\ref{prop:44};
explicitly, if $x,y \in \tilde{X}$, $f \in M$ and $U \in B$, then:
\begin{itemize}
\item $f \cdot x \in \tilde{X}$ is given by $(f \cdot x)_a =
f_a \cdot x_{f^\ast a}$;
\item $U(x,y) \in \tilde{X}$ is given by $U(x,y)_a = x_a$
  for $a \in U$ and $U(x,y)_a = y_a$ for $a \notin U$.
\end{itemize}

We are once again in the situation where there are enough $\J$-closed
ideals in $\BJ$ to separate elements, so that there is a topological,
rather than localic, perspective on $\BJM$. Indeed, $\BJ$ is the
Grothendieck Boolean algebra of clopen sets of the discrete space
$\mathrm{ob}(\mathbb{A})$, and under this correspondence, the action
of $f \in M$ on $B$ is given by inverse image under the function
$a \mapsto f^\ast a$. It follows from this that the classifying localic
category of $\BJM$ is again spatial, and is simply the \emph{discrete}
topological category $\mathbb{A}$. Of course, this topological
category is a groupoid just when $\mathbb{A}$ is a groupoid, and so
this characterises when the cartesian closed variety
$[\mathbb{A}, \cat{Set}]_\mathrm{tv}$ is groupoidal. On the other
hand, $\mathbb{A}$ is minimal, so that
$[\mathbb{A}, \cat{Set}]_\mathrm{tv} = [\mathbb{A}, \cat{Set}]$ is a
topos, just when every object of $\mathbb{A}$ admits an arrow to every
other object of $\mathbb{A}$; which is to say that $\mathbb{A}$ is
\emph{strongly connected} in the sense
of~\cite[Example~8.7]{Johnstone1990Collapsed}.

\subsection{Cuntz--Krieger toposes}
\label{sec:cuntz-krieg-topos}
We now describe the cartesian closed varieties which correspond to
Cuntz--Krieger $C^\ast$-algebras associated to directed graphs. As
explained, these varieties will be obtained from many-sorted varieties
which are (Grothendieck) toposes. These toposes were were introduced
by Leinster~\cite{Leinster2007Jonsson}, with the connection to
operator algebra being made explicit in~\cite[\sec
5]{Henry2016Convolution}.

\begin{Defn}
  \label{def:41}
  Let $\mathbb{A}$ be a directed graph, that is, a pair of sets
  $A_1, A_0$ together with source and target functions
  $s, t \colon A_1 \rightrightarrows A_0$. As usual, we write
  $e \colon v \rightarrow v'$ to indicate that $e \in A_1$ with
  $s(e) = v$ and $t(e) = v'$, and we will also make use of the sets
  $s^{-1}(v)$ of all edges in $\mathbb{A}$ with a given fixed source
  $v$. Now a \emph{Cuntz--Krieger $\mathbb{A}$-algebra} is a family of
  sets $(X_v : v \in A_0)$ together with, for each $v \in A_0$, a
  specified isomorphism between $X_v$ and the set
  \begin{equation*}
    \prod_{e \in s^{-1}(v)} X_{t(e)} = \prod_{e \colon v \rightarrow v'} X_{v'}\rlap{ .}
  \end{equation*}
  We write $\C\K_{\mathbb{A}}$ for the many-sorted variety of
  Cuntz--Krieger $\mathbb{A}$-algebras.
\end{Defn}

As shown in~\cite{Leinster2007Jonsson, Henry2016Convolution},
$\C\K_\mathbb{A}$ is a topos. To see this, we first define
$\mathbb{A}^\ast$ to be the \emph{free category} on the graph
$\mathbb{A}$, whose objects are vertices of $\mathbb{A}$, and whose
morphisms $v \rightarrow w$ are finite paths of edges from $v$ to $w$,
i.e.:
\begin{equation*}
  \mathbb{A}(v,w) = \{\,e_n \cdots e_1 \mid s(e_1) = v, t(e_{i}) = s(e_{i+1}), t(e_n) = w\,\}\rlap{ ,}
\end{equation*}
where by convention $\mathbb{A}(v,v)$ also contains the empty path
$\epsilon_v$ from $v$ to $v$. Now a left $\mathbb{A}^\ast$-set $X$ is
the same as a family of sets $(X_v : v \in A_0)$ together with
functions $e \cdot (\thg) \colon X_v \rightarrow X_{v'}$ for each edge
$e \colon v \rightarrow v'$ of $\mathbb{A}$. We can endow
$\mathbb{A}^\ast$ with a topology by requiring that, for each object
$v$, the family $(e \colon v \rightarrow v' \mid e \in s^{-1}(v))$ is a
cover of $v$ (note that, since we are taking covariant presheaves, a
covering family is a family of morphisms with common \emph{domain},
rather than common \emph{codomain}). Now as explained
in~\cite{Henry2016Convolution}, a left $\mathbb{A}^\ast$-set $X$ will
satisfy the sheaf condition for this topology just when, for each
vertex $v$, the map
$X_v \rightarrow \prod_{e \in s^{-1}(v)} X_{t(e)}$ induced by
the functions $e \cdot (\thg) \colon X_v \rightarrow X_{v'}$ is an
isomorphism. Thus $\C\K_{\mathbb{A}} \simeq \cat{Sh}(\mathbb{A}^\ast)$
as claimed.

In the single-sorted case, we described $\J\T_A$ in terms of a
localisation of the category of \emph{separated} left $A^\ast$-sets.
We can proceed in exactly the same way here. Unfolding the
definitions yields:
\begin{Defn}
  \label{def:42}
  Given a left $\mathbb{A}^\ast$-set $X$ and a
  sub-left-$\mathbb{A}^\ast$-set $Y \leqslant X$:
  \begin{itemize}
  \item $X$ is \emph{separated} if $x,y \in X_v$ are equal whenever
    $e \cdot x = e \cdot y$ for all $e \in s^{-1}(v)$.
  \item $Y \leqslant X$ is \emph{closed} if any $x \in
    X_v$ with $e \cdot x \in Y_{t(e)}$ for all $e \in s^{-1}(v)$ is in
    $Y_v$.
  \item $Y \leqslant X$ is \emph{dense} if the closure of $Y$ in $X$ is $X$.
  \end{itemize}
\end{Defn}

With these definitions in place, we can now identify the
Cuntz--Krieger topos $\C\K_\mathbb{A}$, just like before, with the
category $\C\K'_\mathbb{A}$ of maximal dense partial maps between
separated left $\mathbb{A}^\ast$-sets, with composition given by
partial map composition followed by maximal extension. We now use this
to describe the matched pair $\BJM$ which presents the cartesian
closed variety $(\C\K_{\mathbb{A}})_{\mathrm{tv}}$.

First, as we saw in the preceding section, the free left
$\mathbb{A}^\ast$-set on one generator of each sort is
$\mathbb{A}^\ast$ acting on itself by composition: thus,
$(\mathbb{A}^\ast)_v$ is the set of all finite $\mathbb{A}$-paths
$e_n e_{n-1} \cdots e_1$ ending at the vertex $v$, and the function
$(\mathbb{A}^\ast)_v \rightarrow (\mathbb{A}^\ast)_{v'}$ induced by an
edge $e \colon v \rightarrow v'$ simply appends $e$ to the end of the
path: $e \cdot (e_n \cdots e_1) = ee_n \cdots e_1$. Clearly
$\mathbb{A}^\ast$ is separated as an $\mathbb{A}^\ast$-set, and so the
monoid $M$ is equally well the monoid
$\C\K'(\mathbb{A}^\ast,\mathbb{A}^\ast)$ of all maximal dense partial
left $\mathbb{A}^\ast$-set maps
$\mathbb{A}^\ast \rightarrow \mathbb{A}^\ast$.
Now, a sub-$\mathbb{A}^\ast$-set $I \leqslant \mathbb{A}^\ast$ is an
\emph{ideal} of $\mathbb{A}^\ast$: that is, a collection
$I \subseteq \mathrm{mor}(\mathbb{A}^\ast)$ of morphisms of
$\mathbb{A}^\ast$ which is closed under postcomposition, and as
before, we can be more explicit about the \emph{dense} ideals.
Intuitively, these are given by a family $(\tau_a : a \in A_0)$ of
well-founded trees, where:
\begin{itemize}
\item Each vertex of each tree is labelled by a vertex of
  $\mathbb{A}$;
\item The child edges of a $v$-labelled vertex are labelled
  bijectively by edges $e \in s^{-1}(v)$, with the far end of the
  $e$-labelled edge being a $t(e)$-labelled vertex; and
\item The root of each $\tau_a$ is labelled by $a$.
\end{itemize}
Such a family of trees can, as before, be specified by listing the
addresses of its leaves, where the ``address'' of a leaf is now the
path of edges to the leaf from the root. These addresses generate an
ideal of $\mathbb{A}^\ast$, and well-foundedness assures that the
ideals so arising should be the dense ones. Said algebraically,
this becomes the following generalisation of Lemma~\ref{lem:4}; the
proof is, \emph{mutatis mutandis}, the same.
\begin{Lemma}
  \label{lem:8}
  An ideal $I \leqslant \mathbb{A}^\ast$ is dense if, and only if,
  each infinite path of edges $\cdots e_3 e_2 e_1$ has a finite initial
  segment $e_n \cdots e_1$ in $I$.
\end{Lemma}

Similarly, we can characterise the Boolean algebra
$B = \C\K'(\mathbb{A}^\ast, 1+1)$ as comprising all complemented
closed ideals of $\mathbb{A}^\ast$, for which we have the following
recognition result generalising Lemma~\ref{lem:5}. Here, we write in
the obvious manner $\mathbb{A}^\ast w$ for the ideal generated by a
finite path $w$.
\begin{Lemma}
  \label{lem:9}
  A closed ideal $I \leqslant_c \mathbb{A}^\ast$ has a complement if,
  and only if, for every infinite path of edges $\cdots e_3 e_2 e_1$
  of $\mathbb{A}$ there is a finite initial segment $w = e_n \cdots e_1$
  of $W$ for which either $\mathbb{A}^\ast w \leqslant I$ or
  $\mathbb{A}^\ast w \cap I = \emptyset$.
\end{Lemma}

With these results in place, the description of the zero-dimensional
topology on $B$ and the actions of $M$ and $B$ on each other goes
through \emph{mutatis mutandis} as before. Once again, there are
enough $\J$-closed ideals to separate elements of $B$, and so there is
a legitimate topological perspective on these data. Indeed, $B_\J$ in
this case is the Grothendieck Boolean algebra of clopen sets of the
\emph{infinite path space} $\mathbb{A}^{-\omega}$, whose elements are
infinite paths $\cdots e_2 e_1 e_0$ in $\mathbb{A}$ starting at any
vertex of $\mathbb{A}$, and whose topology is generated by the basic
clopen sets $[e_n \cdots e_1]$ of all paths which have
$e_n \cdots e_1$ as an initial segment.

We can now use this to describe the continuous map
$\varphi \colon \mathbb{A}^{-\omega} \rightarrow \mathbb{A}^{-\omega}$
induced by a maximal dense partial map
$(I,f) \colon \mathbb{A}^\ast \rightarrow \mathbb{A}^\ast$. First, we
can like before find a basis $\{p_i\}$ of minimal-length paths for the
dense ideal $I$. Suppose that each $p_i$ is a path from $u_i$ to
$v_i$; then $q_i = f(p_i)$ is some \emph{other} path with target $v_i$
and source, say, $w_i$. One way to visualise this is in terms of the
family of well-founded trees $(\tau_a : a \in A_0)$ associated to the
dense ideal $I$; the maximal-length directed paths from the root are
labelled by the basis elements $p_i$, and we can imagine the
$v_i$-labelled leaf at the end of each of these paths as having the
path $q_i$, which also ends at $v_i$, attached to it. Now the set of
pairs of paths $\{(p_i, q_i)\}$ completely specify $(I,f)$'s action on
infinite paths as being the function
$\varphi \colon \mathbb{A}^{-\omega} \rightarrow \mathbb{A}^{-\omega}$
given by:
\begin{equation}
  \label{eq:59}
  \varphi(W'p_i) = W'q_i \qquad \text{for all $W' \in \mathbb{A}^{-\omega}$ starting at $t(p_i)$.}
\end{equation}

From this description, it follows that the classifying topological
category of $(\C\K_\mathbb{A})_{\mathrm{tv}}$ is the category whose
space of objects is $\mathbb{A}^{-\omega}$, and whose morphisms
$W \rightarrow W'$ are germs at $W$ of continuous functions of the
form~\eqref{eq:59} with $\varphi(W) = W'$. It is not hard to identify
such germs with integers $i$ such that $W_n = W'_{i+n}$ for
sufficiently large $n$, so that the classifying topological category is
the well-known \emph{path groupoid} $P(\mathbb{A})$ of
$\mathbb{A}$~\cite{Kumjian1997Graphs}.

Of course, we conclude from this that the theory of Cuntz--Krieger
$\mathbb{A}$-algebras is groupoidal. On the other hand, it is not
necessarily the case that $(\C\K_\mathbb{A})_\mathrm{tv}$ is a topos.
This will be so just when, in fact,
$(\C\K_\mathbb{A})_\mathrm{tv} = \C\K_\mathbb{A}$, or equivalently,
just when the path groupoid is minimal, the condition for which is
well known in the literature. We sketch another proof of this fact
which exploits our ideal-theoretic perspective.

\begin{Defn}
  \label{def:43}
  Let $\mathbb{A}$ be a directed graph. A vertex $v$ of $\mathbb{A}$
  is \emph{cofinal} if for any infinite path
  $\cdots v_2 \xleftarrow{e_2} v_1 \xleftarrow{e_1} v_1
  \xleftarrow{e_0} v_0$ in $\mathbb{A}$ there is some $k$ for which
  there exists a finite path from $v$ to $v_k$.
\end{Defn}

\begin{Prop}
  \label{prop:50}
  For any directed graph $\mathbb{A}$, the following are equivalent:
  \begin{enumerate}[(a)]
  \item The cartesian closed variety $(\C\K_\mathbb{A})_\mathrm{tv}$
    is a topos (and thus equal to $\C\K_\mathbb{A}$);
  \item Every vertex of $\mathbb{A}$ is cofinal.
  \end{enumerate}
\end{Prop}
\begin{proof}
  We first prove (a) $\Rightarrow$ (b). Given a vertex $v$ of
  $\mathbb{A}$, consider $b \in B$ given by the closed complemented
  ideal $\mathbb{A}^\ast v \leqslant \mathbb{A}^\ast$ of all paths
  starting at the vertex $v$. Since (a) holds, by Theorem~\ref{thm:8}
  there must exist $m \in M$ with $m^\ast b = 1$, i.e., there is a
  maximal dense partial map
  $(I,f) \colon \mathbb{A}^\ast \rightarrow \mathbb{A}^\ast$ with
  $f^{-1}(\mathbb{A}^\ast v)$ dense in $\mathbb{A}^\ast$. Thus, for
  any infinite path
  $\cdots v_2 \xleftarrow{e_2} v_1 \xleftarrow{e_1} v_1
  \xleftarrow{e_0} v_0$ there is some $k$ for which
  $e_k \cdots e_0 \in f^{-1}(\mathbb{A}^\ast v)$. But this says that
  $f(e_k \cdots e_0)$ is a path starting at $v$ and ending, like
  $e_k \cdots e_0$, at $v_k$, which shows that $v$ is cofinal in
  $\mathbb{A}$.

  Conversely, suppose every vertex is cofinal in $\mathbb{A}$, and let
  $b \neq 0 \in B$; we must find some $m \in M$ with $m^\ast b = 1$.
  Now $b$ is a non-empty closed ideal $I \leqslant_c \mathbb{A}^\ast$;
  so let $p$ be any path in it and let $u = t(p)$. Consider the set
  \begin{equation*}
    J = \{q \in \mathbb{A}^\ast(v,w) \mid \mathbb{A}^\ast(u,w) \text{ is non-empty}\} \subseteq \mathrm{mor}(\mathbb{A}^\ast)\rlap{ .}
  \end{equation*}
  This is clearly an ideal, and because $u$ is cofinal it is dense in
  $\mathbb{A}^\ast$. Letting $\{q_i\} \subseteq J$ be the basis of
  minimal paths, we can now define an $\mathbb{A}^\ast$-set map
  $f \colon J \rightarrow \mathbb{A}^\ast$ by taking $f(q_i) = r_i
  \cdot p$, where $r_i$ is any path in $\mathbb{A}^\ast(u,t(q_i))$. If
  we let $m = (J,f) \in M$, then $m^\ast(b) = f^{-1}(I)$ contains
  $f^{-1}(\mathbb{A}^\ast p)$, which is clearly all of the dense ideal
  $J \leqslant \mathbb{A}^\ast$; whence $m^\ast b = 1$ as desired.
\end{proof}

\end{document}